\documentclass[11pt,a4paper,leqno]{amsart}
\usepackage{anysize}
\marginsize{3.8cm}{3.8cm}{2cm}{3cm}
\usepackage{amsmath}
\usepackage{amsfonts,amssymb}
\usepackage{enumerate}
\usepackage{amsthm}
\usepackage{mathrsfs}
\usepackage{graphicx}

\DeclareSymbolFont{pletters}{OT1}{cmr}{m}{sl}
\DeclareMathSymbol{s}{\mathalpha}{pletters}{`s}

\theoremstyle{plain}
\newtheorem{theo}{Theorem}[section]
\newtheorem{prop}[theo]{Proposition}
\newtheorem{lemm}[theo]{Lemma}
\newtheorem{coro}[theo]{Corollary}
\newtheorem{defi}[theo]{Definition}

\theoremstyle{definition}
\newtheorem{rema}[theo]{Remark}
\newtheorem{nota}[theo]{Notation}
\DeclareMathOperator{\cnx}{div}
\DeclareMathOperator{\RE}{Re}
\DeclareMathOperator{\dist}{dist}
\DeclareMathOperator{\IM}{Im}
\DeclareMathOperator{\sgn}{sgn}
\DeclareMathOperator{\supp}{supp}

\def\Cr{\mathscr{C}}

\def\CZ#1{C_{z}^{0}\big(H_{x}^{{#1}}\big)}
\def\defn{\mathrel{:=}}
\def\Deltax{\Delta}
\def\Deltayx{\Delta_{x,y}}
\def\eps{\varepsilon}
\def\F{\tilde{F}}
\def\la{\left\lvert}
\def\lA{\left\lVert}
\def\le{\leq}
\def\les{\lesssim}
\def\leo{}
\def\L#1{\langle #1 \rangle}
\def\mez{\frac{1}{2}}
\def\partialx{\nabla}
\def\partialyx{\nabla_{x,y}}
\def\ra{\right\rvert}
\def\rA{\right\rVert}
\def\tdm{\frac{3}{2}}
\def\tq{\frac{3}{4}}

\def\U{\tilde{U}}
\def\vareta{\tilde{\eta}}
\def\vPhi{\tilde{\Phi}}
\def\varpsi{\tilde{\psi}}
\def\vvarphi{\tilde{\varphi}}

\def\xN{\mathbf{N}}
\def\xR{\mathbf{R}}

\numberwithin{equation}{section}
\begin{document}


\title{On the water waves equations with surface tension}
\author{T. Alazard}\address{T. Alazard \\ Univ Paris Sud-11 \& CNRS \\ Laboratoire de Math\'ematiques\\91405 Orsay cedex.}
\email{thomas.alazard@math.u-psud.fr}
\author{N. Burq}\address{N. Burq \\ Univ Paris Sud-11  \\ Laboratoire de Math\'ematiques\\91405 Orsay cedex\& Institut universitaire de France.}
\email{nicolas.burq@math.u-psud.fr}
\author{C. Zuily}\address{C. Zuily \\ Univ Paris Sud-11 \\ Laboratoire de Math\'ematiques\\91405 Orsay cedex.}
\email{claude.zuily@math.u-psud.fr}

\thanks{Support by the french Agence Nationale de la Recherche, project EDP Dispersives, r\'ef\'erence ANR-07-BLAN-0250, is acknowledged.}


\begin{abstract}The purpose of this article is to clarify the Cauchy theory of the water waves equations
as well in terms of regularity indexes for the initial conditions as for the smoothness of the 
bottom of the domain (namely no regularity assumption is assumed on the bottom). 
Our main result is that, following the approach developped in~\cite{AM}, after suitable paralinearizations,
 the system can be arranged into an explicit symmetric system of Schr\"odinger type. We then show that 
the smoothing effect for the (one dimensional) surface tension water waves proved in~\cite{CHS}, is in fact a rather direct consequence of this reduction, which allows also to lower the regularity 
indexes of the initial data, and to obtain the natural weights in the estimates.
\end{abstract}

\maketitle

\tableofcontents

\section{Introduction}

We consider a solution of the incompressible Euler equations for a potential flow
in a domain with free boundary, of the form
$$
\{\,(t,x,y)\in [0,T]\times\xR^d\times\xR \, : \, (x,y) \in \Omega_t\,\},
$$
where $\Omega_t$ is the domain located between a 
free surface
$$
\Sigma_t= \{\, (x,y)\in \xR^d\times \xR\, : \, y= \eta(t,x)\,\},
$$
and a given bottom denoted by $\Gamma= \partial \Omega_t \setminus \Sigma_t$. 
The only assumption we shall make on the domain is that the top boundary, 
$\Sigma_t$, and the bottom boundary, $\Gamma$ are separated by a "strip" of fixed length.

More precisely, we assume that the initial domain satisfy the following assumption for $t=0$.
\begin{itemize}
\item[\bf{$\mathbf{H_t}$)}] The domain $\Omega_{t}$ is the intersection of the half space, 
denoted by $\Omega_{1,t}$, located below the free surface $\Sigma_t$,
$$ 
\Omega_{1,t}= \{  (x,y)\in \xR^d\times \xR\, : \, y< \eta(t,x)\}
$$ 
and an open set $\Omega_{2}\subset \xR^ {d+1}$ such that  $\Omega_{2}$ contains 
a fixed strip around $\Sigma_{t}$, which means that there exists $h>0$ such that, 
$$
\{ (x,y)\in \xR^d\times \xR\, : \, \eta(t,x) -h \le y \leq \eta(t,x) \} \subset \Omega_2.
$$
We shall also assume that the domain $\Omega_2$ (and hence the domain $\Omega_{t}=\Omega_{1,t}\cap \Omega_2$) 
is connected. 
\end{itemize}
We emphasize that no regularity assumption is made on the domain (apart from the regularity of the top boundary $\Sigma_t$).
Notice that our setting contains both cases of infinite depth and bounded depth bottoms (and all cases in-between).

\begin{figure}[h]
\includegraphics{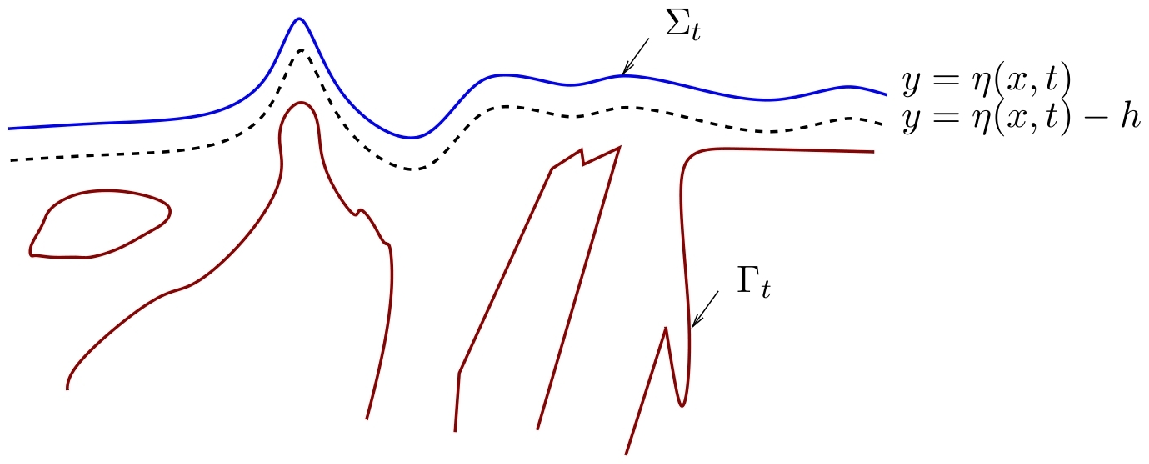}
\begin{center}{The domain}\end{center}
\end{figure}

A key feature of the water waves equations is that
there are two boundary conditions on the free surface $\Sigma_t=\{y=\eta(t,x)\}$. Namely, 
we consider a potential flow so that 
the velocity field is the gradient of a potential~$\phi=\phi(t,x,y)$
which is a harmonic function. 
The water waves equations are then given by the Neumann boundary condition on the bottom $\Gamma$, 
and the classical kinematic and dynamic boundary conditions on the free surface $\Sigma_t$. The system reads 
\begin{equation}\label{sI}
\left\{
\begin{aligned}
&\Deltax \phi+\partial_{y}^2\phi=0 &&\text{in \,}\Omega_t,\\
&\partial_{t} \eta = \partial_{y}\phi -\partialx\eta\cdot\partialx \phi &&\text{on } \Sigma_t, \\
&\partial_{t}\phi=-g \eta + \kappa H(\eta)-\frac{1}{2}\la \partialx\phi\ra^2 -\frac{1}{2}\la \partial_{y}\phi\ra^2
&&\text{on } \Sigma_t ,\\
&\partial_{n}\phi = 0 &&\text{on } \Gamma,
\end{aligned}
\right.
\end{equation}
where $\partialx=(\partial_{x_i})_{1\le i\le d}$, $\Deltax=\sum_{i=1}^d \partial_{x_i}^2$, $n$ is the normal to the boundary $\Gamma$, 
$g>0$ denotes the acceleration of gravity, $\kappa\ge 0$ is the coefficient of surface tension 
and $H(\eta)$ is the mean curvature of the free surface:
\begin{equation*}
H(\eta)= \cnx \left(\frac{\partialx\eta}{\sqrt{1+|\partialx\eta|^2}}\right).
\end{equation*}
We are concerned with the problem with surface tension and then we set $\kappa=1$. 

Since we make no regularity assumption on the bottom, 
to make sense of the system~\eqref{sI} requires some care (see Section~\ref{sec:2} for a precise definition).  
Following Zakharov we shall first define $\psi=\psi(t,x) \in\xR$ by
$$
\psi(t,x)=\phi(t,x,\eta(t,x)),
$$
and for $\chi \in C^\infty _0 ( ]-1,1[)$ equal to $1$ near $0$, 
$$
\widetilde \psi (t, x,y) = \chi\left(\frac{y-\eta(t,x)}{h} \right) \psi(t,x).
$$ 
The function  $\phi$ being harmonic, $\phi - \widetilde \psi= \widetilde \phi $ will be 
defined as the variational solution of the system 
$$ -\Deltayx \widetilde \phi =\Deltayx \widetilde \psi \quad\text{in }\Omega_t, 
\qquad \widetilde \phi\mid_{\Sigma_t} = 0, \qquad \partial_{n}\phi\mid_{\Gamma}  = 0.
$$ 
Let us now define the Dirichlet-Neumann operator by 
\begin{align*}
(G(\eta) \psi)  (t,x)&=
\sqrt{1+|\partialx\eta|^2}\,
\partial _n \phi\arrowvert_{y=\eta(t,x)},\\
&=(\partial_y \phi)(t,x,\eta(t,x))-\partialx \eta (t,x)\cdot (\partialx \phi)(t,x,\eta(t,x)).
\end{align*}
Now  $(\eta,\psi)$ solves
\begin{equation}\label{system}
\left\{
\begin{aligned}
&\partial_{t}\eta-G(\eta)\psi=0,\\
&\partial_{t}\psi+g \eta- H(\eta)
+ \frac{1}{2}\la\partialx \psi\ra^2  -\frac{1}{2}
\frac{\bigl(\partialx  \eta\cdot\partialx \psi +G(\eta) \psi \bigr)^2}{1+|\partialx  \eta|^2}
= 0.
\end{aligned}
\right.
\end{equation}

The fact that the Cauchy problem (without bottom) is well posed 
was proved by Beyer and G\"unther in~\cite{BG}. 
This result, as well as related uniform estimates with respect to $\kappa$, 
have been obtained by different proofs 
in \cite{Sch,Iguchi,AmMa,CS,MZ,RT}. The purpose of this article is twofold: 
first we want to clarify the Cauchy theory 
as well in terms of regularity indexes for the initial conditions as for the smoothness 
of the bottom (to our knowledge, previous results required the bottom to be the graph 
of an $H^{13}$ function). Second we want to show that the smoothing effect for the 
(one dimensional) surface tension water waves, as proved in~\cite{CHS}, is in 
fact a rather direct consequence of the paralinearization approach developped in~\cite{AM}.

Our first result (Cauchy theory) is the following
\begin{theo}
\label{theo:Cauchy}
Let $d\ge 1$, $s>2+d/2$ and 
$(\eta_{0},\psi_{0})\in H^{s+\mez}(\xR^d)\times H^{s}(\xR^d)$ be such that 
the assumption $H_{t=0}$ is satisfied.  
Then there exists $T>0$ such that 
the Cauchy problem for \eqref{system} 
with initial data  $(\eta_{0},\psi_{0})$ has a unique solution 
$$
(\eta,\psi)\in C^0\big([0,T];H^{s+\mez}(\xR^d)\times H^{s}(\xR^d)\big)
$$
such that the assumption $H_t$ is satisfied for $t\in [0, T]$.
\end{theo}
\begin{rema}
The assumption $\psi_{0}\in H^s(\xR^d)$ could be replaced by $\partialx\psi_{0}\in H^{s-1}(\xR^d)$. 
We then obtain solutions such that 
$\psi-\psi_{0}\in C^0\big([0,T];H^{s}(\xR^d)\big)$ (cf \cite{LannesJAMS}). 
Notice that our thresholds of regularities appear to be the natural ones, 
as they  control the Lipschitz norm of the non-linearities. However, working at that level of regularity 
gives rise to many technical difficulties, which would be avoided by choosing $s> 3 + \frac d 2$. 
\end{rema}

Our second result is the following $1/4$-smoothing effect for 2D-water waves. 

\begin{theo}\label{theo:main}
Assume that $d=1$ and let $s>5/2$ and $T>0$. Consider a   solution 
$(\eta,\psi)$ of ~\eqref{system} on the time interval $[0,T]$, 
such that $\Omega_t$ satisfies the assumption $H_t$. If 
$$
(\eta,\psi)\in C^0\big([0,T];H^{s+\mez}(\xR)\times H^{s}(\xR)\big),
$$
then 
$$
\L{x}^{-\mez-\delta}
(\eta,\psi)\in L^2\big(0,T;H^{s+\frac{3}{4}}(\xR)\times H^{s+\frac{1}{4}}(\xR)\big),
$$
for any $\delta>0$.
\end{theo}

This $1/4$-smoothing effect was first established recently by Christianson, 
Hur and Staffilani in~\cite{CHS} by a different method. 
Theorem~\ref{theo:main} improves the result in \cite{CHS} in the following directions. 
Firstly, we obtain the smoothing effect 
on the lifespan of the solution and not only for a time small enough. 
Secondly, we lower the index of regularity (in \cite{CHS} the authors require $s\ge 15$) 
and we improve the decay rate in space to the optimal one (in \cite{CHS} the authors require $\delta > 5/2$). 
In addition, we allow much more general domains, which is interesting for 
applications to the cases where one takes into account the surface tension effect. 
Notice finally that our proof would apply to the radial case in dimension $3$. 

\smallbreak

Many variations are possible concerning the fluid domain. 
Our method would apply to the case where the free surface is not a graph 
over the hyperplane $\xR ^{d} \times \{0\}$, but rather a graph over a fixed hypersurface. 
Our results hold also in the case where the bottom is time-dependent, under an additional 
Lipschitz regularity assumption on the bottom and we prove (see Appendix~\ref{app}) 
\begin{theo}\label{th.4} Assume that the domain is time dependent and 
satisfies the assumptions $H_2), H_3)$ in Appendix~\ref{app}. Then the conclusions in 
Theorems~\ref{theo:Cauchy} and~\ref{theo:main} still hold for the system 
of the water-wave equations with time dependent bottom~\eqref{sIbis}.
\end{theo}

 To prove Theorem~\ref{theo:main}, we start in \S\ref{sec:2} by defining and proving regularity properties of the Dirichlet-Neumann operator. 
Then in \S\ref{s3} we perform several reductions to a paradifferential system on the boundary by means of 
the analysis in \cite{AM}. The key technical lemma in this paper in a reduction of the system~\eqref{system} to a simple hyperbolic form. 
To perform this reduction, we prove in \S\ref{Sym} the existence of a paradifferential symmetrizer. 
We deduce Theorem~\ref{theo:Cauchy} from this symmetrization in \S\ref{sle}. 
Theorem~\ref{theo:main} is then proved in \S\ref{s4} by means of Doi's approach~\cite{Doi1,Doi2}.
Finaly, we give in Appendix~\ref{app} the modifications required to prove Theorem~\ref{th.4}. 
Note that our strategy is based on a direct analysis in Eulerian coordinates. 
In this direction it is influenced by the 
important paper by Lannes (\cite{LannesJAMS}). 

As it was shown by Zakharov (see \cite{Zakharov} and references there in), 
the system \eqref{system} is a Hamiltonian one, of the form
$$
\frac{\partial \eta}{\partial t}=\frac{\delta H}{\delta\psi} ,\quad 
\frac{\partial \psi}{\partial t}=-\frac{\delta H}{\delta\eta},
$$
where $H$ is the total energy of the system. 
Denoting by $H_0$ the Hamiltonian associated to 
the linearized system at the origin, we have
$$
H_0=\mez \int \left[ \la \xi\ra | \widehat{\psi}|^2 + ( g  +\la \xi\ra^2) | \widehat{\eta}|^2 \right]\, d\xi,
$$
where $\widehat{f}$ denotes the Fourier transform, 
$\widehat{f}(\xi)=\int e^{-ix\cdot \xi} f(x)\, dx$. An important observation is that 
the canonical transformation $(\eta,\psi)\mapsto a$ with 
$$
\widehat{a}= \frac{1}{\sqrt{2}}\bigg\{ \bigg(\frac{ g  +\la \xi\ra^2}{\la \xi\ra}\bigg)^{1/4}\widehat{\eta}
-i \bigg( \frac{\la \xi\ra}{ g  +\la \xi\ra^2}\bigg)^{1/4}\widehat{\psi}\bigg\},
$$
diagonalizes the Hamiltonian $H_0$ and reduces the analysis 
of the linearized system to one complex equation (see \cite{Zakharov}). 
We shall show that there exists a similar diagonalization for the nonlinear equation, 
by using paradifferential calculus instead of Fourier transform. 
As already mentionned, this is the main technical result in this paper. 
In fact, we strongly believe that all dispersive estimates on the water waves system with surface tension 
could be obtained by using our reduction.

\smallbreak



\section{The Dirichlet-Neumann operator}\label{sec:2}
\subsection{Definition of the operator}\label{sec.2.1}
The purpose of this section is to define the Dirichlet-Neumann operator and prove some basic regularity properties.
Let us recall that we assume that $\Omega_t$ is the intersection of the half space located below the free surface 
$$ 
\Omega_{1,t}= \{  (x,y)\in \xR^d\times \xR\, : \, y< \eta(t,x)\}
$$ 
and an open set $\Omega_{2}\subset \xR^ {d+1}$ 
and that $\Omega_{2}$ contains a fixed strip around $\Sigma_t$, which means that there exists $h>0$ such that
$$
\{ (x,y)\in \xR^d\times \xR\, : \, \eta(t,x) -h \le y \leq \eta(t,x) \} \subset \Omega_{2}.
$$
We shall also assume that the domain $\Omega_2$ (and hence the domain $\Omega_t$) 
is connected. In the remainder of this subsection, we will drop the time dependence of the domain, 
and it will appear clearly from the proofs that all estimates are uniform as long as $\eta(t,x)$ 
remains bounded in the set of functions such that $\|\eta(t, \cdot)\|_{H^s(\mathbb{R}^d)}$ remains bounded. 

Below we use the following notations
$$
\partialx=(\partial_{x_i})_{1\le i\le d},\quad \partialyx =(\partialx,\partial_y), 
\quad \Delta = \sum_{1\le i\le d} \partial_{x_i}^2,\quad 
\Deltayx = \Delta+\partial_y^2.
$$
\begin{nota}\label{notaD}
Denote by $\mathscr{D}$ the space of functions $u\in C^{\infty}(\Omega)$ such that 
$\nabla_{x,y}u\in L^2(\Omega)$. We then define $\mathscr{D}_0$ as the subspace of functions $u\in \mathscr{D}$ such that $u$ is 
equal to $0$ near the top boundary $\Sigma$.
\end{nota}
\begin{prop}\label{corog}
There exists a positive weight $g\in L^\infty_{loc} ( \Omega)$, equal to~$1$ 
near the top boundary of $\Omega$ and a positive constant $C$ such that 
\begin{equation}\label{eq.Poinc}
\int_\Omega g(x,y) |u(x,y)|^2 \,dx dy \leq  C \int_{\Omega} |\partialyx u (x,y)|^2 \,dx dy,
\end{equation}
for all $u\in \mathscr{D}_0$.
\end{prop}
Here is the proof. Let us set 
\begin{equation}\label{defiO1}
\begin{aligned}
\mathcal{O}_1&=\left\{ (x,y)\in\xR^d\times\xR\,:\, \eta(x)-h<y<\eta(x)\right\},\\
\mathcal{O}_2&=\Big\{ (x,y)\in\Omega\,:\, y<\eta(x)-h\Big\}.
\end{aligned}
\end{equation}
To prove Proposition~\ref{corog}, the starting point is the following Poincar\'e inequality on $\mathcal{O}_1$.
\begin{lemm}\label{lem.Poinc}
For all $u\in \mathscr{D}_0$ we have
$$
\int_{\mathcal{O}_1} |u|^2 \, dx dy \leq h^2 \int_{\Omega} |\partialyx u |^2 \, dx dy.
$$
\end{lemm}
\begin{proof}
For $(x,y)\in \mathcal{O}_1$ we can write $u(x,y)=-\int_{y}^{\eta(x)}(\partial_y u)(x,z)\, dz$, 
so using the H\"older inequality we obtain 
$$
 | u(x,y)|^2 \le  h \int_{\eta(x)-h}^{\eta(x)}\la (\partial_y u)(x,z)\ra^2 \, dz.
 $$
Integrating on $\mathcal{O}_1$ we obtain the desired conclusion.
\end{proof}


\begin{lemm}\label{lem.accr}
Let $m_0\in \Omega$ and $\delta>0$ such that 
$$
B(m_0, 2\delta) = \{ m\in \xR^d\times \xR \,:\, |m- m_0| <2 \delta\} \subset \Omega.
$$
Then for any $m_1\in B(m_0, \delta)$ and any $u\in\mathscr{D}$,
\begin{equation}\label{eq.accr} 
\int_{B(m_0, \delta) } \la u\ra^2 \,dxdy \leq 2 \int_{B(m_1, \delta)} \la u\ra^2 \,dxdy
+ 2 \delta ^2 \int_{B(m_0, 2\delta)} \la \partialyx u\ra^2 \,dxdy. 
\end{equation}
\end{lemm}
\begin{proof}Denote by $v= m_0 - m_1$ and write
$$ u(m+ v) = u(m) + \int_0^1 v\cdot \partialyx u (m + tv)dt $$
As a consequence, we get
$$
|u(m+v)|^2 \leq 2 |u(m)|^2 + 2 |v|^2 \int_0^1 \la \partialyx u(m+ tv)\ra^2 dt,
$$
and integrating this last inequality on $B(m_1, \delta)\subset B(m_0,2\delta)\subset \Omega$, 
we obtain~\eqref{eq.accr}.
\end{proof}
\begin{coro}\label{coro.K}
For any compact $K\subset \mathcal{O}_2$, there exists a  constant $C(K)>0$ such that, for all $u\in\mathscr{D}_0$, we have
 $$ \int_K |u|^2 \, dx dy \leq C(K) \int_{\Omega} |\partialyx u |^2 \, dx dy.
$$
\end{coro}
\begin{proof}
Consider now an arbitrary point $m_0\in\mathcal{O}_2$. Since $\Omega$ is open and connected, there exists a continuous map $\gamma: [0,1]\rightarrow \Omega$ such that $\gamma(0)=m_0$ and $\gamma(1)\in \mathcal{O}_1$.  By compactness, there exists $\delta>0$ such that for any $t\in [0,1]$ $B(\gamma(t), 2 \delta) \subset \Omega$. Taking smaller $\delta$ if necessary, we can also assume that $B(\gamma(1),\delta)\subset \mathcal{O}_1$ so that by Lemma~\ref{lem.Poinc}
$$ \int_{B(\gamma(1), \delta)} |u|^2 \, dx dy \leq C \int_{\Omega} |\partialyx u |^2 \, dx dy.
$$
We now can find a sequence $t_0=0, t_1, \cdots, t_N=1$ such that the points $m_n= \gamma(t_n)$ satisfy $m_{N+1}\in B(m_N, \delta)$. Applying Lemma~\ref{lem.accr} successively, we obtain
$$
\int_{B(m_0, \delta)} |u|^2 \, dx dy \leq C' \int_{\Omega} |\partialyx u |^2 \, dx dy.
$$
Then Corollary~\ref{coro.K} follows by compactness.
\end{proof}

\begin{proof}[Proof of Proposition~\ref{corog}] 
Writing $\mathcal{O}_2= \cup_{n=1}^\infty K_n$, 
and taking a partition of unity $(\chi_n)$ such that $0\leq \chi_n\leq 1 $ and $\supp \chi_n\subset K_n$, 
we can define the continuous function
$$
\widetilde{g}(x,y) = \sum_{n=1}^\infty \frac{ \chi_n(x,y)} { (1+C(K_n)) n^2},
$$
which is clearly positive. Then by Corollary~\ref{coro.K}, 
\begin{equation}\label{w2}
\begin{aligned}
\int_{\mathcal{O}_2}\widetilde{g}(x,y)\la u\ra^2\, dxdy 
&\le  \sum_{n=1}^\infty \frac{ 1} { (1+C(K_n)) n^2} \int_{K_n}\la u\ra^2\, dx dy\\
&\le 2\int _{\mathcal{O}_2}\la \nabla_{x,y} u\ra^2\, dx dy.
\end{aligned}
\end{equation}
Finally, let us set
$$
g(x,y)=1\quad\text{for } (x,y)\in \mathcal{O}_1,\quad
g(x,y)=\widetilde{g}(x,y)\quad\text{for } (x,y)\in \mathcal{O}_2.
$$
Then Proposition~\ref{corog} follows from Lemma~\ref{lem.Poinc} and \eqref{w2}.
\end{proof}

We now introduce the space in which we shall solve the variational formulation of our Dirichlet problem.
\begin{defi} Denote by $H^{1,0}( \Omega)$ the space of functions $u$ on $\Omega$ 
such that there exists a sequence $(u_n)\in \mathscr{D}_0$ such that,
$$ \partialyx u_n  \rightarrow \partialyx u \text{ in } L^2( \Omega, dxdy), 
\qquad u_n  \rightarrow  u \text{ in } L^2( \Omega, g(x,y)dxdy).$$
We endow the space $H^{1,0}$ with the norm
$$ \lA u\rA = \lA \partialyx u \rA_{L^2( \Omega)}.
$$
\end{defi}  
The key point is that the space $H^{1,0}(\Omega)$ is a Hilbert space. Indeed, passing to the limit in~\eqref{eq.Poinc}, 
we obtain first  that by definition, the norm on $H^{1,0}(\Omega)$ is equivalent to 
$$
\lA \partialyx u \rA_{L^2( \Omega, dxdy)}+ \lA u \rA_{L^2( \Omega, g(x,y)dxdy)}.
$$ 
As a consequence, if $(u_n)$ is a Cauchy  sequence in $H^{1,0}(\Omega)$, we obtain easily from 
the completeness of $L^2$ spaces that there exists $u \in L^2( \Omega, g(x,y) dxdy)$ 
and $v \in L^2( \Omega, dxdy)$ such that 
$$
u_n \rightarrow u \text{ in }  L^2( \Omega, g(x,y) dxdy),  
\qquad \partialyx u_n \rightarrow v \text{ in } L^2(\Omega, dxdy).
$$
But the convergence in $ L^2( \Omega, g(x,y) dxdy)$ implies the 
convergence in $\mathcal{D}'( \Omega)$ and consequently $v = \partialyx u$ 
in $\mathcal{D}'( \Omega)$ and it is easy to see that $u\in H^{1,0}(\Omega)$.

\smallbreak

We are now in position to define the Dirichlet-Neumann operator. 
Let $\psi (x) \in H^1( \xR^d)$. For $\chi \in C^\infty_0 (]-1,1[)$ equal to $1$ near $0$, we first define 
$$
\widetilde \psi (x,y)= \chi \left(\frac {y-\eta(x)} {h}\right) \psi(x) \in H^{1} ( \xR^{d+1}),
$$
which is the most simple lifting of $\psi$. Then the map
$$
v\mapsto \langle \Deltayx \widetilde \psi ,v\rangle  = - \int_{\Omega}  \partialyx  \widetilde \psi \cdot \partialyx v \, dx dy 
$$
is a bounded linear form on $H^{1,0}(\Omega)$. It follows from the Riesz theorem that there exists a unique 
$\widetilde{\phi}\in H^{1,0}(\Omega)$ such that
\begin{equation}\label{eq.variabis}
 \forall v \in H^{1,0} ( \Omega), \qquad \int_\Omega \partialyx \widetilde \phi \cdot \partialyx v \, dxdy
=  \langle \Deltayx \widetilde \psi ,v\rangle .
 \end{equation}
Then $\widetilde \phi$ solves the problem 
$$
- \Deltayx \widetilde \phi = \Deltayx \widetilde \psi\quad\text{in }\mathcal{D}'(\Omega), \qquad \widetilde \phi \mid_{\Sigma} =0, 
\qquad \partial _n \widetilde \phi \mid_{\Gamma} = 0,
$$
the latter condition being justified as soon as the bottom $\Gamma$ is regular enough.

We now set $\phi = \widetilde \phi + \widetilde \psi$ and define the Dirichlet-Neumann operator by
\begin{align*}
G(\eta) \psi  (x)&=
\sqrt{1+|\partialx\eta|^2}\,
\partial _n \phi\arrowvert_{y=\eta(x)},\\
&=(\partial_y \phi)(x,\eta(x))-\partialx \eta (x)\cdot (\partialx \phi)(x,\eta(x)),
\end{align*}
Notice that a simple calculation shows that this definition is independent on the choice of the lifting function $\widetilde\psi$ as long as it remains bounded in 
$H^1(\Omega)$ and vanishes near the bottom.

\subsection{Boundedness on Sobolev spaces}

\begin{prop}\label{estDN}
Let $d\ge 1$, $ s > 2+\frac{d}{2}$ and $1  \le \sigma \le s$. Consider $\eta\in H^{s+\mez}(\xR^d)$. Then 
$G(\eta)$ maps $H^{\sigma}(\xR^d)$ to $H^{\sigma-1}(\xR^d)$. Moreover, 
there exists a function $C$ such that, for all 
$\psi \in H^\sigma(\xR ^d)$ and $\eta \in H^{s+\mez } ( \xR ^d)$,
$$
\lA G(\eta)\psi\rA_{H^{\sigma-1}(\xR^d)}\le C\left(\lA \eta\rA_{H^{s+\mez}}\right)\lA \partialx\psi\rA_{H^{\sigma-1}}.
$$
\end{prop}
\begin{proof} The proof is in two steps.

\smallbreak
\noindent\textbf{First step: A localization argument.} 
Let us define (by regularizing the function $\eta$), a smooth function $\widetilde{\eta}\in H^\infty(\xR^d)$ such that
$\lA \widetilde{\eta} - \eta\rA_{L^\infty}\le h/100$ and $\lA \widetilde{\eta} - \eta\rA_{H^{s+1/2}}\le h/100$. We now set
$$
\eta_1=\widetilde{\eta}-\frac{9h}{20}.
$$
Then $\eta_{1}$ satisfies
\begin{equation}\label{eta1eta}
\eta(x)-\frac{h}{4}<\eta_1(x)\le \eta(x)-\frac{h}{5}.
\end{equation}

\begin{lemm}\label{cor01}
Consider for $-3h/4<a<b<h/5$ , the strip 
$$
S_{a,b}= \{ (x,y) \in \mathbb{R}^{d+1}; a<y-\eta_1(x)<b\},
$$
which is included in $\Omega$. Let $k\ge 1$ and assume that $\|\phi\|_{H^k(S_{a,b})}<+\infty $. 
Then for any $a<a'<b'<b$ there exists $C>0$ such that 
$$
\|\phi\|_{H^{k+1}(S_{a',b'})}\leq C \|\phi\|_{H^k(S_{a,b})}.
$$
\end{lemm} 
\begin{proof}
Choose a function $\chi \in C^\infty_0 (a,b)$ equal to $1$ on $(a',b')$. 
The function $w= \chi(y- \eta_1(x)) \phi(x,y)$ is solution to 
$$
\Deltayx w= [\Deltayx, \chi(y- \eta_1(x))] \phi,
$$
and since the assumption implies that the right hand side 
is bounded in $H^{k-1}$, the result follows from the (explicit) elliptic regularity of the operator $\Deltayx$ in $\xR ^{d+1}$. 
\end{proof}

\begin{lemm}\label{cor.1}
Assume that $-3h/4<a<b<h/5$  then the strip $S_{a,b}= \{ (x,y) \in \mathbb{R}^{d+1}\,:\, a<y- \eta_1(x)<b\}$ 
is included in $\Omega$ and for any $k \ge 1$, there exists $C>0$ such that 
$$
\lA\phi\rA_{H^{k}(S_{a,b})}\leq C \|\psi\|_{H^1(\xR^d)}.
$$
\end{lemm}
\begin{proof}
It follows from the variational problem \eqref{eq.variabis}, the definition of $\phi=\widetilde{\phi}+\widetilde{\psi}$, that
$$
\lA \nabla_{x,y}\phi\rA_{L^2(\Omega)}
\le c \lA \psi\rA_{H^1(\xR^d)}.
$$
Noticing that $S_{a,b}\subset \mathcal{O}_1$ (cf \eqref{defiO1}) and applying Lemma~\ref{lem.Poinc} we obtain 
the {\em a priori} $H^1$ bound 
\begin{align*}
\lA \phi\rA_{H^1(S_{a,b})}
\le \lA \phi\rA_{H^1(\mathcal{O}_1)} 
\le (1+h)\lA \nabla_{x,y}\phi\rA_{L^2(\Omega)}
\le c(1+h) \lA \psi\rA_{H^1(\xR^d)}.
\end{align*}
Since it is always possible to chose $a<a_2<\cdots < a_k= a' <b'= b_k < \cdots < b_2<b$, we 
deduce Lemma~\ref{cor.1} from Lemma~\ref{cor01}.
\end{proof}

We next introduce $\chi_0\in C^{\infty}(\xR)$ such that $0\le \chi_0\le 1$, 
$$
\chi_0(z)= 1 \quad \text{for } z\ge 0,
\quad
\chi_0(z)= 0 \quad \text{for }  z\le -\frac{1}{4}
$$
Then the function
$$
\Phi(x,y)=\chi_0\left(\frac{y-\eta_1(x)}{h}\right)\phi(x,y)
$$
is solution to 
$$
\Deltayx \Phi =f\defn\left[ \Deltayx,\chi_0\left(\frac{y-\eta_1(x)}{h}\right)\right] \phi.
$$
In view of \eqref{eta1eta}, notice that $f$ is supported in a set where $\phi$ is $H^\infty$ according to Lemma~\ref{cor.1}, we find that 
$$
f\in H^\infty(\Pi_\eta)\quad\text{where } \Pi_\eta\defn \left\{ (x,y)\in \xR^{d}\times\xR\,:\, \eta(x)-h<y\le \eta(x)\right\}.
$$
In addition, using that $\chi_0(0)=1$ and that $\Phi(x,y)$ is identically equal to $0$ near the set $\{y=\eta-h \}$, we immediately verify 
that $\Phi$ satisfies the boundary conditions
$$ 
\Phi \mid _{y= \eta(x) } = \psi (x), \quad 
\partial_y\Phi\mid_{y=\eta(x) - h} = 0,\quad
\Phi\mid_{y=\eta(x) - h} = 0.
$$


The fact that the strip $\Pi_\eta$ depends on $\eta$ and not on $\eta_1$ is not a typographical error. Indeed, with this choice, 
the strip $\Pi_\eta$ is made of two parallel curves. As a result, a very simple (affine) change of variables will flatten both the top surface $\{y=\eta(x)\}$
and the bottom surface $\{y=\eta(x)-h\}$. 
\smallbreak
\noindent\textbf{Second step: Elliptic estimates.} 
To prove elliptic estimates, we shall consider the most simple change of variables. Namely, introduce
$$
\rho(x,z) =hz + \eta(x).
$$
Then
$$
(x,z) \mapsto (x, \rho(x,z)),
$$
is a diffeomorphism from the strip $\xR^d\times [-1,0]$ to the set 
$$
\left\{\,(x,y)\in\xR^d\times \xR\,:\, \eta(x)-h \leq y \leq \eta(x)\,\right\}.
$$
Let us define the function $v\colon  \xR^d \times[-1,0] \rightarrow \xR$ by
\begin{equation}\label{defiv}
v(x,z)= \Phi(x, \rho(x,z)).
\end{equation}
From $\Deltayx \Phi=f$ with $f\in H^\infty(\Pi_\eta)$, we deduce that $v$ satisfies the elliptic equation
\begin{equation}\label{evpf}
\left(\frac{1}{ \partial_z \rho}\partial_z \right)^2 v
+ \left( \partialx -\frac{\partialx \rho}{ \partial_z \rho}\partial_z\right)^2 v=g,
\end{equation}
where $g(x,z)=f(x,hz+\eta(x))$ is in $C^2_z([-1,0];H^{s+\mez}(\xR^d_x))$. This yields
\begin{equation}\label{dnint1}
\alpha \partial_z^2 v +\Delta v + \beta \cdot\partialx\partial_z v  - \gamma \partial_z v=g,
\end{equation}
where 
\begin{equation}\label{abc}
\alpha\defn \frac{1+|\partialx  \eta |^2 }{ h^2},\quad 
\beta\defn  - \frac{2\partialx  \eta}{h},\quad 
\gamma \defn \frac{\Delta \eta}{h }  .
\end{equation}
Also $v$ satisfies the boundary conditions 
\begin{equation}\label{dnbord}
v\big\arrowvert_{z=0}= \psi, \qquad  \partial_z v  \arrowvert_{z=-1}=0,\quad v\arrowvert_{z=-1}=0.
\end{equation}
We are now in position to apply elliptic regularity results obtained by 
Alvarez-Samaniego and Lannes in~\cite[Section 2.2]{ASL} to deduce the following result.
\begin{lemm}\label{L1}
Suppose that $v$ satisfies the elliptic equation \eqref{dnint1} with the boundary condtions 
\eqref{dnbord} with 
$\psi \in H^\sigma(\xR ^d)$ and $\eta \in H^{s+\mez } ( \xR ^d)$ where $1 \le \sigma \leq s$, $ s > 2+\frac{d}{2}$, $\dist(\Sigma,\Gamma)>0$. 
Then 
$$
\partialx v,\partial_{z}v\in L^2_z\bigl([-1,0];H_x^{\sigma-\frac{1}{2}}(\xR^d)\bigr).
$$
\end{lemm}


It follows from Lemma~\ref{L1} and a 
classical interpolation argument that 
$(\partialx v,\partial_{z}v)$ are continuous in $z\in [-1,0]$ with values in $H^{\sigma-1}(\xR^d)$. 
Now note that, by definition,
\begin{equation*}
G(\eta)\psi=
\frac{1 + | \partialx \eta |^2}{h} \partial_z v  - \partialx \eta \cdot\partialx v
\Big\arrowvert_{z=0}.
\end{equation*}
Therefore, 
we conclude that $G(\eta)\psi\in H^{\sigma-1}(\xR^d)$. 
Moreover we have the desired estimate. 

This completes the proof of Proposition~\ref{estDN}.
\end{proof}

\subsection{Linearization of the Dirichlet-Neumann operator}\label{s.2.3}
The next pro\-position gives an explicit expression of the shape derivative 
of the Dirichlet-Neumann operator, 
that is, of its derivative with respect to the surface parame\-trization.

\begin{prop}\label{Lannes1}
Let $\psi \in H^\sigma(\xR ^d)$ and $\eta \in H^{s+\mez } ( \xR ^d)$ with 
$1 \le  \sigma \leq s$, $ s > 2+\frac{d}{2}$ be such that $\dist(\Sigma,\Gamma)>0$. 
Then there exists a neighborhood $\mathcal{U}_\eta\subset H^{s+\mez}(\xR^d)$ of $\eta$ such that
the mapping
$$
\sigma \in \mathcal{U}_\eta\subset H^{s+\mez}(\xR^d) \mapsto G(\sigma)\psi \in H^{\sigma-1}(\xR^d)
$$
is differentiable. Moreover, for all $h\in H^{s+\mez}(\xR^d)$, we have 
\begin{align*}
dG(\eta)\psi \cdot  h \defn
\lim_{\eps\rightarrow 0} \frac{1}{\eps}\big\{ G(\eta+\eps h)\psi -G(\eta)\psi\big\}
= -G(\eta)( \mathfrak{B} h) -\cnx (V h),
\end{align*}
where
$$
\mathfrak{B}=\frac{\partialx\eta\cdot\partialx\psi+G(\eta)\psi}{1+\la \partialx\eta\ra^2},
\quad V =\partialx\psi-\mathfrak{B}\partialx \eta.
$$
\end{prop}
The above result goes back to Zakharov~\cite{Zakharov}. Notice that in the previous paragraph we reduced the analysis 
to studying an elliptic equation in a flat strip $\xR^d\times [-1,0]$. As a consequence, the proof of this result by Lannes~\cite{LannesJAMS} applies (see also \cite{BLS,IP,AM}).

\smallbreak
Let us mention a key cancellation in the previous formula, which is proved in \cite[Lemma 1]{BLS} (see also~\cite{LannesJAMS}).
\begin{lemm}\label{cancellation}
We have $G(\eta)\mathfrak{B}=-\cnx V$.
\end{lemm}
\begin{proof}
Recalling that, by definition,
$$
G(\eta)\psi=( \partial_y \phi-\nabla\eta\cdot\nabla\phi)\big\arrowvert_{y=\eta},
$$
and using the chain rule to write 
$$
\nabla \psi=\nabla(\phi\arrowvert_{y=\eta})=(\partialx \phi + \partial_y \phi \partialx\eta)\big\arrowvert_{y=\eta},
$$
we obtain
\begin{align*}
\mathfrak{B}&\defn
\frac{\partialx\eta\cdot\partialx\psi+G(\eta)\psi}{1+\la \partialx\eta\ra^2}\\
&=\frac{1}{1+\la \partialx\eta\ra^2}\left\{ \partialx\eta\cdot (\partialx \phi + \partial_y \phi \partialx\eta)+\partial_y \phi-\nabla\eta\cdot\nabla\phi\right\}
\big\arrowvert_{y=\eta}
=(\partial_y \phi)\big\arrowvert_{y=\eta}.
\end{align*}
Therefore the function $\Phi$ defined by $\Phi (x,y)=\partial_y \phi (x,y)$  is the solution to the system
$$
\Deltayx \Phi =0,\quad \Phi\arrowvert_{y=\eta}=\mathfrak{B},\quad
\partial_n \Phi\arrowvert_{\Gamma}=0.
$$
Consequently, directly from the definition of the Dirichlet-Neumann operator, we  have
$$
G(\eta)\mathfrak{B}=
\partial_y \Phi - \partialx\eta\cdot \partialx \Phi \big\arrowvert_{y=\eta}.
$$
Now we have $\partial_y \Phi=\partial_y^2 \phi=-\Delta\phi$ and hence
$$
G(\eta)\mathfrak{B}=
-\Delta \phi - \partialx\eta\cdot \partialx \Phi \big\arrowvert_{y=\eta}.
$$
On the other hand, directly from the definition of $V$, we have
$$
\cnx V = \cnx (\partialx\psi-\mathfrak{B}\partialx\eta) 
=\Delta \psi -\cnx(\mathfrak{B}\nabla\eta).
$$
Using that $\psi(x)=\phi(x,\eta(x))$, we check that
\begin{align*}
\Delta \psi &= \cnx \nabla \psi = \cnx \big( \nabla\phi \big\arrowvert_{y=\eta}+  \partial_y \phi  \big\arrowvert_{y=\eta}  \nabla\eta\big)\\
&= \left(\Delta \phi + \partialx\partial_y \phi \cdot \partialx \eta \right) \big\arrowvert_{y=\eta}
+\cnx \big( \partial_y \phi  \big\arrowvert_{y=\eta}  \nabla\eta \big)\\
&=\left(\Delta \phi + \partialx\partial_y \phi \cdot \partialx \eta \right) \big\arrowvert_{y=\eta}+\cnx ( \mathfrak{B}\nabla\eta )
\end{align*}
so that
\begin{align*}
\cnx V &=\Delta\psi - \cnx ( \mathfrak{B}\nabla\eta )=\left(\Delta\phi + \partialx\partial_y \phi \cdot \partialx \eta \right)\big\arrowvert_{y=\eta}\\
&=(\Delta\phi +  \partialx \Phi\cdot\partialx\eta) \big\arrowvert_{y=\eta}
=-G(\eta)\mathfrak{B},
\end{align*}
which is the desired identity.
\end{proof}

\section{Paralinearization}\label{s3}

\subsection{Paradifferential calculus}\label{s2}

In this paragraph we 
review notations and results about Bony's paradifferential calculus. 
We refer to \cite{Bony,Hormander,MePise,Meyer,Taylor} for the general theory. 
Here we follow the presentation by M\'etivier in \cite{MePise}. 

\smallbreak

For $\rho\in\xN$, according to the usual definition, we denote 
by $W^{\rho,\infty}(\xR^d)$ the Sobolev spaces of $L^\infty$ functions 
whose derivatives of order $\leo \rho$ are in $L^\infty$. 
For $\rho\in ]0,+\infty[\setminus \xN$, we denote 
by $W^{\rho,\infty}(\xR^d)$ the 
space of bounded functions whose derivatives of order $[\rho]$  are uniformly H\"older continuous with 
exponent $\rho- [\rho]$. Recall also that, for all $C^\infty$ function $F$, if 
$u \in W^{\rho,\infty}(\xR^d)$ for some $\rho\ge 0$ then $F(u)\in W^{\rho,\infty}(\xR^d)$.

\begin{defi}
Given $\rho\ge 0$ and $m\in\xR$, $\Gamma_{\rho}^{m}(\xR^d)$ denotes the space of
locally bounded functions $a(x,\xi)$
on $\xR^d\times(\xR^d\setminus 0)$,
which are $C^\infty$ with respect to $\xi$ for $\xi\neq 0$ and
such that, for all $\alpha\in\xN^d$ and all $\xi\neq 0$, the function
$x\mapsto \partial_\xi^\alpha a(x,\xi)$ belongs to $W^{\rho,\infty}(\xR^d)$ and there exists a constant
$C_\alpha$ such that,
\begin{equation}\label{para:10}
\forall\la \xi\ra\ge \mez,\quad \lA \partial_\xi^\alpha a(\cdot,\xi)\rA_{W^{\rho,\infty}}\le C_\alpha
(1+\la\xi\ra)^{m-\la\alpha\ra}.
\end{equation}
\end{defi}

We next introduce the spaces of (poly)homogeneous symbols.
\begin{defi}
i) $\dot\Gamma_{\rho}^{m}(\xR^d)$ denotes the subspace of 
$\Gamma_{\rho}^{m}(\xR^d)$ which consists of symbols 
$a(x,\xi)$ which are homogeneous of degree $m$  
with respect to $\xi$. 

\noindent ii) If 
$$
a=\sum_{0\le j <\rho}a^{(m-j)} \qquad (j\in \xN),
$$
where $a^{(m-j)}\in \dot\Gamma^{m-j}_{\rho-j}(\xR^d)$, then we say that 
$a^{(m)}$ is the principal symbol of $a$.
\end{defi}

Given a symbol $a$, we define
the paradifferential operator $T_a$ by
\begin{equation}\label{eq.para}
\widehat{T_a u}(\xi)=(2\pi)^{-d}\int \chi(\xi-\eta,\eta)\widehat{a}(\xi-\eta,\eta)\psi(\eta)\widehat{u}(\eta)
\, d\eta,
\end{equation}
where
$\widehat{a}(\theta,\xi)=\int e^{-ix\cdot\theta}a(x,\xi)\, dx$
is the Fourier transform of $a$ with respect to the first variable; 
$\chi$ and $\psi$ are two fixed $C^\infty$ functions such that:
$$
\psi(\eta)=0\quad \text{for } \la\eta\ra\le 1,\qquad
\psi(\eta)=1\quad \text{for }\la\eta\ra\geq 2,
$$
and $\chi(\theta,\eta)$ is 
homogeneous of degree $0$ and satisfies, for $0<\eps_1<\eps_2$ small enough,
$$
\chi(\theta,\eta)=1 \quad \text{if}\quad \la\theta\ra\le \eps_1\la \eta\ra,\qquad
\chi(\theta,\eta)=0 \quad \text{if}\quad \la\theta\ra\geq \eps_2\la\eta\ra.
$$ 

\smallbreak

We shall use quantitative results from \cite{MePise} about operator norms estimates in symbolic calculus. 
To do so, introduce the following semi-norms.
\begin{defi}
For $m\in\xR$, $\rho\ge 0$ and $a\in \Gamma^m_{\rho}(\xR^d)$, we set
\begin{equation}\label{defi:norms}
M_{\rho}^{m}(a)= 
\sup_{\la\alpha\ra\le \frac{d}{2}+1+\rho ~}\sup_{\la\xi\ra \ge 1/2~}
\lA (1+\la\xi\ra)^{\la\alpha\ra-m}\partial_\xi^\alpha a(\cdot,\xi)\rA_{W^{\rho,\infty}(\xR^d)}.
\end{equation}
\end{defi}
\begin{rema}
If $a$ is homogeneous of degree $m$ in $\xi$, then 
$$
M_{\rho}^{m}(a)\le   K_{d,m}\sup_{\la\alpha\ra\le \frac{d}{2}+1+\rho ~}\sup_{\la\xi\ra =1 }
\lA \partial_\xi^\alpha a(\cdot,\xi)\rA_{W^{\rho,\infty}(\xR^d)}.
$$
\end{rema}

The main features of symbolic calculus for paradifferential operators are given by the following theorems.
\begin{defi}\label{defi:order}
Let $m\in\xR$.
An operator $T$ is said of order $\leo m$ if, for all $\mu\in\xR$,
it is bounded from $H^{\mu}$ to $H^{\mu-m}$.
\end{defi}
\begin{theo}\label{theo:sc0}
Let $m\in\xR$. If $a \in \Gamma^m_0(\xR^d)$, then $T_a$ is of order $\leo m$. 
Moreover, for all $\mu\in\xR$ there exists a constant $K$ such that
\begin{equation}\label{esti:quant1}
\lA T_a \rA_{H^{\mu}\rightarrow H^{\mu-m}}\le K M_{0}^{m}(a).
\end{equation}
\end{theo}

\begin{theo}[Composition]\label{theo:sc}
Let $m\in\xR$ and $\rho>0$. 
If $a\in \Gamma^{m}_{\rho}(\xR^d), b\in \Gamma^{m'}_{\rho}(\xR^d)$ then 
$T_a T_b -T_{a\# b}$ is of order $\leo m+m'-\rho$ where
$$
a\# b=
\sum_{\la \alpha\ra < \rho} \frac{1}{i^{\la\alpha\ra} \alpha !} \partial_\xi^{\alpha} a \partial_{x}^\alpha b.
$$
Moreover, for all $\mu\in\xR$ there exists a constant $K$ such that
\begin{equation}\label{esti:quant2}
\lA T_a T_b  - T_{a\# b}   \rA_{H^{\mu}\rightarrow H^{\mu-m-m'+\rho}}\le 
K M_{\rho}^{m}(a)M_{\rho}^{m'}(b).
\end{equation}
\end{theo}
\begin{rema}
We have the following corollary for poly-homogeneous symbols: if 
$$
a=\sum_{0\le j <\rho}a^{(m-j)}\in \sum_{0\le j <\rho} \Gamma^{m-j}_{\rho-j}(\xR^d),\quad 
b=\sum_{0\le k <\rho}b^{(m-k)}\in \sum_{0\le k <\rho} \Gamma^{m'-k}_{\rho-k}(\xR^d),
$$
with $m,m'\in\xR$ and $\rho>0$, then 
$T_a T_b -T_{c}$ is of order $\leo m+m'-\rho$ with
$$
c=
\sum_{ \la \alpha\ra +j+k < \rho} \frac{1}{i^{\la\alpha\ra} \alpha !} \partial_\xi^{\alpha} a^{(m-j)} \partial_{x}^\alpha b^{(m'-k)}.
$$
\end{rema}
\begin{rema}\label{R3.9}
Clearly a paradifferential operator is not invertible ($T_a u=0$ for any function $u$ whose spectrum is included in the ball 
$\la \xi\ra\le 1/2$). However, the previous result implies that there are left and right parametrix for elliptic symbols. 
Namely, assume that $a\in \Gamma^m_\rho$ is an elliptic symbol (such that $\la a \ra \ge K \la \xi\ra^m$ for some $K>0$), 
then there exists $b,b'\in \Gamma^{-m}_\rho$ such that
$$
T_b T_a -I \text{Ê and  } T_a T_{b'}-I \quad \text{are of order }\leo -\rho.
$$
Consequently, 
if $u\in H^s$ and $T_a u \in H^{\mu}$ then $u\in H^{r}$ with $r=\min \{\mu+m, s+\rho\}$.
\end{rema}

\begin{theo}[Adjoint]\label{theo:sc2}
Let $m\in\xR$, $\rho>0$ and $a\in \Gamma^{m}_{\rho}(\xR^d)$. Denote by 
$(T_a)^*$ the adjoint operator of $T_a$ and by $\overline{a}$ the complex-conjugated of $a$. Then 
$(T_a)^* -T_{a^*}$ is of order $\leo m-\rho$ where
$$
a^*=
\sum_{\la \alpha\ra < \rho} \frac{1}{i^{\la\alpha\ra} \alpha !} \partial_\xi^\alpha \partial_x^{\alpha} \overline{a} .
$$
Moreover, for all $\mu$ there exists a constant $K$ such that
\begin{equation}\label{esti:quant3}
\lA (T_a)^*   - T_{a^*}   \rA_{H^{\mu}\rightarrow H^{\mu-m+\rho}}\le 
K M_{\rho}^{m}(a).
\end{equation}
\end{theo}

If $a=a(x)$ is a function of $x$ only, the paradifferential operator $T_a$ is a called a paraproduct. 
It follows from Theorem~\ref{theo:sc} and Theorem~\ref{theo:sc2} that:
\begin{enumerate}[(i)]
\item 
If $a\in H^{\alpha}(\xR^d)$ and $b\in H^{\beta}(\xR^d)$ with $\alpha>\frac{d}{2}$, $\beta>\frac{d}{2}$,
then
\begin{equation}\label{iii}
T_a T_b - T_{a b} \text{ is of order  }\leo - \left(\min \{\alpha,\beta\}-\frac{d}{2}\right).
\end{equation}
\item 
If $a\in H^{\alpha}(\xR^d)$ with $\alpha>\frac{d}{2}$, then
\begin{equation}\label{iv}
(T_a)^* -T_{\overline{a}} \text{ is of order  }\leo - \left(\alpha-\frac{d}{2}\right).
\end{equation}
\end{enumerate}
We also have operator norm estimates in terms of the Sobolev norms of the functions. 

A first nice feature of paraproducts is that they are well defined for functions $a=a(x)$ which are not in $L^\infty$ but merely in some 
Sobolev spaces $H^r$ with $r<d/2$.
\begin{lemm}\label{negmu}
Let $m\ge 0$. If $a\in H^{\frac{d}{2}-m}(\xR^d)$ and $u\in H^\mu(\xR^d)$ then 
$$
T_a u \in H^{\mu-m}(\xR^d).
$$
Moreover, 
$$
\lA T_a u \rA_{H^{\mu-m}}\le K \lA a \rA_{H^{\frac{d}{2}-m}}\lA u\rA_{H^{\mu}},
$$
for some positive constant $K$ independent of $a$ and $u$.
\end{lemm} 

On the other hand, a key feature of paraproducts is that one can replace 
nonlinear expressions by paradifferential expressions, 
to the price of error terms which are smoother than the main terms. 

\begin{theo}\label{lemPa} Let $\alpha,\beta\in\xR$ be such that 
$\alpha>\frac{d}{2}$, $\beta>\frac{d}{2}$, then
\begin{enumerate}[(i)]
\item 
For all $C^\infty$ function $F$, if $a \in H^{\alpha}(\xR^d)$ then 
\begin{equation*}
F( a)-F(0) - T_{F'(a)}a \in H^{2\alpha-\frac{d}{2}}(\xR^d).
\end{equation*}
\item If $a \in H^{\alpha}(\xR^d)$ and 
$b \in H^{\beta}(\xR^d)$, then $ab - T_a b- T_b a \in H^{\alpha + \beta-\frac{d}{2}}(\xR^d)$. Moreover,
\begin{equation*}
\lA ab - T_a b- T_b a \rA _{H^{\alpha + \beta-\frac{d}{2}}(\xR^d)} 
\leq K \lA a \rA _{H^{\alpha}(\xR^d)}\lA b\rA _{H^{\beta}(\xR^d)},
\end{equation*}
for some positive constant $K$ independent of $a$, $b$.
\end{enumerate}
\end{theo}

We also recall the usual nonlinear estimates in Sobolev spaces (see chapter 8 in \cite{Hormander}): 
\begin{itemize}
\item If $u_j\in H^{s_j}(\xR^d)$, $j=1,2$, and $s_1+s_2>0$ then $u_1 u_2\in H^{s_0}(\xR^d)$ and 
\begin{equation}\label{pr}
\lA u_1 u_2 \rA_{H^{s_0}}\le K \lA u_1\rA_{H^{s_1}}\lA u_2\rA_{H^{s_2}},
\end{equation}
if 
$$
s_0\le s_j, \quad j=1,2,\quad\text{and}\quad s_0\le s_1+s_2-d/2,
$$
where the last inequality is strict if $s_1$ or $s_2$ or $-s_0$ is equal to $d/2$.
\item For all $C^\infty$ function $F$ vanishing at the origin, if $u \in H^{s}(\xR^d)$ with $s>d/2$ then
\begin{equation}\label{Fr}
\lA F(u)\rA_{H^s} \le C\left(\lA u\rA_{H^s}\right),
\end{equation}
for some non-decreasing function $C$ depending only on $F$.
\end{itemize}

\subsection{Symbol of the Dirichlet-Neumann operator}

Given $\eta\in C^\infty(\xR^d)$, consider the domain (without bottom)
$$
\Omega=\{ (x,y)\in\xR^d\times \xR\,:\,y<\eta(x)\}.
$$
It is well known that the Dirichlet-Neumann operator associated to $\Omega$ is a classical elliptic 
pseudo-differential
operator of order $1$, 
whose symbol has an asymptotic expansion of the form
\begin{equation*}
\lambda^{(1)}(x,\xi)+\lambda^{(0)}(x,\xi)+\lambda^{(-1)}(x,\xi)+\cdots
\end{equation*}
where $\lambda^{(k)}$ are homogeneous of degree $k$ in $\xi$, and
the principal symbol $\lambda^{(1)}$ and the sub-principal symbol $\lambda^{(0)}$ are given by (cf \cite{IP})
\begin{equation}\label{dmu10}
\begin{aligned}
\lambda^{(1)}&=\sqrt{(1+\la\nabla\eta\ra^2)\la\xi\ra^2-(\nabla \eta\cdot\xi)^2},\\
\lambda^{(0)}&=\frac{1+\la \partialx\eta\ra^2}{2\lambda^{(1)}}
\left\{ \cnx \left( \alpha^{(1)}  \partialx \eta\right) 
+i\partial_\xi \lambda^{(1)} \cdot\partialx \alpha^{(1)} \right\},
\end{aligned}
\end{equation}
with
$$
\alpha^{(1)}= \frac{1}{1+\la  \partialx\eta\ra^2}\left(  \lambda^{(1)}+i\partialx \eta\cdot\xi \right).
$$
The symbols $\lambda^{(-1)},\ldots$ are defined by induction and we can prove that 
$\lambda^{(k)}$ involves only derivatives of $\eta$ of order $\leo \la k\ra+2$. 
 
In our case the function $ \eta$ will not be $C^\infty$ but only at least $C^2$, so we shall set 
\begin{equation}\label{lambda}
\lambda=\lambda^{(1)}+\lambda^{(0)},
\end{equation}
which will be well-defined in the  $C^2$ case.

The following observation contains 
one of the key dichotomy between 2D waves and 3D waves.
\begin{prop}\label{A2D}
If $d=1$ then $\lambda$ simplifies to
$$
\lambda (x,\xi)=\la \xi\ra.
$$
\end{prop}

Also, directly from \eqref{dmu10}, one can check the following formula (which holds for all $d\ge 1$)
\begin{equation}\label{Adlambda}
\IM \lambda^{(0)}=-\mez (\partial_\xi\cdot\partial_x) \lambda^{(1)},
\end{equation}
which reflects the fact that the Dirichlet-Neumann operator is a symmetric operator.

\subsection{Paralinearization of the Dirichlet-Neumann operator}\label{s3.2}
Here is the main result of this section. 
Following the analysis in \cite{AM}, we shall paralinearize the Dirichlet-Neumann operator. 
The main novelties are that we consider the case of finite 
depth (with a general bottom) and that we lower the regularity assumptions.

\begin{prop}
\label{prop:csystem}
Let $d\ge 1$ and $s>2+d/2$. 
Assume that 
$$
(\eta,\psi)\in H^{s+\mez}(\xR^d)\times H^{s}(\xR^d),
$$
and that $\eta$ is such that $\dist (\Sigma,\Gamma)>0$. Then 
\begin{equation}\label{t1}
G(\eta)\psi= T_{\lambda} \bigl(\psi-T_{\mathfrak{B}}\eta \bigr) - T_{V} \cdot\partialx \eta 
+f(\eta,\psi),
\end{equation}
where $\lambda$ is given by \eqref{dmu10} and \eqref{lambda}, 
$$
\mathfrak{B}\defn\frac{\partialx \eta \cdot\partialx \psi+G(\eta) \psi}{1+|\partialx  \eta|^2}, 
\quad 
V\defn \partialx \psi -\mathfrak{B} \partialx\eta,
$$
and 
$f(\eta,\psi)\in H^{s+\frac{1}{2}}(\xR^d)$. Moreover, we have the estimate
$$
\lA f(\eta,\psi)\rA_{H^{s+\frac{1}{2}}}\le 
C\left( \lA \partialx\eta\rA_{H^{s-\mez}}\right)
 \lA \partialx\psi\rA_{H^{s-1}},
$$
for some non-decreasing function $C$ depending only on $\dist (\Sigma,\Gamma)>0$.
\end{prop}
\begin{rema}
It is well known that $\mathfrak{B}$ and $V$ play a key role 
in the study of the water waves (these are simply the projection 
of the velocity field on the vertical and horizontal directions). 
The reason to introduce the unknown $\psi-T_{\mathfrak{B}}\eta$, 
which is related to the so-called 
good unknown of Alinhac (\cite{Ali}), is explained in \cite{AM} (see also~\cite{LannesJAMS,Trakhinin}). 
\end{rema}

\subsection{Proof of Proposition~\ref{prop:csystem}}

Let $v$ be given by \eqref{defiv}. According to~\eqref{dnint1}, $v$ solves 
\begin{equation*}
\alpha \partial_z^2 v +\Delta v + \beta \cdot\partialx\partial_z v  - \gamma \partial_z v=g,
\end{equation*}
where $g\in C^2_z([-1,0];H^{s+\mez}(\xR^d))$ is given by \eqref{evpf} and 
\begin{equation}\label{abc2}
\alpha\defn \frac{1+|\partialx  \eta |^2 }{ h^2},\quad 
\beta\defn  -2 \frac{\partialx \eta}{ h},\quad 
\gamma \defn \frac{\Delta \eta}{ h}  .
\end{equation}
Also $v$ satisfies the boundary conditions 
\begin{equation*}
v\big\arrowvert_{z=0}= \psi, \qquad  v  \arrowvert_{z=-1}=0,\quad 
\partial_z v  \arrowvert_{z=-1}=0.
\end{equation*}

Henceforth we make intensive use of the following notations.
\begin{nota}\label{sim}
$\CZ{r}$ denotes the space of continuous functions 
in $z\in [-1,0]$ with values in $H^r(\xR^d)$.
\end{nota}

It follows from Proposition~\ref{L1} and a 
classical interpolation argument that
\begin{equation*}
(\partialx v,\partial_{z}v)\in \CZ{s-1}.
\end{equation*}
In addition, directly from the equation \eqref{dnint1} and the usual product rule in Sobolev spaces (cf \eqref{pr}), we obtain
$$
\partial_z^2  v\in \CZ{s-2}.
$$

\subsubsection{The good unknown of Alinhac}
Below, we use the tangential  paradifferential calculus, that is the paradifferential
quantization $T_a$  of symbols $a(z;x,\xi)$ depending on the phase space variables
$(x, \xi) \in T^*\xR^{d}$ and the parameter $z \in [-1, 0]$. In particular, 
denote by $T_a u$ the operator acting on functions $u=u(z;x)$ so that for each fixed $z$, 
$(T_a u)(z)=T_{a(z)}u(z)$.

Note that a simple computation shows 
\begin{equation*}
G(\eta)\psi=\frac{1+|\partialx  \rho|^2}{ h}\partial_z v
-\partialx \eta\cdot\partialx  v \Big\arrowvert_{z=0}.
\end{equation*}
Our purpose is to express $\partial_z v\arrowvert _{z=0}$ in terms of tangential derivatives. 
To do this, the key technical point is to obtain an equation for 
$\psi-T_{\mathfrak{B}}\eta$. 

Note that
\begin{equation*}
\psi-T_{\mathfrak{B}}\eta=v - T_{\frac{\partial_z v}{ h}}  \rho \big\arrowvert _{z=0}.
\end{equation*}
We thus introduce 
$$
\mathfrak{b}\defn \frac{\partial_z v}{ h} \quad\text{and}\quad
u\defn  v - T_{\mathfrak{b}}\rho= v - T_{\mathfrak{b}}\eta,
$$
since $T_\mathfrak{b}(hz)=0$, so that $\psi-T_{\mathfrak{B}}\eta=u\arrowvert_{z=0}$.
\begin{lemm}
The good unknown $u = v - T_{\mathfrak{b}} \rho$
satisfies the paradifferential equation
\begin{equation}\label{2m.11}
T_{\alpha}\partial_z^2 u  +\Deltax u +T_{\beta}\cdot \partialx \partial_z u 
- T_{\gamma}\partial_z u =g+f,
\end{equation}
where $\alpha,\beta,\gamma$ are as defined in \eqref{abc2}, $g\in C^1_z(H^{s+\mez}_x)$ is given by \eqref{evpf} and
$$
f\in \CZ{2s-\frac{5+d}{2}}.
$$ 
\end{lemm}
\begin{proof}
We shall use the notation $f_1\sim f_2 $ to say that $f_1-f_2 \in \CZ{2s-\frac{5+d}{2}}$. 

Introduce the operators
$$
E\defn \alpha \partial_z^2 +\Deltax +\beta\cdot \partialx \partial_z     - \gamma \partial_z,
$$
and
\begin{equation*}
P \defn  T_{\alpha}\partial_z^2   +\Deltax + T_{\beta}\cdot \partialx \partial_z - T_{\gamma}\partial_z.
\end{equation*}
We shall prove that $P u\sim g_1$, where $g_1 \in \CZ{s+\mez}$. To do so, we begin with 
the paralinearization formula for products. Recall that
$$
\eta \in H^{s+\mez}(\xR^d)\quad \text{and}\quad \partial_{z}^k v\in \CZ{s-k} \text{ for } k\in\{1,2\}.
$$
According to Theorem~\ref{lemPa}, $ii)$, we have
\begin{equation*}
E v  \sim P v +T_{\partial_{z}^2 v} \alpha + T_{\partialx \partial_z v}\cdot \beta - T_{\partial_z v} \gamma.
\end{equation*}

Since $Ev=g\in \CZ{s+\mez}$ and since $v=u+T_\mathfrak{b}\eta$, this yields
$$
P u + P T_\mathfrak{b}\eta+T_{\partial_{z}^2 v} \alpha + 
T_{\partialx \partial_z v}\cdot \beta - T_{\partial_z v} \gamma \sim g.
$$
Hence, we need only prove that
\begin{equation}\label{3121}
P T_\mathfrak{b}\eta+T_{\partial_{z}^2 v} \alpha + 
T_{\partialx \partial_z v}\cdot \beta - T_{\partial_z v} \gamma \sim g_2\in \CZ{s+\mez}.
\end{equation}

By using the Leibniz rule and \eqref{iii}, we have
\begin{align*}
PT_\mathfrak{b}\eta&\sim 
T_{E \mathfrak{b}}\eta 
+2T_{\partialx \mathfrak{b}}\cdot\partialx \eta
+T_{\beta \partial_z \mathfrak{b}}\cdot\partialx \eta
+T_\mathfrak{b} \Deltax\eta.
\end{align*}
The first key observation is that
$$
E\mathfrak{b}=\frac{\partial_z g} h\in \CZ{s+\mez}.
$$
To establish this identity, note that by definition (cf  \eqref{evpf}) we have
$$
E  \mathfrak{b}= 
\Bigl[\bigl(\frac{1}{ h}\partial_z \bigr)^2 
+ \bigl( \partialx -\frac{\partialx \eta}{ h}\partial_z\bigr)^2 \Bigr]\frac{1}{ h}\partial_z  v= \frac{1}{ h}\partial_z E v,= \frac{1}{ h}\partial_z g .$$
It follows that 
$$ T_{E\mathfrak{b}} \eta \in \CZ{s+\mez}.$$

On the other hand, according to~\eqref{abc2}, we have 
$$ T_{\partial_z v} \gamma = T_\mathfrak{b} \Delta \eta, \quad T_{\nabla\partial _z v} \beta = -2 T_{\nabla \mathfrak{b} } \nabla \eta, $$
$$ T_{\beta \partial _z b} \nabla \eta= - \frac 2 {h^2} T_{\partial ^2 _z v \nabla \eta} \nabla \eta\sim - T_{\partial ^2 _z v} \alpha,$$
where the last equivalence is a consequence of $(i)$ in Theorem~\ref{lemPa} and~\eqref{iii}.

Consequently, we end up with the second key cancelation
\begin{equation*}
 T_{\partial_{z}^2 v} \alpha + T_{\partialx \partial_z v}\cdot \beta - T_{\partial_z v} \gamma +2T_{\partialx b}\cdot\partialx \eta
+T_{\beta  \partial_z \mathfrak{b}}\cdot\partialx \eta  +T_\mathfrak{b} \Deltax\eta\sim g_3 \in \CZ{s+\mez}.
\end{equation*}
This concludes the proof of \eqref{3121} and hence of the lemma.
\end{proof}

\subsubsection{Reduction to the boundary}
Our next task is to perform a decoupling into forward and backward elliptic evolution equations.

\begin{lemm}\label{lemm:total}
Assume that $\eta\in H^{s+\mez}(\xR^d)$. Set
$$
\delta =\min\Bigl\{ \mez , s-2-\frac{d}{2}\Bigr\}>0.
$$
There exist two symbols $a=a(x,\xi)$, $A=A(x,\xi)$ (independent of $z$) with
\begin{align*}
a&=a^{(1)}+a^{(0)}\in \dot\Gamma^{1}_{3/2+\delta}(\xR^d)+\dot\Gamma^{0}_{1/2+\delta}(\xR^d),\\
A&=A^{(1)}+A^{(0)}\in \dot\Gamma^{1}_{3/2+\delta}(\xR^d)+\dot\Gamma^{0}_{1/2+\delta}(\xR^d),
\end{align*}
such that,
\begin{equation}\label{2m2.b}
\begin{aligned}
T_{\alpha}\partial_z^2   +\Deltax  +T_{\beta}\cdot \partialx \partial_z 
- T_{\gamma}\partial_z
=T_{\alpha}( \partial_z - T_ a ) (\partial_z - T_A)u 
+R_0+R_1 \partial_z,
\end{aligned}
\end{equation}
where $R_0$ is of order $\leo 1/2-\delta$ and $R_1$ is of order $\leo -1/2-\delta$.
\end{lemm}
\begin{proof}
We seek $a$ and $A$  such that
\begin{equation}\label{cascade1}
\begin{aligned}
&a^{(1)}A^{(1)}+\frac{1}{i}\partial_\xi a^{(1)}\cdot\partial_x A^{(1)} +a^{(1)}A^{(0)}+a^{(0)}A^{(1)}
=-\frac{\la \xi\ra^2}{\alpha},\\
&a+A=\frac{1}{\alpha}\left(  -i \beta\cdot\xi  + \gamma\right).
\end{aligned}
\end{equation}
According to Theorem~\ref{theo:sc} and \eqref{iii},
$$
R_0\defn T_{\alpha} T_a T_A   - \Delta  \quad\text{is of order }2- \tdm -\delta=\mez-\delta,
$$
while the second equation gives
$$
R_1\defn -T_{\alpha} (T_a + T_A)  
+\big( T_{\beta}\cdot \partialx  - T_{\gamma} \big)
\quad\text{is of order } 1- \tdm -\delta=-\mez-\delta.
$$
We thus obtain the desired result \eqref{2m2.b} from \eqref{2m.11}.

\smallbreak
To solve \eqref{cascade1}, we first solve the principal system:
\begin{align*}
a^{(1)}A^{(1)}&=-\frac{\la \xi\ra^2}{\alpha},\\
a^{(1)}+A^{(1)}&=-\frac{i\beta\cdot\xi}{\alpha},
\end{align*}
by setting
\begin{align*}
a^{(1)} (z, x, \xi) &= \frac{1}{2\alpha}\left( -i  \beta\cdot \xi
-   \sqrt{ 4\alpha \la \xi \ra^2 -  (\beta \cdot \xi)^2}\right),\\
A^{(1)}(z,x, \xi)  &= \frac{1}{2\alpha}\left( -i  \beta\cdot \xi
+   \sqrt{ 4\alpha \la \xi \ra^2 -  (\beta \cdot \xi)^2}\right).
\end{align*}
Directly from the definition of $\alpha$ and $\beta$ note that
$$
\sqrt{ 4\alpha \la \xi \ra^2 -  (\beta \cdot \xi)^2}  \geq  \frac{2}{h}\la\xi\ra,
$$
so that the symbols $a^{(1)},A^{(1)}$ belong to $\dot\Gamma^1_{3/2+\delta}(\xR^d)$ 
(actually $a^{(1)},A^{(1)}$ belong to $\dot\Gamma^1_{s-(d+1)/2}(\xR^d)$ provided that $s-(d+1)/2$ is not an integer).

We next solve the system
\begin{align*}
&a^{(0)}A^{(1)}+a^{(1)}A^{(0)}+
\frac{1}{i}\partial_\xi a^{(1)} \partial_x A^{(1)}=0,\\
&a^{(0)}+A^{(0)}= \frac{\gamma}{\alpha}.
\end{align*}
It is found that
\begin{align*}
a^{(0)}&=\frac{1}{A^{(1)}-a^{(1)}}\left(i\partial_\xi a^{(1)} \cdot\partial_x A^{(1)}
-\frac{\gamma}{\alpha}a^{(1)}\right) ,\\
A^{(0)}&=
\frac{1}{a^{(1)}-A^{(1)}}\left(i\partial_\xi a^{(1)} \cdot\partial_x A^{(1)}
-\frac{\gamma}{\alpha}A^{(1)}\right) ,
\end{align*}
so that the symbols $a^{(0)},A^{(0)}$ belong to $\dot\Gamma^0_{1/2+\delta}(\xR^d)$.  
\end{proof}

We shall need the following elliptic regularity result.

\begin{prop}
Let $a\in \Gamma_{1}^1(\xR^d)$ and $b\in \Gamma^0_{0}(\xR^d)$, 
with the assumption that
$$
\RE a (x,\xi) \geq c \la\xi\ra,
$$
for some positive constant $c$. If $w\in C_{z}^{1}(H_{x}^{-\infty})$ 
solves the elliptic evolution equation
$$
\partial_z w + T_a w =T_b w +f,
$$
with
$f\in \CZ{r}$ for some $r\in\xR$, then
\begin{equation}\label{conclusion}
w(0) \in H^{r+1-\eps}(\xR^d),
\end{equation}
for all $\eps>0$.
\end{prop}
\begin{rema}
This is a local result which means that 
the conclusion \eqref{conclusion} 
remains true if we only assume that, for some $\delta>0$, 
$$
f\arrowvert_{-1 \le z \le -\delta} \in C^{0}([-1,-\delta];H^{-\infty}(\xR^d)),\quad
f\arrowvert_{-\delta \le z \le 0} \in C^{0}([-\delta,0];H^{r}(\xR^d)).
$$
In addition, the result still holds true for symbols depending on $z$, such that 
$a\in C_{z}^0(\Gamma_{1}^1)$ and $b\in C_{z}^0(\Gamma^0_{0})$, 
with the assumption that
$\RE a \geq c \la\xi\ra$, for some positive constant $c$.
\end{rema}
\begin{proof}
The following proof gives the stronger conclusion that $w$ is continuous 
in $z\in ]-1,0]$ with values in 
$H^{r+1-\eps}(\xR^d)$. Therefore, by an elementary induction argument, 
we can assume without loss of generality
that $b=0$ and $w\in \CZ{r}$. In addition one can assume that 
there exists $\delta >0$ such that 
$w(x,z)=0$ for $z\le -1/2$.
\smallbreak
For $z\in [-1,0]$, introduce the symbol
\begin{equation*}
e(z;x,\xi)\defn \exp \left( z a(x,\xi) \right),
\end{equation*}
so that $e\arrowvert_{z=0}=1$ and 
\begin{equation*}
\partial_{z}e=ea.
\end{equation*}
According
to our assumption that $\RE a \geq c\la\xi\ra$, we have the simple estimates
$$
( \la z\ra \la \xi\ra)^m e(z;x,\xi)  \le C_{m}.
$$
Write
$$
\partial_{z}\left(T_{e} w \right)
=T_{e} f  +(T_{\partial_z e}-T_e T_a ) w,
$$
and integrate on $[-1,0]$ to obtain
\begin{equation*}
T_{1}w(0)=
\int_{-1}^{0} (T_{\partial_z e}-T_e T_a ) w (y)\, dy+
\int_{-1}^{0} (T_{e} f)(y)\, dy.
\end{equation*}
Since 
$w(0)-T_{1}w(0)\in H^{+\infty}(\xR^d)$ 
it remains only to prove that the right-hand side belongs to $H^{r+1-\eps}(\xR^d)$. Set 
$$
w_1(0)=\int_{-1}^{0} (T_{\partial_z e}-T_e T_a ) w (y)\, dy,\qquad
w_2(0)= \int_{-1}^{0} (T_{e} f)(y)\, dy.
$$

To prove that $w_2(0)$ belongs to $H^{r+1-\eps}(\xR^d)$,
the key observation is that, since $\RE a \geq c\la\xi\ra$,
the family
$$
\left\{\,  (\la y\ra \la \xi\ra)^{1-\eps} e(y;x,\xi) \,:\, -1 \le y \le 0 \,\right\}
$$
is bounded in~$\Gamma^{0}_{1}(\xR^d)$.
According to the operator norm estimate \eqref{esti:quant1}, 
we thus obtain that there is a constant $K$ such that,
for all $-1\le y\le 0$ and all $v\in H^r(\xR^d)$,
$$
\lA (\la y\ra \la D_x\ra)^{1-\eps} (T_{e} v)\rA_{H^{r}}\le K
\lA v\rA_{H^{r}}.
$$
Consequently, there is a constant $K$ such that, for all $y\in [-1,0[$,
$$
\lA (T_{e} f)(y)\rA_{H^{r+1-\eps}}
\le \frac{K}{\la y\ra^{1-\eps}}\lA f(y)\rA_{H^{r}}.
$$
Since $\la y\ra^{-(1-\eps)} \in L^{1}(]-1,0[)$, this implies that 
$w_2(0)\in H^{r+1-\eps}(\xR^d)$. 

With regards to the first term, we claim that, similarly, 
$$
\lA (T_{\partial_z e}-T_e T_a)(y) \rA _{H^{r}\rightarrow H^{r+1-\eps}}\le \frac{K}{\la y\ra^{1-\eps}}.
$$
Indeed, since $\partial_z e =ea$, this follows from \eqref{esti:quant2} applied with $(m,m',r)=(-1+\eps,1,1)$ and the fact that 
$M^{-1+\eps}_1 ( (\la y\ra ^{1-\eps} e(y;\cdot,\cdot))$ is uniformly bounded for $-1 \le y \le 0 $. 
This yields the desired result.
\end{proof}

We are now in position to describe 
the boundary value of $\partial_z u$ up to an error in
$H^{s+\frac{1}{2}}(\xR^d)$.

\begin{coro}\label{prop:total}
Let $A$ be as given by Lemma~$\ref{lemm:total}$. Then, on the boundary $\{ z = 0 \}$, there holds
\begin{equation*}
(\partial_z u  -    T_A u)\arrowvert_{z=0}  \in H^{s+\frac{1}{2}}(\xR^d).
\end{equation*}
\end{coro} 
\begin{proof}   
Introduce $ w \defn (\partial_z - T_A)u$ and write
\begin{equation*}
\partial_z w  - T_{a^{(1)}} w = T_{a^{(0)}}w+f',
\end{equation*}
with $f'\in \CZ{s-\frac{1}{2}+\delta}$. Since $\RE a^{(1)}<-c \la \xi\ra$, 
the previous proposition applied with $a=-a^{(1)}$, $b=a^{(0)}$ and $\eps=\delta>0$ 
implies that  
$w\arrowvert_{z=0}\in H^{s+\frac{1}{2}}(\xR^d)$.
\end{proof}

By definition
\begin{equation*}
G(\eta)\psi=
\frac{1 + | \partialx \eta |^2}{  h} \partial_z v  - \partialx \eta \cdot\partialx v
\Big\arrowvert_{z=0}.
\end{equation*}
As before, we find that
\begin{align*}
&\frac{1 + | \partialx \eta |^2}{  h} \partial_z v  - \partialx \eta \cdot\partialx v\\
&\quad =
T_{ \frac{1 + | \partialx \eta |^2}{  h}}  \partial_z v  + 2T_{ \mathfrak{b}\partialx \eta}\cdot  \partialx \eta 
- T_{\mathfrak{b} \frac{1 + | \partialx \eta |^2}{  h} }  h\\
&\quad\quad -\left( T_{ \partialx \eta}  \cdot\partialx v + T_ {\partialx v} \cdot\partialx \eta\right)+R,
\end{align*}
where $R\in \CZ{2s-\frac{3+d}{2}}$. 
We next replace $\partial_{z}v$ and $\partialx v$ by 
$\partial_{z}(u+T_{\mathfrak{b}}\rho)$ and $\partialx(u+T_{\mathfrak{b}}\rho)$ in the right hand-side 
to obtain, after a few computations,
\begin{align*}
&\frac{1 + | \partialx \eta |^2}{  h} \partial_z v  - \partialx \eta \cdot\partialx v\\
&\quad =
T_{ \frac{1 + | \partialx \eta |^2}{  h}}  \partial_z u  - T_{ \partialx \eta}  \cdot\partialx u
-T_{\partialx  v-\mathfrak{b}\partialx \eta}\cdot\partialx  \rho -T_{\cnx \left(\partialx  v-\mathfrak{b}\partialx \eta\right)} \rho+R',
\end{align*}
with $R'\in \CZ{2s-\frac{3+d}{2}}$. Furthermore, Corollary~\ref{prop:total} implies that
\begin{equation}\label{IVc2}
T_{\frac{1 + | \partialx \eta |^2}{  h}}  \partial_z u - T_{\partialx \eta}\cdot \partialx  u
\Big\arrowvert_{z=0}
= T_{\lambda}  U +r,
\end{equation}
with $U=u\arrowvert_{z=0}=v-T_{\mathfrak{b}}\rho\arrowvert _{z=0}=\psi-T_\mathfrak{B}\eta$, $r\in H^{s+\frac{1}{2}}(\xR^d)$ and
\begin{equation}\label{defidns}
\lambda=\frac{1 + | \partialx \eta |^2}{  h} A -i \partialx \eta\cdot  \xi\Big\arrowvert_{z=0}.
\end{equation}
After a few computations, we check that $\lambda$ is as given by \eqref{dmu10}--\eqref{lambda}.

This concludes the analysis of the Dirichlet-Neumann operator. Indeed, we have obtained 
$$
G(\eta)\psi =T_{\lambda}U - T_{\partialx  v-\mathfrak{b}\partialx \eta}\cdot\partialx \eta
-T_{\cnx \left(\partialx  v-\mathfrak{b}\partialx \eta\right)}\rho + f(\eta,\psi),
$$
with $f(\eta,\psi) \in  H^{s+\frac{1}{2}}(\xR^d)$. This 
yields the first equation in \eqref{t1} since
$$
V= \partialx  v-\mathfrak{b}\partialx \eta \arrowvert_{z=0},\quad
\partialx \eta\arrowvert_{z=0}=\partialx\eta,
$$
and since 
$$
T_{\cnx V} \eta \in H^{s+\frac{1}{2}}(\xR^d).
$$

\subsection{A simpler case}

Let us remark that if $(\eta, \psi )\in H^{s+ \mez }(\xR^d) \times H^{s-1}( \xR^d)$, 
the expressions above can be simplified and we have the following result that we shall use in Section~\ref{sec.unique}.
\begin{prop}
\label{prop:csystem3}
Let $d\ge 1$, $s>2+d/2$ and $1\le \sigma\le s-1$. 
Assume that 
$$
(\eta,\psi)\in H^{s+\mez}(\xR^d)\times H^{\sigma}(\xR^d),
$$
and that $\eta$ is such that $\dist (\Sigma,\Gamma)>0$. Then 
\begin{equation*}
G(\eta)\psi= T_{\lambda^{(1)}} \psi
+F(\eta,\psi),
\end{equation*}
where $F(\eta,\psi)\in H^{\sigma}(\xR^d)$ (and recall that $\lambda^{(1)}$ denotes the principal symbol ,of the Dirichlet-Neumann operator). 
Moreover, 
$$
\lA F(\eta,\psi)\rA_{H^{\sigma}}\le 
C\left( \lA \partialx\eta\rA_{H^{s-\mez}}\right)
 \lA \partialx\psi\rA_{H^{\sigma-2}},
$$
for some non-decreasing function $C$ depending only on $\dist (\Sigma,\Gamma)>0$.
\end{prop}
\begin{rema}
Notice that the proof below would still work assuming only  
$$
\eta\in H^{s+\eps}(\xR^d),\quad v\in \CZ{\sigma},
$$
with the same conclusion. A more involved proof (using regularized lifting for the function $\eta$ following Lannes~\cite{LannesJAMS}) would give the result assuming only 
$$
(\eta,\psi)\in H^{s}(\xR^d)\times H^{\sigma}(\xR^d).
$$ 
\end{rema}
\begin{proof}
We follow the proof of Proposition~\ref{prop:csystem}. 
Let $v$ be as given by \eqref{defiv}: $v$ solves 
\begin{equation*}
\alpha \partial_z^2 v +\Delta v + \beta \cdot\partialx\partial_z v  - \gamma \partial_z v=g,
\end{equation*}
where $g\in C^0([-1,0];H^{s+\mez}(\xR^d))$ is given by \eqref{evpf} and 
\begin{equation*}
\alpha\defn \frac{(1+|\partialx  \eta |^2 )}{ h^2},\quad 
\beta\defn  -2 \frac{\partialx \eta}{ h},\quad 
\gamma \defn \frac{\Delta \eta}{ h}  .
\end{equation*}
Comparing with the proof of Proposition~\ref{prop:csystem}, an important simplification is that we need only in this proof to paralinearize with respect to $v$. 
In this direction, we claim that 
\begin{equation}\label{2m.11ii} 
T_{\alpha}\partial_z^2 v  +\Deltax v +T_{\beta}\cdot \partialx \partial_z v 
- T_{\gamma}\partial_z v \in \CZ{\sigma-\mez}.
\end{equation}
To see this we first apply point (ii) in Theorem~\ref{lemPa} to obtain
\begin{align*}
&\alpha \partial_z^2 v- T_\alpha \partial_z^2 v- T_{\partial_z^2 v}\alpha \in \CZ{(s-\mez)+\sigma-2-d/2}\subset \CZ{\sigma-\mez},
\intertext{and similarly}
&\beta \cdot\partialx \partial_z v- T_\beta \cdot\partialx\partial_z v - T_{\partialx\partial_z v}\cdot\beta \in \CZ{\sigma-\mez},\\
&\gamma \partial_z v- T_\gamma \partial_z v - T_{\partial_z v}\gamma \in \CZ{\sigma-\mez}.
\end{align*}
Moreover, writing $\sigma-2=d/2-(d/2+2-\sigma)$, using Lemma~\ref{negmu} with $m=d/2+2-\sigma$, we obtain 
$$
T_{\partial_z^2 v}\alpha \in \CZ{s-\mez-(d/2+2-\sigma)}\subset \CZ{\sigma-\mez},
$$
and
$$
T_{\partialx\partial_z v}\cdot\beta \in \CZ{\sigma-\mez}.
$$
Similarly, we have
$$
T_{\partial_z v}\gamma \in \CZ{\sigma-\mez}.
$$
Therefore, summing up directly gives the desired result \eqref{2m.11ii}.

Now, by applying Lemma~\ref{lemm:total}, we obtain that 
$$
T_{\alpha}\partial_z^2   +\Deltax v +T_{\beta}\cdot \partialx \partial_z  v
- T_{\gamma}\partial_z v=
T_{\alpha}( \partial_z - T_{a} ) (\partial_z - T_{A})v 
+f
$$
with $f=R_0 v +R_1 \partial_z v\in \CZ{\sigma-1+\delta}$ where $\delta =\min\left\{ \mez , s-2-\frac{d}{2}\right\}>0$. 

Then, as in Corollary~\ref{prop:total}, we deduce that
\begin{equation*}
(\partial_z v  -    T_{A} v)\arrowvert_{z=0}  \in H^{\sigma}(\xR^d).
\end{equation*}
Since $v(0)\in H^{s-1}(\xR^d)$ we deduce 
$T_{A^{(0)}}v \arrowvert_{z=0}\in H^{s-1}(\xR^d)$ ($A^{(0)}$ is the subprincipal symbol of $A$, which is of order $0$) 
and hence
\begin{equation*} 
(\partial_z v  -    T_{A^{(0)}} v)\arrowvert_{z=0}  \in H^{\sigma}(\xR^d).
\end{equation*}
The rest of the proof is as in the proof of Proposition~\ref{prop:csystem}
\end{proof}

\subsection{Paralinearization of the full system}
Consider a given solution 
$(\eta,\psi)$ of ~\eqref{system} on the time interval $[0,T]$ with $0<T<+\infty$, 
such that 
$$
(\eta,\psi)\in C^0\big([0,T];H^{s+\mez}(\xR^d)\times H^{s}(\xR^d)\big),
$$
for some $s>2+d/2$, with $d\ge 1$.

In the sequel we consider functions of $(t,x)$, considered as 
functions of $t$ with values in various spaces of functions of $x$. In particular, 
denote by $T_a u$ the operator acting on $u$ so that for each fixed $t$, 
$(T_a u)(t)=T_{a(t)}u(t)$. 

Our first result is a paralinearization of the water-waves system \eqref{system}. 
\begin{prop}
\label{prop:csystem2}
Introduce
$$
U\defn \psi-T_{\mathfrak{B}}\eta.
$$
Then $(\eta,U)$ satisfies a system of the form 
\begin{equation}\label{t1bis}
\left\{
\begin{aligned}
&\partial_{t}\eta   +T_V \cdot\partialx \eta -T_{\lambda } U  =f_1,
\\
&\partial_{t}U+  T_V \cdot\partialx U + T_{h}\eta=f_2,
\end{aligned}
\right.
\end{equation}
with
\begin{equation*}
f_{1}\in L^\infty\big(0,T;H^{s+\frac{1}{2}}(\xR^d)\big),\quad 
f_{2}\in L^\infty\big(0,T;H^{s}(\xR^d)\big).
\end{equation*}
Moreover,
$$
\lA (f_{1},f_{2})\rA_{L^{\infty}(0,T;H^{s+\mez}\times H^{s})}\le 
C \left( \lA (\eta,\psi)\rA_{L^\infty(0,T;H^{s+\mez}\times H^{s})}\right),
$$
for some function $C$ depending only on $\dist (\Sigma_0,\Gamma)$.
\end{prop}

We have already performed the paralinearization of the Dirichlet-Neumann operator. 
We now paralinearize the nonlinear terms which appear in the 
dynamic boundary condition. This step is much easier. 

\begin{lemm}\label{paraH}
There holds
$$
H( \eta)=-T_{h}\eta + f,
$$
where $h=h ^{(2)}+h ^{(1)}$ with
\begin{equation}\label{dh21}
\begin{aligned}
&h ^{(2)}= \left(1+|\partialx \eta |^2\right)^{-\frac{1}{2}}
\left(\la \xi\ra^2 - \frac{(\partialx \eta\cdot \xi)^2}{1+|\partialx \eta |^2}\right),\\
&h ^{(1)}=-\frac{i}{2}(\partial_x\cdot\partial_\xi) h ^{(2)}, 
\end{aligned}
\end{equation}
and 
$f\in L^\infty(0,T;H^{2s-2-d/2})$ is such that
\begin{equation}\label{ef}
\lA f\rA_{L^\infty(0,T;H^{2s-2-\frac{d}{2}})} \le C ( \lA \eta\rA_{L^{\infty}(0,T;H^{s+1/2})}),
\end{equation}
for some non-decreasing function $C$.
\end{lemm}
\begin{proof}
Theorem~\ref{lemPa} applied with $\alpha=s-1/2$ implies that
$$
\frac{\partialx\eta}{\sqrt{1 + | \partialx\eta |^2}} 
=T_M \partialx\eta + \widetilde{f}  
$$
where
$$
M= \frac{1}{\sqrt{1 + | \partialx \eta  |^2}}I-\frac{\partialx\eta\otimes\partialx\eta}{(1+|\partialx\eta|^2)^{3/2}},
$$
and $\tilde{f}\in L^\infty(0,T;H^{2s-1-\frac{d}{2}})$ is such that
\begin{equation*}
\lA \widetilde{f}\rA_{L^\infty(0,T;H^{2s-1-\frac{d}{2}})} \le C ( \lA \eta\rA_{L^{\infty}(0,T;H^{s+\mez})}),
\end{equation*}
for some non-decreasing function $C$. Since  
$$
\cnx (T_M \partialx \eta) = T_{- M \xi \cdot \xi+i \cnx M \xi} \eta,
$$
we obtain the desired result with  
$h^{(2)}=  M \xi \cdot \xi$, $h^{(1)}=-i \cnx M \xi$ and $f=\cnx \tilde{f}$.
\end{proof}

Recall the notations
$$
\mathfrak{B}=\frac{\partialx \eta \cdot\partialx \psi+G(\eta) \psi}{1+|\partialx  \eta|^2}, 
\quad 
V=\partialx \psi -\mathfrak{B} \partialx\eta.
$$

\begin{lemm}\label{lemmDBC}
We have
\begin{align*}
&\mez \la \partialx \psi\ra^2 -\mez \frac{\left( \partialx \eta \cdot \partialx \psi + G(\eta)\psi \right)^2}{1+\la \partialx\eta\ra^2} 
\\
&\qquad= T_V\cdot \partialx \psi - T_\mathfrak{B}T_V\cdot \partialx  \eta  -T_{\mathfrak{B}} G(\eta)\psi 
 +f',
\end{align*}
with $f'\in L^{\infty}(0,T;H^{2s-2-\frac{d}{2}}(\xR^d))$ satisfies
$$
\lA f'\rA_{L^\infty(0,T;H^{2s-2-\frac{d}{2}})} \le C \left( \lA (\eta,\psi)\rA_{L^{\infty}(0,T;H^{s+\mez}\times H^s)}\right),
$$
for some non-decreasing function $C$.
\end{lemm}
\begin{proof}
Again, we shall use the paralinearization lemma. Note that for
$$
F (a, b,c) = \mez   \frac{ (a \cdot b+c)^2}{1 + | a |^2} \quad (a\in\xR^d,b\in\xR^d,c\in\xR)
$$
there holds
$$
\partial_ a F  =  \frac{  (a \cdot b+c) }{1 + | a |^2}  \Big( b - \frac{  (a \cdot b+c) }{1 + | a |^2} a \Big), 
~
\partial_b F =   \frac{   (a \cdot b+c)}{1 + | a |^2}  a ,
~
\partial_c F =   \frac{   (a \cdot b+c)}{1 + | a |^2}. 
$$
Using these identities for $a = \partialx \eta$, $b =\partialx \psi$ and $c=G(\eta)\psi$, 
the paralinearization lemma (cf $(i)$ in Theorem~\ref{lemPa})
implies that
$$
\mez \frac{\left( \partialx \eta \cdot \partialx \psi + G(\eta)\psi \right)^2}{1+\la \partialx\eta\ra^2}
=
\left\{ T _{ V \mathfrak{B}} \cdot \partialx \eta + T_{\mathfrak{B}\partialx\eta}\partialx\psi 
+ T_{\mathfrak{B}}G(\eta)\psi \right\} 
+r,
$$
with $r \in L^{\infty}(0,T;H^{2s-2-\frac{d}{2}}(\xR^d))$ satisfies the desired estimate. 
Since $V=\partialx\psi -\mathfrak{B}\partialx\eta$, this yields
$$
\mez \la \partialx \psi\ra^2 -\mez \frac{\left( \partialx \eta \cdot \partialx \psi + G(\eta)\psi \right)^2}{1+\la \partialx\eta\ra^2}
=\left\{T_{V } \cdot\partialx \psi
- T _{ V \mathfrak{B}} \cdot \partialx \eta 
-T_{\mathfrak{B}}G(\eta)\psi \right\} 
+r'
$$
with $r'\in  L^{\infty}(0,T;H^{2s-2-\frac{d}{2}}(\xR^d))$. 
Since by \eqref{iii} 
$$
T _{ \mathfrak{B}V} - T_{\mathfrak{B}}T_{V}\quad\text{is of order }-\left(s -1-\frac{d}{2}\right),
$$
this completes the proof. 
\end{proof}

\begin{lemm}\label{lemm:be0}
There exists a function $C$ such that,
$$
\lA T_{\partial_{t}\mathfrak{B}}^{} \eta\rA _{H^s}\le 
C\left( \lA (\eta,\psi)\rA_{H^{s+\mez}\times H^{s}}\right).
$$
\end{lemm}
\begin{proof}\textbf{a)}
We claim that
\begin{equation}\label{c43}
\lA \partial_{t}\eta\rA_{H^{s-1}}+\lA \partial_{t}\psi\rA_{H^{s-\tdm}}+
\lA \mathfrak{B}\rA_{H^{s-1}}
+\lA V\rA_{H^{s-1}}
\le  C\left( \lA(\eta,\psi)\rA_{H^{s+\mez}\times H^s}\right).
\end{equation}
The proof of this claim is straightforward. Indeed, recall that 
$$
\mathfrak{B}=\frac{\partialx \eta \cdot\partialx \psi+G(\eta)\psi}{1+|\partialx  \eta|^2}.
$$
It follows from Proposition~\ref{estDN} that we have the estimate
$$
\lA G(\eta)\psi\rA_{H^{s-1}}\le  C\left( \lA ( \eta,\psi)\rA_{H^{s+\mez}\times H^s}\right).
$$
Using that $H^{s-1}$ is an algebra since $s-1>d/2$, we thus get the desired estimate for $\mathfrak{B}$. 
This in turn implies that $V= \partialx \psi-\mathfrak{B}\partialx\eta$ satisfies the desired estimate. In addition, 
since $\partial_t \eta=G(\eta)\psi$, this gives the estimate of $\lA \partial_{t}\eta\rA_{H^{s-1}}$. 
To estimate $\partial_t \psi$ we simply write that
$$
\partial_t \psi = F(\partialx\psi,\partialx\eta,\partialx^2\eta),
$$
for some $C^\infty$ function $F$ vanishing at the origin. Consequently, since $s-3/2>d/2$, the usual nonlinear rule in Sobolev space implies that
$$
\lA \partial_{t}\psi\rA_{H^{s-3/2}}\le 
C\left( \lA (\partialx\psi,\partialx\eta,\partialx^2\eta)\rA_{H^{s-3/2}}\right) \le 
C\left( \lA ( \eta,\psi)\rA_{H^{s+\mez}\times H^s}\right).
$$

\smallbreak
\noindent \textbf{b)} We are now in position to estimate $\partial_t \mathfrak{B}$. We claim that
\begin{equation}\label{c44}
\lA \partial_t \mathfrak{B}\rA _{H^{s- \frac 5 2}}\le 
C\left( \lA (\eta,\psi)\rA_{H^{s+\mez}\times H^{s}}\right).
\end{equation}
In view of \eqref{c43} and the product rule \eqref{pr}, the only non 
trivial point is to estimate $\partial_{t}[G(\eta)\psi]$. To do so, we use the identity for the 
shape derivative of the Dirichlet-Neumann (see \S\ref{s.2.3}) to obtain
$$
\partial_{t}[G(\eta)\psi] = 
G(\eta)\left( \partial_t \psi-\mathfrak{B}\partial_{t}\eta \right) - \cnx( V\partial_{t}\eta). 
$$
Therefore \eqref{c43} and  the boundedness of $G(\eta)$ on Sobolev spaces (cf  Proposition~\ref{estDN}) 
imply that 
$$
\lA \partial_{t}[G(\eta)\psi]\rA _{H^{s- \frac 5 2}}\le 
C\left( \lA (\eta,\psi)\rA_{H^{s+\mez}\times H^{s}}\right).
$$
This proves \eqref{c44}.

\smallbreak
\noindent \textbf{c)} Next we use Lemma~\ref{negmu} with $m=1/2$ (which asserts that 
if $a\in H^{\frac{d}{2}-\mez}(\xR^d)$ then the paraproduct $T_{a}$ is of order $\leo 1/2$). 
Therefore, since by assumption $s-5/2>d/2-1/2$ for all $d\ge 1$, we conclude
$$
\lA T_{\partial_{t}\mathfrak{B}}^{}\eta\rA _{H^s}\le 
 \lA T_{\partial_{t}\mathfrak{B}}\rA _{H^{s+\mez}\rightarrow H^s}
\lA \eta\rA _{H^{s+\mez}}\le 
C\left( \lA (\eta,\psi)\rA_{H^{s+\mez}\times H^{s}}\right).
$$
This completes the proof.
\end{proof}

\section{Symmetrization}\label{Sym}
Consider a given solution 
$(\eta,\psi)$ of ~\eqref{system} on the time interval $[0,T]$ with $0<T<+\infty$, 
such that 
$$
(\eta,\psi)\in C^0\big([0,T];H^{s+\mez}(\xR^d)\times H^{s}(\xR^d)\big),
$$
for some $s>2+d/2$, with $d\ge 1$. We proved in Proposition~\ref{prop:csystem2} that $\eta$ and $U=\psi-T_\mathfrak{B} \eta$ 
satisfy the system 
\begin{equation}\label{t10}
\left(\partial_{t}   +T_V \cdot\partialx\right)\begin{pmatrix}  \eta \\ U \end{pmatrix} 
+ \begin{pmatrix} 0 & -T_\lambda \\ T_h & 0 \end{pmatrix} \begin{pmatrix}  \eta \\ U \end{pmatrix} =f,
\end{equation}
where $f\in L^{\infty}(0,T;H^{s+\mez}(\xR^d)\times H^s(\xR^d))$. 
The main result of this section is that there exists a symmetrizer $S$ of the form  
$$
S=\begin{pmatrix} T_p & 0 \\ 0 & T_q \end{pmatrix},
$$
which conjugates $\left(\begin{smallmatrix} 0 &-T_\lambda  \\  T_h & 0 \end{smallmatrix}\right)$ 
to a skew-symmetric operator. Indeed we shall prove that there exists $S$ such that, modulo admissible remainders, 
$$
S 
\begin{pmatrix} 0 &-T_\lambda  \\  T_h & 0 \end{pmatrix} 
\simeq \begin{pmatrix} 0 & -T_\gamma \\ (T_\gamma)^* & 0 \end{pmatrix} 
S.
$$
In addition, we shall obtain that the new unknown 
$$
\Phi=S \begin{pmatrix}  \eta \\ U \end{pmatrix}
$$
satisfies a system of the form
\begin{equation}\label{t11}
\partial_{t}\Phi   +T_V \cdot\partialx \Phi
+\begin{pmatrix} 0 & -T_\gamma \\ (T_\gamma)^* & 0 \end{pmatrix}\Phi =F,
\end{equation}
with $F\in L^{\infty}(0,T;H^{s}(\xR^d)\times H^s(\xR^d))$; morevoer 
$\lA (\eta,\psi)\rA_{H^{s+\mez}\times H^s}$ is controlled by means 
of~$\lA \Phi\rA_{H^s}$. 

This symmetrization has many consequences. 
In particular, in the following sections, we shall deduce our two main results from this symmetrization.

\subsection{Symbolic calculus with low regularity}
All the symbols which we consider below are of the form
$$
a=a^{(m)}+a^{(m-1)}
$$ 
where
\begin{enumerate}[(i)]
\item $a^{(m)}$ is a real-valued  elliptic symbol, homogenous of degree $m$ in $\xi$ and depends only on the first order-derivatives of $\eta$;
\item $a^{(m-1)}$ is homogenous of degree $m-1$ in $\xi$ and depends also, but only linearly, on the second order-derivatives of $\eta$.
\end{enumerate}
Recall that in this section 
$\eta\in C^{0}([0,T];H^{s+\mez}(\xR^d))$ is a fixed given function.


\begin{defi}\label{defiSigma}
Given $m\in \xR$, $\Sigma^m$ denotes the class of symbols $a$ of the form
$$
a=a^{(m)}+a^{(m-1)}
$$
with
$$
a^{(m)}(t,x,\xi)= F(\partialx\eta(t,x),\xi), \quad 
a^{(m-1)}(t,x,\xi)=\sum_{\la \alpha\ra=2}G_\alpha(\partialx\eta(t,x),\xi)\partial_x^\alpha \eta (t,x),
$$
such that
\begin{enumerate}[(i)]
\item $T_a$ maps real-valued functions to real valued functions;
\item $F$ is a $C^\infty$ real-valued function of $(\zeta,\xi)\in \xR^d\times (\xR^d\setminus 0)$, 
homogeneous of order $m$  in $\xi$; and such that there exists a continuous function $K=K(\zeta)>0$ such that 
$$
F(\zeta,\xi)\ge K (\zeta)\la \xi\ra^m,
$$
for all $(\zeta,\xi)\in \xR^d \times  (\xR^d\setminus 0)$;
\item$G_\alpha$ is a $C^\infty$ complex-valued function of $(\zeta,\xi)\in \xR^d\times (\xR^d\setminus 0)$, 
homogeneous of order $m-1$ in $\xi$. 
\end{enumerate}
\end{defi}

Notice that, as we only assume $s>2 +d/2$, some technical difficulties appear. 
To overcome these problems, the observation that for all our symbols, the sub-principal terms 
have only a linear dependence on the second order derivative of $\eta$ will play a crucial role. 

Our first result contains the important observation that the previous class of symbols is stable 
by the standard rules of symbolic calculus (this explains why all the symbols which we shall introduce below 
are of this form). We shall state a symbolic calculus result modulo admissible remainders. 
To clarify the meaning of admissible remainder, we introduce the following notation.
\begin{defi}\label{defisym}
Let $m\in\xR$ and consider two families of operators order $\leo m$, 
$$
\{ A(t) \,:\, t\in [0,T]\},\quad 
\{ B(t) \,:\, t\in [0,T]\}.
$$
We shall say that $A\sim B$ if $A-B$ is of order $m-3/2$ (see Definition~\ref{defi:order}) and satisfies the following estimate: for all $\mu\in\xR$, 
there exists a continuous function $C$ such that
$$
\lA A(t)-B(t)\rA_{H^{\mu}\rightarrow H^{\mu-m+\tdm}}\le C\left( \lA \eta(t)\rA_{H^{s+\mez}}\right),
$$
for all $t\in [0,T]$.
\end{defi}

\begin{prop}\label{pcs}
Let $m,m'\in\xR$. Then
\begin{enumerate}
\item If $a\in \Sigma^m$ and $b\in \Sigma^{m'}$ then 
$T_a T_b\sim T_{a\sharp b}$ where $a\sharp b\in \Sigma^{m+m'}$ is given by 
$$
a\sharp b =a^{(m)}b^{(m')}+a^{(m-1)}b^{(m')}+a^{(m)}b^{(m'-1)}+\frac{1}{i}
\partial_\xi a^{(m)}\cdot \partial_x b^{(m')}.
$$
\item If $a\in \Sigma^m$ then $(T_a)^* \sim T_b$ where $b\in \Sigma^m$ is given by
$$
b=a^{(m)}+\overline{a^{(m-1)}} + \frac{1}{i}(\partial_x\cdot \partial_\xi) a^{(m)}.
$$
\end{enumerate}
\end{prop}
\begin{proof}It follows from \eqref{esti:quant2} applied with $\rho=3/2$ that 
$$
\lA T_{a^{(m)}}T_{b^{(m')}}-T_{a^{(m)}b^{(m')}+\frac{1}{i}
\partial_\xi a^{(m)}\cdot \partial_x b^{(m')}}\rA_{H^{\mu}\rightarrow H^{\mu-m-m'+3/2}} \le 
C( \lA \partialx\eta\rA_{W^{3/2,\infty}}).
$$
On the other hand, \eqref{esti:quant2} applied with $\rho=1/2$ implies that
\begin{align*}
&\lA T_{a^{(m)}}T_{b^{(m'-1)}}-T_{a^{(m)}b^{(m'-1)}}\rA_{H^{\mu}\rightarrow H^{\mu-m-m'+3/2}} \le 
C( \lA \partialx\eta\rA_{W^{3/2,\infty}}), \\
&\lA T_{a^{(m-1)}}T_{b^{(m')}}-T_{a^{(m-1)}b^{(m')}}\rA_{H^{\mu}\rightarrow H^{\mu-m-m'+3/2}} \le 
C( \lA \partialx\eta\rA_{W^{3/2,\infty}}).
\end{align*}
Eventually \eqref{esti:quant1} implies that
$$
\lA T_{a^{(m-1)}}T_{b^{(m'-1)}}\rA_{H^{\mu}\rightarrow H^{\mu-m-m'+2}} \le 
C( \lA \partialx\eta\rA_{W^{1,\infty}}).
$$
The first point in the proposition then follows from the Sobolev embedding $H^{s+\mez}(\xR^d)\subset W^{\frac{5}{2},\infty}(\xR^d)$. 
Furthermore, we easily verify that $a\sharp b\in\Sigma^{m+m'}$.

Similarly, the second point is a straightforward consequence of Theorem~\ref{theo:sc2} and the fact that 
$a^{(m)}$ is, by assumption, a real-valued symbol.
\end{proof}

Given that $a\in \Sigma^{m}$, since $a^{(m-1)}$ involves two derivatives of 
$\eta$, the usual boundedness result for paradifferential operators and the embedding 
$H^{s}(\xR^d)\subset W^{2,\infty}(\xR^d)$ implies that we have estimates of the form
\begin{equation}\label{ong}
\lA T_{a(t)} \rA_{H^{\mu}\rightarrow H^{\mu-m}}\les \sup_{\la \xi\ra=1}\lA a(t,\cdot,\xi)\rA_{L^\infty}\le C\left( \lA \eta(t)\rA_{H^s}\right). 
\end{equation}
Our second observation concerning the class $\Sigma^m$ is that one can prove a continuity result which requires 
only an estimate of $\lA \eta\rA_{H^{s-1}}$.

\begin{prop}\label{2d21}
Let $m\in\xR$ and $\mu\in \xR$. Then 
there exists a function $C$ such that for all symbol $a\in \Sigma^m$ and all $t\in [0,T]$,
$$
\lA T_{a(t)} u \rA_{H^{\mu-m}} \le C( \lA\eta(t)\rA_{H^{s-1}}) \lA u \rA_{H^{\mu}}.
$$
\end{prop}
\begin{rema}This result is obvious for $s>3+d/2$ since the $L^\infty$-norm of $a(t,\cdot,\xi)$ is controlled by 
$\lA \eta(t)\rA_{H^{s-1}}$ in this case. 
As alluded to above, this proposition solves the technical difficulty which appears since we only assume $s>2+d/2$.
\end{rema}
\begin{proof}
By abuse of notations, we omit the dependence in time.

\textbf{a)} Consider a symbol $p=p(x,\xi)$ homogeneous of degree $r$ in $\xi$ such that 
$$
x\mapsto \partial_\xi^\alpha p(\cdot,\xi) \quad\text{ belongs to }H^{s-3}(\xR^d) \quad \forall\alpha\in \xN^d.
$$
Let $q$ be defined by 
$$
\widehat{q}(\theta,\xi) = \frac{\chi_1(\theta,\xi)\psi_1(\xi)}{\la\xi\ra} \widehat{p}(\theta,\xi)
$$
where $\chi_1=1$ on $\supp\chi$, $\psi_1=1$ on $\supp\psi$ (see \eqref{eq.para}), 
$\psi_1(\xi)=0$ for $\la\xi\ra\le \frac{1}{3}$, 
$\chi_1(\theta,\xi)=0$ for $\la\theta\ra \ge \la\xi\ra$ and 
$\widehat{f}(\theta,\xi)=\int e^{-ix\cdot\theta}f(x,\xi)\, dx$. Then
\begin{equation}\label{nom0}
T_q \la D_x\ra =T_p,
\end{equation}
and
$$
\la \partial_\xi^\alpha \widehat{q}(\theta,\xi) \ra 
\les \L{\theta}^{-1}\sum_{\beta \le \alpha}\la \partial_\xi^\beta \widehat{p}(\theta,\xi) \ra.
$$
Therefore we have
\begin{equation}\label{nom1}
\lA \partial_\xi^\alpha q(\cdot,\xi)\rA_{H^{s-2}}\les 
\sum_{\beta\le \alpha}\lA \partial_\xi^\beta p(\cdot,\xi)\rA_{H^{s-3}}.
\end{equation}
Now, it follows from the above estimate and the 
embedding $H^{s-2}(\xR^d)\subset L^\infty(\xR^d)$ that 
$q$ is $L^\infty$ in $x$ and hence 
 $q\in \Gamma^{r-1}_0\subset \Gamma^r_0$. Then,  
according to \eqref{esti:quant1} applied with $m=r$ (and not $m=r-1$), we have for all $\sigma\in\xR$,
$$
\lA T_q v\rA_{H^{\sigma-r}}\les \sup_{\la\alpha\ra\le \frac{d}{2}+1} 
\sup_{\la \xi\ra\ge \mez} \la\xi\ra^{\la\alpha\ra-r} \lA  \partial_\xi^\alpha q(\cdot,\xi)\rA_{L^{\infty}}
\lA v\rA_{H^\sigma}.
$$
Applying this inequality with $v=\la D_x\ra u$, $\sigma=\mu-1$ 
and using again the Sobolev embedding and \eqref{nom0}, 
\eqref{nom1}, we obtain
\begin{equation}\label{esti:quants}
\lA T_p u\rA_{H^{\mu-r-1}}\les  \sup_{\la\alpha\ra\le \frac{d}{2}+1} 
\sup_{\la \xi\ra =1} \lA \partial_\xi^\alpha p(\cdot,\xi)\rA_{H^{s-3}} \lA 
u\rA_{H^\mu}.
\end{equation}

\smallbreak
\noindent \textbf{b)} Consider a symbol $a\in\Sigma^m$ of the form
\begin{equation}
\label{eq.symbol}
a=a^{(m)}+a^{(m-1)}= F(\partialx\eta,\xi) +\sum_{\la \alpha\ra=2}G_\alpha(\partialx\eta,\xi)\partial_x^\alpha \eta.
\end{equation}
Up to substracting the symbol of a Fourier multiplier of order $m$, 
we can assume without loss of generality that $F(0,\xi)=0$. 

It follows from the previous estimates that
$$
\lA T_{a^{(m)}} u  \rA_{H^{\mu-m}} \les \sup_{\la \xi\ra=1 }\| a^{(m)}(\cdot,\xi)\|_{H^{s-2}} \lA u \rA_{H^{\mu}},
$$
and
$$
\lA T_{a^{(m-1)}} u  \rA_{H^{\mu-m}} \les \sup_{\la \xi\ra=1 } \| a^{(m-1)}(\cdot,\xi)\|_{H^{s-3}} \lA u \rA_{H^{\mu}}.
$$
Now since $s>2+d/2$ it follows from the usual nonlinear estimates in Sobolev spaces (see \eqref{Fr}) that
$$
 \sup_{\la \xi\ra=1 } \| a^{(m)}(\cdot,\xi)\|_{H^{s-2}} =  \sup_{\la \xi\ra=1 } \| F(\partialx\eta ,\xi) \|_{H^{s-2}} \le C\left(\lA \eta\rA_{H^{s-1}}\right).
$$
On the other hand, by using the product rule \eqref{pr} with $(s_0,s_1,s_2)=(s-3,s-2,s-3)$ we obtain
\begin{multline*}
 \| a^{(m-1)}(\cdot,\xi)\|_{H^{s-3}}\le  \sum_{\la\alpha\ra=2} \lA G_\alpha(\partialx\eta,\xi)\partial_x^\alpha \eta\rA_{H^{s-3}} \\
 \les \Bigl( \la G_\alpha(0,\xi)\ra + 
 \sum_{\la\alpha\ra=2}\lA G_\alpha(\partialx\eta,\xi)-G_\alpha(0,\xi)\rA_{H^{s-2}}\Bigr)\lA \partial_x^\alpha \eta\rA_{H^{s-3}},
 \end{multline*}
for all $\la\xi\ra\le 1$. Therefore, \eqref{Fr} implies that
$$
 \| a^{(m-1)}(\cdot,\xi)\|_{H^{s-3}}\le C(\lA \eta\rA_{H^{s-1}}).
$$
This completes the proof.
\end{proof}

Similarly we have the following result about elliptic regularity 
where one controls the various constants by the $H^{s-1}$-norm of $\eta$ only.
\begin{prop}\label{2d22}
Let $m\in\xR$ and $\mu\in \xR$. Then 
there exists a function $C$ such that for all 
$a\in \Sigma^m$ 
and all $t\in [0,T]$, we have
$$
\lA u \rA_{H^{\mu+m}}\le  C(\lA \eta(t)\rA_{H^{s-1}}) \left\{ \lA T_{a(t)} u\rA_{H^\mu}+  \lA u\rA_{L^2}\right\}.
$$
\end{prop}
\begin{rema}
As mentionned in Remark~\ref{R3.9}, the classical result is that, for all elliptic symbol $a\in \Gamma^{m}_\rho(\xR^d)$ with $\rho>0$, there holds
$$
\lA f \rA_{H^m}\le K \left\{ \lA T_a f \rA_{L^2}+ \lA f\rA_{L^2}\right\},
$$
where $K$ depends only on $M^{m}_\rho(a)$. Hence, if we use the natural estimate 
$$
M^{m-1}_\rho(a^{(m-1)}(t))\le C(\lA \eta(t)\rA_{W^{2+\rho}})\le C(\lA \eta(t)\rA_{H^{s}})
$$ 
for $\rho>0$ small enough, then we obtain an estimate which is worse than the one just stated 
for $2+d/2<s<3+d/2$.
\end{rema}
\begin{proof}
Again, by abuse of notations, we omit the dependence in time. 

Introduce $b=1/a^{(m)}$ and consider $\eps$ such that
$$
0< \eps<\min\{s-2-d/2,1\}.
$$
By applying \eqref{esti:quant2} with $\rho=\eps$ we find that 
$T_{b} T_{a^{(m)}} = I +r$ where $r$ is of order $-\eps$ and satisfies 
$$
\lA r u \rA_{H^{\mu+\eps}}\le C(\lA \partialx \eta\rA_{W^{\eps,\infty}}) \lA  u\rA_{H^{\mu}} 
\le C(\lA \eta\rA_{H^{s-1}})  \lA  u\rA_{H^\mu}.
$$
Then
$$
u  = T_b T_{a} u -r u - T_b T_{a^{(m-1)}}.
$$
Set
$$
R=-r-T_b T_{a^{(m-1)}}.
$$
Then 
$$
(I-R)u=T_b T_a u.
$$
We claim that there exists a function $C$ such that
$$
\lA T_{a^{(m-1)}} u  \rA_{H^{\mu-m+\eps}} \le C( \lA  \eta\rA_{H^{s-1}}) \lA u \rA_{H^{\mu}}.
$$
To see this, notice that the previous proof applies with the decomposition $T_p = T_{q} \la D_x\ra^{1-\eps}$ where 
$$
\widehat{q}(\theta,\xi) = \frac{\chi_1(\theta,\xi)\psi_1(\xi)}{\la\xi\ra^{1-\eps}} \widehat{p}(\theta,\xi).
$$
Once this claim is granted, since $T_b$ is of order $-m$, we find that $R$ satisfies
$$
\lA R u \rA_{H^{\mu+\eps}}\le C( \lA  \eta\rA_{H^{s-1}}) \lA  u\rA_{H^{\mu}}.
$$
Writing 
$$
(I+R+\cdots +R^N)(I-R)u =
(I+R+\cdots +R^N) T_b T_{a}u
$$
we get
$$
u = (I+R+\cdots +R^N) T_b T_{a} u + R^{N+1} u.
$$
The first term in the right hand side is estimated by means of the obvious inequality
\begin{multline*}
\lA (I+R+\cdots +R^N) T_b  \rA_{H^{\mu}\rightarrow H^{\mu+m}}\\
\le 
\lA (I+R+\cdots +R^N)  \rA_{H^{\mu+m}\rightarrow H^{\mu+m}}\lA  T_b \rA_{H^{\mu}\rightarrow H^{\mu+m}},
\end{multline*}
so that
$$
\lA (I+R+\cdots +R^N) T_b T_{a} u \rA_{H^{\mu+m}}\le C(\lA \eta\rA_{H^{s-1}})  \lA  T_{a} u\rA_{H^\mu}.
$$
Choosing $N$ so large that $(N+1)\eps > \mu +m$, we obtain that
$$
\lA R^{N+1} \rA_{H^\mu\rightarrow H^{\mu+m}}\les \lA R\rA_{H^{\mu+m-\eps}\rightarrow H^{\mu+m}}\cdots 
\lA R\rA_{H^{\mu}\rightarrow H^{\mu+\eps}}\le C( \lA  \eta\rA_{H^{s-1}}),
$$
which yields the desired estimate for the second term.
\end{proof}

\subsection{Symmetrization}
The main result of this section is that one can symmetrize the equations. 
Namely, we shall prove that there exist three symbols $p,q,\gamma$ such that 
\begin{equation}\label{sy}
\left\{
\begin{aligned}
&T_p T_\lambda \sim T_\gamma T_q ,\\
&T_{q} T_{h} \sim T_{\gamma} T_{p} ,\\
&T_\gamma \sim (T_{\gamma})^*,
\end{aligned}
\right.
\end{equation}
where recall that the notation $A\sim B$ was introduced in Definition~\ref{defisym}.

We want to explain how we find $p,q,\gamma$ 
by a systematic method. We first observe that if \eqref{sy} holds true then $\gamma$ is of order $3/2$. 
To be definite, we chose $q$ of order $0$, and then necessarily $p$ is of order $1/2$. 
Therefore 
we seek $p,q,\gamma$ under the form
\begin{equation}\label{12-12}
p=p^{(1/2)}+p^{(-1/2)},\quad q=q^{(0)}+q^{(-1)},\quad 
\gamma=\gamma^{(3/2)}
+\gamma^{(1/2)},
\end{equation}
where $a^{(m)}$ is a symbol homogeneous in $\xi$ of order $m\in \xR$.

Let us list some necessary contrainsts on these symbols. Firstly, we seek real elliptic symbols such that, 
$$
p^{(1/2)}\ge K \la \xi\ra^{1/2}, \quad q^{(0)} \ge K ,\quad \gamma^{(3/2)} \ge K \la \xi\ra^{3/2},
$$
for some positive constant 
$K$. Secondly, in order for $T_p,T_q,T_\gamma$ to map real valued functions to real valued functions, we must have
\begin{equation}\label{cr}
\overline{p(t,x,\xi)}=p(t,x,-\xi),\quad \overline{q(t,x,\xi)}=q(t,x,\xi),\quad \overline{\gamma(t,x,\xi)}=\gamma(t,x,-\xi).
\end{equation}
According to Proposition~\ref{pcs}, in order for $T_\gamma$ to satisfy the last identity in \eqref{sy}, 
$\gamma^{(1/2)}$ must satisfy 
\begin{equation}\label{g12}
\IM \gamma^{(1/2)} = -\mez (\partial_\xi\cdot\partial_x ) \gamma^{(3/2)}.
\end{equation}
Our strategy is then to seek $q$ and $\gamma$ such that
\begin{equation}\label{qhl}
T_{q} T_{h}T_\lambda \sim T_{\gamma} T_\gamma T_{q} .
\end{equation}
The idea is that if this identity is satisfied then the first two equations in \eqref{sy} are compatible; this means that 
if any of these two equations is satisfied, then the second one is automaticaly satisfied. 
Therefore, once $q$ and $\gamma$ are so chosen that \eqref{qhl} is satisfied, 
then one can define $p$ by solving either one of the first two equations. The latter task being immediate. 

Recall that the symbol $\lambda=\lambda^{(1)}+\lambda^{(0)}$ (resp.\ $h=h^{(2)}+h^{(1)}$)  is defined by $\eqref{dmu10}$ (resp.\ \eqref{dh21}). In particular, by notation, 
\begin{equation}\label{dnsh}
\begin{aligned}
\lambda^{(1)}&=\sqrt{(1+\la\nabla\eta\ra^2)\la\xi\ra^2-(\nabla \eta\cdot\xi)^2}, \\
h ^{(2)}&= \left(1+|\partialx \eta |^2\right)^{-\frac{1}{2}}
\left(\la \xi\ra^2 - \frac{(\partialx \eta\cdot \xi)^2}{1+|\partialx \eta |^2}\right).
\end{aligned}
\end{equation}
Introduce the notations
$$
h\sharp \lambda = h^{(2)}\lambda^{(1)} + h^{(1)}\lambda^{(1)} +h^{(2)}\lambda^{(0)} +
\frac{1}{i}\partial_\xi h^{(2)}\cdot\partial_x\lambda^{(1)},
$$
and
$$
\gamma\sharp \gamma =  \left(\gamma^{(3/2)}\right)^2 + 2  \gamma^{(1/2)}\gamma^{(3/2)} +
\frac{1}{i}\partial_\xi \gamma^{(3/2)}\cdot\partial_x\gamma^{(3/2)}.
$$
By symbolic calculus, to solve \eqref{qhl}, it is enough to find $q$ and $\gamma$ such that
\begin{multline}\label{symbp}
q ^{(0)} (h\sharp \lambda)  + q ^{(-1)}h^{(2)}\lambda^{(1)}+\frac{1}{i}\partial_\xi q^{(0)}\cdot\partial_x ( h^{(2)} \lambda^{(1)}) \\
= (\gamma\sharp \gamma ) q^{(0)} + \left(\gamma^{(3/2)}\right)^2 q^{(-1)} + \frac{1}{i}\partial_\xi ( \gamma^{(3/2)}\gamma^{(3/2)} ) \cdot\partial_x q^{(0)}.
\end{multline}
We set  
$$
\gamma^{(3/2)}= \sqrt{h^{(2)}\lambda^{(1)}},
$$
so that the leading symbols of both sides of \eqref{symbp} are equal. 
Then $\IM \gamma^{(1/2)}$ has to be fixed by means of \eqref{g12}. We set
$$
\IM \gamma^{(1/2)}=-\frac{1}{2}(\partial_x\cdot\partial_\xi)\gamma^{(3/2)}.
$$
It next remains only to determine $q^{(0)},q^{(-1)}$ and $\RE \gamma^{(1/2)}$ such that
\begin{equation}\label{tauq}
\tau q^{(0)} = \frac{1}{i} \left\{ h^{(2)}\lambda^{(1)} , q^{(0)}\right\}= 
\frac{1}{i}\partial_\xi (h^{(2)}\lambda^{(1)}) \cdot\partial_x q^{(0)} -\frac{1}{i}\partial_\xi q^{(0)}\cdot\partial_x ( h^{(2)} \lambda^{(1)}) ,
\end{equation}
where
$$
\tau = \frac{1}{i}\partial_\xi h^{(2)}\cdot\partial_x\lambda^{(1)}+h^{(1)}\lambda^{(1)} +h^{(2)}\lambda^{(0)} -2  \gamma^{(1/2)}\gamma^{(3/2)} 
+i \partial_\xi \gamma^{(3/2)}\cdot\partial_x\gamma^{(3/2)}.
$$
Since $q^{(-1)}$ does not appear in this equation, one can freely set $q^{(-1)}=0$. We next take the real part of the right-hand side of \eqref{tauq}. 
Since $q^{(0)},h^{(2)},\lambda^{(1)}$ are real-valued symbol we find $\RE \tau =0$. Since $h^{(1)}(t,x,\xi)\in i \xR$, we deduce that 
$\RE  \gamma^{(1/2)}$ must be given by solving the equation
$$
h^{(2)}\RE \lambda^{(0)}= 2\gamma^{(3/2)} \RE  \gamma^{(1/2)},
$$
that is
$$
\RE  \gamma^{(1/2)} = \frac{h^{(2)}\RE \lambda^{(0)}}{2 \gamma^{(3/2)}}=\sqrt{\frac{h^{(2)}}{\lambda^{(1)}}}\frac{ \RE \lambda^{(0)}}{2 }.
$$
It remains only to define $q^{(0)}$ such that 
\begin{equation}\label{mustq}
q^{(0)} \IM \tau =- \left\{ h^{(2)}\lambda^{(1)} , q^{(0)}\right\}.
\end{equation}
Since
$$
h^{(1)}=-\frac{i}{2} (\partial_x\cdot\partial_\xi)h^{(2)},\, 
\IM \lambda^{(0)}=-\frac{1}{2} (\partial_x\cdot\partial_\xi)\lambda^{(1)},\,
\IM \gamma^{(1/2)}=-\frac{1}{2} (\partial_x\cdot\partial_\xi)\gamma^{(3/2)},
$$
we find
\begin{align*}
\IM \tau 
&=-\partial_\xi h^{(2)}\cdot\partial_x\lambda^{(1)}
-\mez \lambda^{(1)}(\partial_\xi\cdot\partial_x ) h^{(2)} -\mez h^{(2)} (\partial_\xi\cdot\partial_x ) \lambda^{(1)}\\
&\quad +\gamma^{(3/2)} (\partial_\xi\cdot\partial_x )\gamma^{(3/2)}+\partial_\xi \gamma^{(3/2)}\cdot\partial_x\gamma^{(3/2)}.
\end{align*}
Writing
$$
\gamma^{(3/2)} (\partial_\xi\cdot\partial_x )\gamma^{(3/2)}+\partial_\xi \gamma^{(3/2)}\cdot\partial_x\gamma^{(3/2)}
=\mez \partial_x \cdot\partial_\xi \left(\gamma^{(3/2)}\right)^2
=\mez \partial_x \cdot\partial_\xi (h^{(2)}\lambda^{(1)}),
$$
we thus obtain
$$
\IM \tau 
=\mez 
\partial_\xi \lambda^{(1)} \cdot \partial_x h^{(2)}-\mez \partial_\xi h^{(2)} \cdot \partial_x \lambda^{(1)},
$$
and hence \eqref{mustq} simplifies to 
\begin{equation}\label{fre}
\mez  \left\{ h^{(2)},\lambda^{(1)} \right\} q^{(0)}+ \left\{ h^{(2)}\lambda^{(1)} , q^{(0)}\right\}=0.
\end{equation}
The key observation is the following relation between $h^{(2)}$ and $\lambda^{(1)}$ (see \eqref{dnsh}): 
$$
h ^{(2)}=  \left( c \lambda^{(1)}\right)^2
\quad\text{with}\quad c= \left(1+|\partialx \eta |^2\right)^{-\frac{3}{4}}.
$$
Consequently \eqref{fre} reduces to
$$
- q^{(0)} (\lambda^{(1)})^2 \partial_x c^2 \cdot\partial_\xi \lambda^{(1)} 
+ 3 c^{2} (\lambda^{(1)})^2 \partial_\xi \lambda^{(1)}\cdot\partial_x q^{(0)} - 
\partial_\xi q^{(0)} \cdot \partial_x \left( c^2 ( \lambda^{(1)})^3\right)=0.
$$
Seeking a solution $q^{(0)}$ which does not depend on $\xi$, we are led to solve
$$
\frac{\partial_\xi \lambda^{(1)}\cdot\partial_x q^{(0)}}{q^{(0)}} = \frac{1}{3} 
\frac{\partial_\xi \lambda^{(1)} \cdot\partial_x c}{c}.
$$
We find the following explicit solution: 
$$
q^{(0)}=c^{\frac{1}{3}}=\left(1+|\partialx \eta |^2\right)^{-\mez}.
$$

Then, we define $p$ by solving the equation
$$
T_q T_h \sim T_\gamma T_p .
$$
By symbolic calculus, this yields
$$
q h^{(2)} + q h^{(1)} = \gamma^{(3/2)} p^{(1/2)} + \gamma^{(1/2)} p^{(1/2)} 
+ \gamma^{(3/2)} p^{(-1/2)}+\frac{1}{i}\partial_\xi \gamma^{(3/2)}\cdot\partial_x p^{(1/2)}.
$$
Therefore, by identifying terms with the same homogeneity in $\xi$, we successively find that
$$
p^{(1/2)}= \frac{q^{(0)} h^{(2)}}{\gamma^{(3/2)}}= 
q^{(0)}\sqrt{\frac{h^{(2)}}{\lambda^{(1)}}}=
\left(1+|\partialx \eta |^2\right)^{-\frac{5}{4}}\sqrt{\lambda^{(1)}},
$$
and
\begin{equation}\label{defip-12}
p^{(-1/2)}=\frac{1}{\gamma^{(3/2)}} \left\{ q^{(0)} h^{(1)}  - \gamma^{(1/2)} p^{(1/2)}
+i \partial_\xi \gamma^{(3/2)}\cdot\partial_x p^{(1/2)}\right\}.
\end{equation}
Note that the precise value of $p^{(-1/2)}$ is meaningless since we have freely imposed $q^{(-1)}=0$.

Gathering the previous results, and noting that $\gamma^{(1/2)}$ and $p^{(-1/2)}$ depend 
only linearly on the second order derivatives of $\eta$, we have proved the following result.

\begin{prop}\label{key}
Let $q\in\Sigma^0$, $p\in \Sigma^{1/2}$, $\gamma\in \Sigma^{3/2}$ be defined by
\begin{align*}
&q=\left(1+|\partialx \eta |^2\right)^{-\mez},\\
&p=\left(1+|\partialx \eta |^2\right)^{-\frac{5}{4}}\sqrt{\lambda^{(1)}}+p^{(-1/2)},\\
&\gamma=\sqrt{h^{(2)}\lambda^{(1)}}+\sqrt{\frac{h^{(2)}}{\lambda^{(1)}}}\frac{ \RE \lambda^{(0)}}{2 }
-\frac{i}{2}(\partial_\xi\cdot\partial_x) \sqrt{h^{(2)}\lambda^{(1)}}, 
\end{align*}
where $p^{(-1/2)}$ is given by \eqref{defip-12}. 
Then
\begin{equation*}
\left\{
\begin{aligned}
&T_p T_\lambda \sim T_\gamma T_q ,\\
&T_{q} T_{h}\sim T_{\gamma} T_{p} ,\\
&T_\gamma \sim (T_{\gamma})^* .
\end{aligned}
\right.
\end{equation*}
\end{prop}

By combining this symmetrization with the paralinearization, we thus obtain the following symmetrization of the equations.

\begin{coro}\label{psym}
Introduce the new unknowns
$$
\Phi_{1} =  T_{p} \eta  \quad\text{and} \quad\Phi_{2}=T_{q} U.
$$
Then $\Phi_1,\Phi_2 \in C^{0}([0,T];H^s(\xR^d))$ and
\begin{equation}\label{systreduit}
\left\{
\begin{aligned}
\partial_{t}\Phi_{1}+T_{V}\cdot\partialx\Phi_{1} & - T_\gamma \Phi_{2}  = F_1,
\\
\partial_{t}\Phi_{2}+T_{V}\cdot\partialx\Phi_{2} &+ T_\gamma \Phi_{1}=F_2,
\end{aligned}
\right.
\end{equation}
where $F_{1},F_{2}\in L^{\infty}(0,T;H^{s}(\xR^d))$. 
Moreover 
$$
\lA (F_1,F_2)\rA_{L^{\infty}(0,T;H^s\times H^s)}\le 
C \left( \lA (\eta,\psi)\rA_{L^\infty(0,T;H^{s+\mez}\times H^s)}\right),
$$
for some function $C$ depending only on $\dist (\Sigma_0,\Gamma)$.
\end{coro}
To prove Corollary~\ref{psym}, we first note that it follows from Proposition~\ref{key} and Proposition~\ref{prop:csystem2} that
\begin{equation*}
\left\{
\begin{aligned}
\partial_{t}\Phi_{1}+T_{V}\cdot\partialx\Phi_{1} & - T_\gamma \Phi_2 = B_{1}\eta+f_1,
\\
\partial_{t}\Phi_{2}+T_{V}\cdot\partialx\Phi_{2} &+ T_\gamma \Phi_1 = B_{2}U+f_2,
\end{aligned}
\right.
\end{equation*}
with  
$f_{1},f_{2}\in L^{\infty}(0,T;H^{s}(\xR^d))$,
$$
\lA (f_1,f_2) \rA_{L^{\infty}(0,T;H^{s}(\xR^d))}\le 
C \left( \lA (\eta,\psi)\rA_{L^\infty(0,T;H^{s+\mez}(\xR^d)\times H^{s}(\xR^d))}\right),
$$
and 
\begin{align*}
B_{1}&\defn [ \partial_{t}, T_p]+\left[ T_{V}\cdot\partialx,T_{p}\right] ,\\
B_{2}&\defn  [ \partial_{t}, T_q] +\left[ T_{V}\cdot\partialx,T_{q}\right] .
\end{align*}
Writing
\begin{align*}
&\lA B_{1}\eta \rA_{H^s}\le \lA B_1 \rA_{H^{s+\mez}\rightarrow H^s} \lA \eta\rA_{H^{s+\mez}},\\
&\lA B_{2} U \rA_{H^s}\le \lA B_2 \rA_{H^{s}\rightarrow H^s} \lA U\rA_{H^{s}},
\end{align*}
it remains only to estimate $\lA B_1 \rA_{H^{s+\mez}\rightarrow H^s}$ and $\lA B_2 \rA_{H^{s}\rightarrow H^s}$. 
To do so, the only non trivial point is to prove the following lemma.

\begin{lemm}\label{comdtp}
For all $\mu\in\xR$ there exists a non-decreasing function $C$ such that, for 
all $t\in [0,T]$,
$$
\lA T_{ \partial_{t}p(t)} \rA_{H^{\mu}\rightarrow H^{\mu-\mez}}+
\lA T_{\partial_{t}q(t)} \rA_{H^{\mu}\rightarrow H^{\mu}}
\le C\left( \lA (\eta(t),\psi(t))\rA_{H^{s+\mez}\times H^s}\right).
$$
\end{lemm}
\begin{proof}It follows from the Sobolev embedding and \eqref{c43} that
$$
\lA \partial_{t}\eta\rA_{W^{1,\infty}}\les\lA \partial_{t}\eta\rA_{H^{s-1}}\le 
C\left( \lA(  \eta,\psi)\rA_{H^{s+\mez}\times H^s}\right).
$$
This implies that
$$
\lA \partial_t q(t,\cdot)\rA_{L^\infty}+\sup_{\la\xi\ra=1} \lA \partial_{t}p^{(1/2)}(t,\cdot,\xi)\rA_{L^\infty}
\le C\left( \lA( \eta ,\psi )\rA_{H^{s+\mez}\times H^s}\right).
$$
On applying Theorem~\ref{theo:sc0}, this bound implies that 
$$
\lA T_{ \partial_{t}p^{(1/2)}(t)} \rA_{H^{\mu}\rightarrow H^{\mu-\mez}}+
\lA T_{\partial_{t}q(t)} \rA_{H^{\mu}\rightarrow H^{\mu}}
\le C\left( \lA (\eta(t),\psi(t))\rA_{H^{s+\mez}\times H^s}\right).
$$

It remains only to estimate $\lA T_{ \partial_{t}p^{(-1/2)}(t)}\rA_{H^{\mu}\rightarrow H^{\mu-\mez}}$. 
Since we only assume $s>2 +d/2$, a technical difficulty appears. Indeed, 
since $\partial_t$ has the weight of $3/2$ spatial derivatives, and since 
the explicit definition of $p^{(-1/2)}$ involves $2$ spatial derivatives of $\eta$, 
the symbol $\partial_{t}p^{(-1/2)}$ do not belong to $L^\infty$ in general. 
To overcome this technical problem, write $p^{(-1/2)}$ under the form
\begin{align*}
p^{(-1/2)}=\sum_{\la \alpha\ra=2}P_{\alpha}(\partialx\eta,\xi)\partial_{x}^\alpha \eta,
\end{align*}
where the $P_{\alpha}$ are smooth functions of their arguments for $\xi\neq 0$, homogeneous of degree $-1/2$ in $\xi$. 
Now write
\begin{equation}\label{ComSdt}
T_{\partial_{t} p^{(-1/2)}} 
=\sum_{\la\alpha\ra=2}
T_{(\partial_{t}P_{\alpha}(\partialx\eta,\xi) )\partial_x^\alpha\eta}
+\sum_{\la\alpha\ra=2}T_{P_{\alpha}( \partialx\eta,\xi)\partial_{t}\partial_x^\alpha\eta}.
\end{equation}
As above, we obtain
$$
\sup_{\la\xi\ra=1} \lA \partial_{t}P_{\alpha}(\partialx\eta(t),\xi)\rA_{L^{\infty}}
\le C\left( \lA( \eta,\psi)\rA_{H^{s+\mez}\times H^s}\right).
$$
On the other hand we have the obvious estimate  
$\lA \partial_x^\alpha\eta\rA_{L^\infty} \les \lA \eta\rA_{H^{s+\mez}}$. 
On applying Theorem~\ref{theo:sc0}, these bounds imply that the first term 
in the right hand side of \eqref{ComSdt} is uniformly of order $-1/2$. 

The analysis of the second term 
in the right hand side of \eqref{ComSdt} is based on the operator norm estimate \eqref{esti:quants}. 
By applying this estimate with $r=-1/2$, we obtain 
$$
\lA T_{P_{\alpha}( \partialx\eta,\xi)\partial_{t}\partial_x^\alpha\eta}\rA_{H^\mu\rightarrow H^{\mu-1/2}}
\les \lA P_{\alpha}( \partialx\eta,\xi)\partial_{t}\partial_x^\alpha\eta\rA_{H^{s-3}}.
$$
Now the product rule \eqref{pr} implies that
\begin{align*}
&\lA P_{\alpha}( \partialx\eta,\xi)\partial_{t}\partial_x^\alpha\eta\rA_{H^{s-3}}\\
&\qquad\les \left\{ \la P_\alpha(0,\xi)\ra+\lA P_{\alpha}( \partialx\eta,\xi)-P_\alpha(0,\xi)\rA_{H^{s-1}}\right\}
\lA \partial_{t}\partial_x^\alpha\eta\rA_{H^{s-3}},
\end{align*}
and hence
$$
\lA T_{P_{\alpha}( \partialx\eta,\xi)\partial_{t}\partial_x^\alpha\eta}\rA_{H^\mu\rightarrow H^{\mu-m}}
\le  C( \lA \eta\rA_{H^{s}})\lA \partial_t \eta\rA_{H^{s-1}}\le C\left( \lA( \eta,\psi)\rA_{H^{s+\mez}\times H^s}\right) .
$$
This completes the proof. 
\end{proof}

\section{A priori estimates}\label{sle}
Consider the Cauchy problem
\begin{equation}\label{A1}
\left\{
\begin{aligned}
&\partial_{t}\eta-G(\eta)\psi=0,\\
&\partial_{t}\psi+g \eta- H(\eta)
+ \frac{1}{2}\la\partialx \psi\ra^2  -\frac{1}{2}
\frac{\bigl(\partialx  \eta\cdot\partialx \psi +G(\eta) \psi \bigr)^2}{1+|\partialx  \eta|^2}
= 0,
\end{aligned}
\right.
\end{equation}
with initial data
$$
\eta\arrowvert_{t=0}=\eta_0,\quad \psi\arrowvert_{t=0}=\psi_0.
$$
In this section we prove a priori estimates for solutions to the system~\eqref{A1} 
and approximates systems. These estimates are crucial in the proof of existence and uniqueness of solutions to~\eqref{A1} .

\subsection{Reformulation}\label{s.5.1}
The first step is the following reformulation, whose proof is an immediate computation.

\begin{lemm}
$(\eta,\psi)$ solves \eqref{A1} if and only if 
\begin{equation*}
\begin{pmatrix}
I & 0 \\
-T_\mathfrak{B} & I \end{pmatrix}
\left( \partial_t+ T_V \cdot\partialx\right) 
 \begin{pmatrix} \eta \\ \psi \end{pmatrix}
+\begin{pmatrix}
0 & -T_{\lambda} \\
T_{h} & 0 \end{pmatrix}
\begin{pmatrix}
I & 0 \\
-T_\mathfrak{B} & I \end{pmatrix}
\begin{pmatrix} \eta \\ \psi \end{pmatrix}
=\begin{pmatrix} f^1 \\ f^2 \end{pmatrix},
\end{equation*}
where
\begin{equation}\label{f1f2}
\begin{aligned}
f^1&=G(\eta)\psi -\big\{ T_{\lambda} ( \psi -T_\mathfrak{B} \eta) - T_V\cdot\partialx \eta \big\},\\
f^2&=-\mez \la \partialx \psi\ra^2 +\mez \frac{\left( \partialx \eta \cdot \partialx \psi + G(\eta)\psi \right)^2}{1+\la \partialx\eta\ra^2} 
+ H(\eta) \\
&\quad + T_V \partialx \psi - T_\mathfrak{B}T_V\cdot \partialx  \eta  -T_{\mathfrak{B}}G(\eta)\psi 
+ T_{h}\eta -g\eta.
\end{aligned}
\end{equation}
\end{lemm}

Since
$$
\begin{pmatrix}
I & 0 \\
T_\mathfrak{B} & I \end{pmatrix}
\begin{pmatrix}
I & 0 \\
-T_\mathfrak{B} & I \end{pmatrix}
=
\begin{pmatrix}
I & 0 \\
0 & I \end{pmatrix},
$$
we thus find that $(\eta,\psi)$ solves \eqref{A1} 
if and only if
\begin{equation}\label{A2}
\left\{
\begin{aligned}
&\left( \partial_t+ T_V\cdot \partialx 
+
\mathcal{L}\right) \begin{pmatrix}\eta \\ \psi\end{pmatrix}
=f(\eta,\psi),\\
& (\eta,\psi)\arrowvert_{t=0}=( \eta_{0} , \psi_{0}),
\end{aligned}
\right.
\end{equation}
with
$$
\mathcal{L}\defn \begin{pmatrix}
I & 0 \\
T_\mathfrak{B} & I \end{pmatrix}
\begin{pmatrix}
0 & -T_{\lambda}\\
T_{h} & 0 \end{pmatrix}
\begin{pmatrix}
I & 0 \\
-T_\mathfrak{B} & I \end{pmatrix},\quad
f(\eta,\psi)\defn\begin{pmatrix}
I & 0 \\
T_\mathfrak{B} & I \end{pmatrix}\begin{pmatrix} f^1 \\ f^2 \end{pmatrix}.
$$

\subsection{Approximate equations}\label{SAE}
We shall seek solutions of the Cauchy problem~\eqref{A2} as limits of solutions of approximating systems. 
The definition depends on two operators. The first one is a well-chosen mollifier. The second one is an approximate 
right-parametrix for the symmetrizer $S=\left(\begin{smallmatrix} T_p & 0 \\ 0 & T_q \end{smallmatrix}\right)$ defined in Section~\ref{Sym}.

\smallbreak
\noindent\textbf{Mollifiers.} To regularize the equations, we cannot use 
usual mollifiers of the form $\chi(\eps D_x)$. Instead we use the following variant. 
Given $\eps\in [0,1]$, we define $J_\eps$ as the 
paradifferential operator with symbol $\jmath_\eps=\jmath_\eps(t,x,\xi)$ given by
$$
\jmath_\eps=\jmath_\eps^{(0)}+\jmath_\eps^{(-1)}= \exp \big( -\eps \gamma^{(3/2)}\big) -\frac{i}{2}(\partial_x\cdot\partial_\xi)\exp \big( -\eps \gamma^{(3/2)}\big).
$$
The important facts are that 
$$
\jmath_\eps\in C^{0}([0,T];\Gamma^0_{3/2}(\xR^d)),\quad \{\jmath_\eps^{(0)}, \gamma^{(3/2)}\}=0,\quad
\IM \jmath_\eps^{(-1)} = -\frac{1}{2}(\partial_x\cdot\partial_\xi) \jmath_\eps^{(0)}.
$$
Of course, for any $\eps >0$, $\jmath_\eps \in C^{0}([0,T];\Gamma^m_{3/2}(\xR^d))$ for all $m\le 0$. However, the important fact is that
$\jmath_\eps$ is uniformly bounded in $C^{0}([0,T];\Gamma^0_{3/2}(\xR^d))$ for all $\eps\in [0,1]$.  Therefore, 
we have the following uniform estimates:
\begin{align*}
&\lA J_\eps T_\gamma -T_\gamma J_\eps\rA_{H^\mu\rightarrow H^\mu} \le C(\lA \partialx\eta\rA_{W^{3/2,\infty}}),\\
&\lA (J_\eps)^* -J_\eps\rA_{H^\mu\rightarrow H^{\mu+3/2}} \le C(\lA \partialx\eta\rA_{W^{3/2,\infty}}),
\end{align*}
for some non-decreasing function $C$ independent of $\eps\in [0,1]$.
In other words, we have
$$
J_\eps T_\gamma \sim T_\gamma J_\eps, \quad (J_\eps)^*\sim J_\eps,
$$
uniformly in $\eps$.

\smallbreak
\noindent\textbf{Parametrix for the symmetrizer.} Recall that the class of symbols $\Sigma^m$ have been defined in Definition~\ref{defiSigma}. We seek
$$
\wp=\wp^{(-1/2)}+\wp^{(-3/2)}\in \Sigma^{-1/2}
$$
such that 
$$
p\sharp \wp = p^{(1/2)}\wp^{(-1/2)}+p^{(1/2)}\wp^{(-3/2)}+p^{(-1/2)}\wp^{(-1/2)}+\frac{1}{i}\partial_\xi p^{(1/2)}\cdot \partial_x \wp^{(-1/2)} =1.
$$
To solve this equation we explicitely set
\begin{equation}\label{dmu-1}
\begin{aligned}
\wp^{(-1/2)}&=\frac{1}{p^{(1/2)}},\\ 
\wp^{(-3/2)}&= -\frac{1}{p^{(1/2)}}\left( \wp^{(-1/2)}p^{(-1/2)} +\frac{1}{i} \partial_\xi \wp^{(-1/2)} \cdot \partial_x p^{(1/2)} \right).
\end{aligned}
\end{equation}
Therefore
$$
T_p T_\wp  \sim I,
$$
where recall that the notation $A\sim B$ is as defined in Definition~\ref{defisym}.

On the other hand, since $q=\left(1+|\partialx \eta |^2\right)^{-\mez}$ does not depend on $\xi$, it follows from \eqref{iii} that we have
$$
T_q T_{1/q}\sim I.
$$
Hence, with $\wp$ and $q$ as defined above, we have
$$
\begin{pmatrix} T_p & 0 \\ 0 & T_q\end{pmatrix} 
\begin{pmatrix} T_\wp & 0 \\ 0 & T_{1/q} \end{pmatrix}
\sim
\begin{pmatrix} I & 0 \\ 0 & I\end{pmatrix}. 
$$

\smallbreak
\noindent\textbf{Approximate system.} We then define
$$
\mathcal{L}^\eps\defn 
\begin{pmatrix}
I & 0 \\
T_{\mathfrak{B}} & I \end{pmatrix}
\begin{pmatrix}
0 & -  T_{\lambda}   \\
T_{h}  & 0 \end{pmatrix}
\begin{pmatrix}
T_{\wp}  J_\eps T_p & 0\\
0 & T_{1/q}J_\eps T_q   \end{pmatrix}
\begin{pmatrix}
I&0\\
-T_{\mathfrak{B}} & I 
\end{pmatrix}.
$$
(At first one may not expect to have to introduce $J_\eps$ and $\mathcal{L}^\eps$. 
We explain the reason to introduce these operators in \S\ref{keyi} below.) 
We seek solutions $(\eta,\psi)$ of~\eqref{A2} as limits of solutions 
of the following Cauchy problems
\begin{equation}\label{A3}
\left\{
\begin{aligned}
&\left( \partial_t+   T_{V} \cdot\partialx J_\eps 
+
\mathcal{L}^\eps\right)\begin{pmatrix}\eta \\ \psi\end{pmatrix}
=f( J_\eps \eta,J_\eps \psi),\\
& (\eta,\psi)\arrowvert_{t=0}=(\eta_{0},\psi_0).
\end{aligned}
\right.
\end{equation}

\subsection{Uniform estimates}
Our main task will consist in proving uniform estimates for this system. Namely, we shall prove the following proposition.
\begin{prop}\label{ueps}
Let $d\ge 1$ and $s>2+d/2$. Then there exist a non-decreasing function $C$ such that, for all $\eps\in ]0,1$, all $T\in ]0,1]$ and all 
solution $(\eta,\psi)$ of \eqref{A3} such that
$$
(\eta,\psi)\in C^{1}([0,T];H^{s+\mez}(\xR^d)\times H^{s}(\xR^d)),
$$
the norm
$$
M(T) = \lA (\eta,\psi)\rA_{L^{\infty}(0,T;H^{s+\mez}\times H^{s})},
$$
satisfies the estimate
$$
M(T)\le C(M_0) + T C(M(T)),
$$
with $M_0\defn  \lA (\eta_0,\psi_0)\rA_{H^{s+\mez}\times H^{s}}$.
\end{prop}
\begin{rema}
Notice that the estimates holds for $\epsilon = 0$. In particular this proposition contains {\em a priori} 
estimates for the water waves 
system itself. 
\end{rema}

\subsection{The key identities}\label{keyi} To ease the reading, we here explain what are the key identities in the proof of 
Proposition~\ref{ueps}. 

By definition of $\mathcal{L}^\eps$, using that $\left(\begin{smallmatrix}
I&0\\
-T_{\mathfrak{B}} & I 
\end{smallmatrix}\right)\left(
\begin{smallmatrix}
I & 0 \\
T_{\mathfrak{B}} & I \end{smallmatrix}\right)=\left(
\begin{smallmatrix}
I & 0 \\
0 & I \end{smallmatrix}\right)$, 
we have
\begin{align*}
&\begin{pmatrix}
T_{p} & 0 \\ 0 & T_{q}\end{pmatrix}
\begin{pmatrix}
I&0\\
-T_{\mathfrak{B}} & I 
\end{pmatrix}
\mathcal{L}^\eps\\
&\qquad=\begin{pmatrix}
T_{p} & 0 \\ 0 & T_{q}\end{pmatrix}
\begin{pmatrix}
0 & -  T_{\lambda}   \\
T_{h}  & 0 \end{pmatrix}
\begin{pmatrix}
T_{\wp}  J_\eps T_p & 0\\
0 & T_{1/q}J_\eps T_q   \end{pmatrix}
\begin{pmatrix}
I&0\\
-T_{\mathfrak{B}} & I 
\end{pmatrix}.
\end{align*}
Now recall that
$$
\begin{pmatrix}
T_{p} & 0 \\ 0 & T_{q}\end{pmatrix}
\begin{pmatrix} 0 &-T_\lambda  \\  T_h & 0 \end{pmatrix} 
\sim \begin{pmatrix} 0 & -T_\gamma \\ (T_\gamma)^* & 0 \end{pmatrix} 
\begin{pmatrix}
T_{p} & 0 \\ 0 & T_{q}\end{pmatrix},
$$
so that
\begin{align*}
&\begin{pmatrix}
T_{p} & 0 \\ 0 & T_{q}\end{pmatrix}
\begin{pmatrix}
I&0\\
-T_{\mathfrak{B}} & I 
\end{pmatrix}
\mathcal{L}^\eps\\
&\qquad \sim
\begin{pmatrix} 0 & -T_\gamma \\ (T_\gamma)^* & 0 \end{pmatrix} 
\begin{pmatrix}
T_{p} & 0 \\ 0 & T_{q}\end{pmatrix}
\begin{pmatrix}
T_{\wp}  J_\eps T_p & 0\\
0 & T_{1/q}J_\eps T_q   \end{pmatrix}
\begin{pmatrix}
I&0\\
-T_{\mathfrak{B}} & I 
\end{pmatrix}
\end{align*}
uniformly in $\eps$ (notice that the remainders associated to the notation $\sim$ are uniformly bounded). 
We next use
$$
\begin{pmatrix} T_p & 0 \\ 0 & T_q\end{pmatrix} 
\begin{pmatrix} T_\wp & 0 \\ 0 & T_{1/q} \end{pmatrix}
\sim
\begin{pmatrix} I & 0 \\ 0 & I\end{pmatrix},
$$
to obtain that, uniformly in $\eps$, we have the key identity
\begin{align*}
\begin{pmatrix}
T_{p} & 0 \\ 0 & T_{q}\end{pmatrix}
\begin{pmatrix}
I&0\\
-T_{\mathfrak{B}} & I 
\end{pmatrix}
\mathcal{L}^\eps\sim
\begin{pmatrix} 0 & -T_\gamma J_\eps \\ (T_\gamma)^* J_\eps & 0 \end{pmatrix} 
\begin{pmatrix}
T_{p} & 0 \\ 0 & T_{q}\end{pmatrix}
\begin{pmatrix}
I&0\\
-T_{\mathfrak{B}} & I 
\end{pmatrix}.
\end{align*}
In other words, the symmetrizer 
$$
\begin{pmatrix}
T_{p} & 0 \\ 0 & T_{q}\end{pmatrix}
\begin{pmatrix}
I&0\\
-T_{\mathfrak{B}} & I 
\end{pmatrix}
$$
conjugates $\mathcal{L}^\eps$ to a simple operator which is skew symmetric in the following sense: 
$$
\begin{pmatrix} 0 & -T_\gamma J_\eps \\ (T_\gamma)^* J_\eps & 0 \end{pmatrix} ^*
\sim-\begin{pmatrix} 0 & -T_\gamma J_\eps \\ (T_\gamma)^* J_\eps & 0 \end{pmatrix}.
$$
This is our second key identity, which comes from the fact that 
$$
(T_\gamma)^*\sim T_\gamma, \quad J_\eps^*\sim J_\eps, \quad T_\gamma J_\eps \sim J_\eps T_\gamma.
$$ 
In particular, it is essential to chose a good mollifier so that the last two identities hold true. 

\smallbreak

We could mention that, in the proof of Proposition~\ref{ueps} below, the main argument is the fact that 
the term $F_{2,\eps}$ in \eqref{defiF2} is uniformly bounded in $L^\infty(0,T;H^{s}\times H^{s})$. 
The other arguments are only technical arguments. However, since we only assume that $s>2+\frac{d}{2}$, this requires some care 
and we give a complete proof.

\subsection{Proof of Proposition~\ref{ueps}}\label{sueps} We now prove Proposition~\ref{ueps}.

\smallbreak 
\noindent \textbf{a)} Let us set
$$
U= \psi -T_{\mathfrak{B}} \eta,\quad
\Phi
=
\begin{pmatrix}
T_p \eta \\
T_q U
\end{pmatrix}
= \begin{pmatrix}
T_{p} & 0 \\ 0 & T_{q}\end{pmatrix}
\begin{pmatrix}
I&0\\
-T_{\mathfrak{B}} & I 
\end{pmatrix}\begin{pmatrix} \eta \\ \psi \end{pmatrix}.
$$
We claim that $\Phi$ satisfies an equation of the form
\begin{equation}\label{Ya}
\left(\partial_{t}+ T_{V}\cdot\partialx J_\eps \right)  \Phi +
\begin{pmatrix}
0 & - T_{\gamma} J_\eps \\
T_{\gamma} J_\eps  & 0 \end{pmatrix}
\Phi
= F_{\eps},
\end{equation}
where the remainder satisfies
\begin{equation}\label{RHSF}
\lA F_\eps\rA_{L^\infty(0,T;H^{s}\times H^{s})} \le 
C\left(  \lA (\eta,\psi)\rA_{L^\infty(0,T;H^{s+\mez}\times H^s)}\right),
\end{equation}
for some non-decreasing function $C$ independent of $\eps$. 
To prove this claim,

we begin by commuting the equation \eqref{A3} with the matrix 
$$
 \begin{pmatrix}
T_{p} & 0 \\ 0 & T_{q}\end{pmatrix}
\begin{pmatrix}
I&0\\
-T_{\mathfrak{B}} & I 
\end{pmatrix},
$$
to obtain that $\Phi$ satisfies \eqref{Ya} with $F_\eps=F_{1,\eps}+F_{2,\eps}+F_{3,\eps}$ where (cf \S\ref{keyi})
\begin{equation}\label{defiF2}
\begin{aligned}
&F_{1,\eps}=  
\begin{pmatrix}
T_{p}f^1( J_\eps \eta,J_\eps \psi) \\
T_{q}f^2 ( J_\eps \eta,J_\eps \psi)
\end{pmatrix},\\
&F_{2,\eps}= \begin{pmatrix}
0 & -    ( T_p T_\lambda T_{1/q}J_\eps- T_\gamma J_\eps) \\
(T_q T_h T_\wp J_\eps -T_\gamma J_\eps)& 0 
 \end{pmatrix} \Phi,\\
&F_{3,\eps}=\left[ \partial_{t}+ T_{V}\cdot\partialx J_\eps , 
\begin{pmatrix}
T_{p} & 0 \\ 0 & T_{q}\end{pmatrix}
\begin{pmatrix}
I&0\\
-T_{\mathfrak{B}} & I 
\end{pmatrix}
\right] \begin{pmatrix}\eta\\ \psi\end{pmatrix}.
\end{aligned}
\end{equation}
The estimate of the first term follows from 
Proposition~\ref{prop:csystem}, Lemma~\ref{paraH} and Lemma~\ref{lemmDBC} (clearly, these results applies with 
$(\eta,\psi)$ replaced by $(J_\eps \eta,J_\eps \psi)$). 
For the second term we use that, 
$$
T_pT_\lambda\sim T_\gamma T_q ,\quad T_qT_h \sim T_\gamma T_p,
\quad 
T_p T_\wp\sim I,\quad T_q T_{1/q}\sim I,
$$
to obtain
$$
T_p T_\lambda T_{1/q}\sim T_\gamma ,\quad T_q T_h T_\wp \sim T_\gamma.
$$
Eventually, we estimate the last term as in the proof of Corollary~\ref{psym}.

\smallbreak
\noindent \textbf{b)}
We next claim that
\begin{equation}\label{iM1}
\lA (\eta,\psi) \rA_{L^{\infty}(0,T;H^{s-1}\times H^{s-\tdm})}\le C(M_0) + T C(M(T)). 
\end{equation}
We prove the desired estimate for $\partial_t \eta$ only. To do so, 
using the obvious inequality
\begin{align*}
\lA \eta(t)\rA_{H^{s-1}}
&\le \lA \eta(0)\rA_{H^{s-1}}+\int_0^t \lA \partial_t \eta\rA_{H^{s-1}}\\
&\le M_0 + T \lA \partial_t \eta \rA_{L^{\infty}(0,T;H^{s-1})},
\end{align*}
we see that it is enough to prove that
\begin{equation}\label{eeta}
\lA \partial_t \eta \rA_{L^{\infty}(0,T;H^{s-1})}\le C(M(T)). 
\end{equation}
This in turn follows directly from the equation for $\eta$. Indeed, directly from \eqref{A3}, write
$$
\partial_t \eta 
=-T_{V}\cdot\partialx J_\eps \eta 
+T_{\lambda}T_{1/q}J_\eps T_q(\psi -T_{\mathfrak{B}}\eta) 
+f_{1}(J_\eps\eta,J_\eps\psi).
$$
The last term is estimated by means of Proposition~\ref{prop:csystem}. 
Moving to the first two terms, by the usual continuity estimate for paradifferential operators \eqref{esti:quant1}, we have
$$
\lA T_{V}\cdot\partialx J_\eps \eta\rA_{H^{s-1}}\le  \lA V\rA_{L^\infty} \lA J_\eps \eta\rA_{H^{s}},
$$
and
\begin{align*}
&\lA T_{\lambda}T_{1/q}J_\eps T_q (\psi -T_{\mathfrak{B}}\eta)\rA_{H^{s-1}}\\
&\qquad \le \lA T_{\lambda}T_{1/q}J_\eps T_q \rA_{H^{s}\rightarrow H^{s-1}}
\left\{\lA \psi\rA_{H^{s}}+\lA\mathfrak{B}\rA_{L^\infty}\lA \eta\rA_{H^{s}}\right\},
\end{align*}
and hence, since $H^{s-1}(\xR^d)\subset L^\infty(\xR^d)$, the estimates for $\mathfrak{B}$ and $V$ in \eqref{c43} imply that 
$\partial_t\eta$ satisfies the desired estimate \eqref{eeta}. The estimate of 
$\lA \psi\rA_{H^{s-3/2}}$ is analoguous. This completes the proof of the claim.

\smallbreak
\noindent \textbf{c)} To obtain estimates in Sobolev, we shall commute the equation with an elliptic operator of order $s$ and then use an $L^2$-energy estimate. 
Again, one has to chose carefully the elliptic operator. The most natural choice consists in introducing the paradifferential operator 
$T_\beta$ with symbol
\begin{equation}\label{defibeta}
\beta\defn \left( \gamma^{(3/2)}\right)^{\frac{2s}{3}} \in \Sigma^s.
\end{equation}
The key point is that, since $\beta$ and $\jmath_\eps^{(0)}$ are (nonlinear) functions of 
$\gamma^{(3/2)}$, we have
\begin{align*}
&\partial_\xi\beta\cdot\partial_x \gamma^{(3/2)} = \partial_\xi \gamma^{(3/2)}\cdot\partial_x \beta,\\
&\partial_\xi\beta\cdot\partial_x \jmath_\eps^{(0)} = \partial_\xi \jmath_\eps^{(0)}\cdot\partial_x \beta.
\end{align*}
Therefore, as above, we find that $\left[T_{\beta},T_{\gamma}\right]$ is of order $s$, while 
$\left[T_{\beta},J_\eps \right]$ is of order $s-3/2$. 
Also the commutator $\left[T_{\beta},T_{V}\cdot\partialx J_\eps \right]$ is clearly of order~$s$. 
With regards to the commutator $[T_\beta,T_{\partial_t }]=-T_{\partial_t \beta}$ notice that there is no difficulty. Indeed, since 
$\beta$ is of the form $\beta=B(\partialx\eta,\xi)$, 
the most direct estimate shows that the $L_x^\infty(\xR^d)$-norm of $\partial_t \beta$ is 
estimated by the $L_x^\infty(\xR^d)$-norm of $(\partialx\eta,\partial_t\partialx\eta)$ and hence by 
$C(M(T))$ in view of \eqref{eeta} and the Sobolev embedding $H^{s-1}(\xR^d)\subset W^{1,\infty}(\xR^d)$. We thus end up with the following 
uniform estimates
\begin{align*}
\lA \left[T_{\beta},T_{\gamma}\right]J_\eps \rA_{H^s\rightarrow L^2} &\le C(M(T)),\\
\lA T_\gamma \left[T_{\beta},J_\eps \right] \rA_{H^s\rightarrow L^2} &\le C(M(T)), \\
\lA \left[T_{\beta},T_{V}\cdot\partialx J_\eps \right]\rA_{H^s\rightarrow L^2}
&\le C(M(T)),\\
\lA \left[T_{\beta},\partial_t\right]\rA_{H^s\rightarrow L^2}
&\le C(M(T)),
\end{align*}
for some non-decreasing function $C$ independent of $\eps\in [0,1]$. 
Therefore, by commuting the equation \eqref{Ya} with $T_\beta$, we find that 
$$
\varphi\defn T_\beta \Phi
$$
satisfies
\begin{equation}\label{Yb}
\left(\partial_{t}+ T_{V}\cdot\partialx J_\eps \right) \varphi +
\begin{pmatrix}
0 & - T_{\gamma} J_\eps \\
T_{\gamma} J_\eps & 0 \end{pmatrix}
\varphi
= F'_{\eps},
\end{equation}
with
\begin{equation*}
\lA F'_\eps\rA_{L^{\infty}(0,T;L^{2}\times L^{2})} \le 
C(M(T)),
\end{equation*}
for some non-decreasing function $C$ independent of $\eps\in [0,1]$.

\smallbreak
\noindent \textbf{d)} 
Since by assumption $(\eta,\psi)$ is $C^1$ in time with values in $H^{s+\mez}(\xR^d)\times H^s(\xR^d)$, we have
$$
\varphi\in C^{1}([0,T];L^{2}(\xR^d)\times L^2(\xR^d)),
$$
and hence we can write
$$
\frac{d}{dt}\left\langle\varphi,\varphi\right\rangle = 2 \RE \left\langle \partial_{t}\varphi,\varphi\right\rangle,
$$ 
where $\langle\cdot,\cdot\rangle$ denotes the scalar product in $L^{2}(\xR^d)\times L^2(\xR^d)$. Therefore, \eqref{Yb} implies that
\begin{equation*}
\frac{d}{dt}\left\langle\varphi,\varphi\right\rangle
=2\RE \left\langle - T_{V}\cdot\partialx J_\eps\varphi - \begin{pmatrix}
0 & - T_{\gamma} J_\eps \\
T_{\gamma} J_\eps & 0 \end{pmatrix}\varphi+F'_{\eps},\varphi \right\rangle
\end{equation*}
and hence
\begin{equation*}
\frac{d}{dt}\left\langle\varphi,\varphi\right\rangle
= \left\langle   \mathscr{R}^\eps\varphi ,\varphi \right\rangle+2\RE \left\langle F'_{\eps},\varphi \right\rangle,
\end{equation*}
where $\mathscr{R}^\eps$ is the matrix-valued operator
$$
\mathscr{R}^\eps\defn 
-\left\{(T_{V}\cdot\partialx  J_\eps)^*+T_{V}\cdot\partialx J_\eps \right\} I+ \begin{pmatrix}
0 & - T_{\gamma} J_\eps \\
T_{\gamma} J_\eps & 0 \end{pmatrix} + \begin{pmatrix}
0 & - T_{\gamma} J_\eps \\
T_{\gamma} J_\eps & 0 \end{pmatrix}^*.
$$
Now recall that
$$ 
(T_\gamma)^*\sim T_\gamma,\quad (J_\eps)^*\sim J_\eps,\quad 
T_\gamma J_\eps \sim J_\eps T_\gamma.
$$
Moreover, we easily verify that 
$$
\sup_{\eps \in [0,1]}\sup_{t\in [0,T]}\lA \mathscr{R}^\eps(t)\rA_{L^2\times L^2\rightarrow L^2\times L^2}\le C(M(T)).
$$
Therefore, integrating in time we conclude that for all $t\in [0,T]$,
\begin{equation*}
\lA\varphi(t)\rA_{L^2\times L^2}^2-\lA\varphi(0)\rA_{L^2\times L^2}^2 
\le C(M(T))\int_0^T \left( \lA \varphi\rA_{L^2\times L^2}^2+\lA F'_{\eps}\rA_{L^2\times L^2}^2\right)  \, dt',
\end{equation*}
which immediately implies that
$$
\lA\varphi\rA_{L^\infty(0,T;L^2\times L^2)}\le C(M_0)+TC(M(T)).
$$
By definition of $\varphi$, this yields
\begin{equation}\label{im}
 \lA T_\beta T_p \eta \rA_{L^\infty(0,T:L^2)}+\lA T_\beta T_q U\rA_{L^\infty(0,T;L^2)}\le C(M_0)+TC(M(T)).
\end{equation}

\smallbreak 
First of all, we use Proposition~\ref{2d22} to obtain 
\begin{align}
&\lA \eta\rA_{L^\infty(0,T;H^{s+\mez})} \le K\left\{  \lA T_\beta T_p \eta \rA_{L^\infty(0,T;L^{2})} + 
\lA \eta\rA_{L^\infty(0,T;H^{\mez})}\right\},\label{defiK1}\\
&\lA \psi\rA_{L^\infty(0,T;H^{s})} \le  K \left\{\lA T_\beta T_q \psi \rA_{L^\infty(0,T;L^{2})} + 
\lA \psi\rA_{L^\infty(0,T;L^2)}\right\},\label{defiK2}
\end{align}
where $K$ depends only on $\lA \eta\rA_{L^\infty(0,T;H^{s-1})}$. 

Let us prove that the constant $K$ satisfies an inequality of the form
\begin{equation}\label{iM2}
K\le C(M_0)+TC(M(T)).
\end{equation}
To see this, notice that one can assume without loss of generality that
$$
K\le F(\lA \eta\rA_{L^\infty(0,T;H^{s-1})}^2)
$$
for some non-decreasing function $F\in C^1(\xR)$. Set $\Cr(t)= F(\lA \eta(t)\rA_{H^{s-1}}^2)$. 
We then obtain the desired bound \eqref{iM2} from \eqref{eeta} and the inequality
\begin{align*}
K&\le \Cr(0)+\int_0^T\la \Cr'(t)\ra \,dt\\
&\le F(M_0) + \int_{0}^T  2 F'(\lA \eta\rA_{H^{s-1}}^2) \lA \partial_{t} \eta\rA_{H^{s-1}} \lA \eta \rA_{H^{s-1}}\, dt.
\end{align*}
Consequently, \eqref{im} and \eqref{defiK1} imply that we have 
$$ 
\lA \eta\rA_{L^\infty(0,T;H^{s+\mez})} \le C(M_0)+T C(M(T)).
$$

It remains to prove an estimate for $\psi$. To do this, we begin by noting that, 
since $\psi=U+T_{\mathfrak{B}}\eta$, we have
\begin{align*}
&\lA T_\beta T_q \psi\rA_{L^{\infty}(0,T;L^{2})} \\
&\qquad\le  
\lA T_\beta T_q U \rA_{L^\infty(0,T;L^2)} + \lA T_\beta T_q T_\mathfrak{B}\rA_{L^\infty(0,T;H^{s}\rightarrow L^2)}\lA \eta\rA_{L^\infty(0,T;H^s)}.
\end{align*}
Now we have
\begin{align*}
&\lA T_\beta T_q T_\mathfrak{B}\rA_{L^\infty(0,T;H^{s}\rightarrow L^2)}\\
&\qquad \le 
 \sup_{t\in [0,T]}\sup_{\la\xi\ra=1}\lA \beta(t,\cdot,\xi)\rA_{L_x^\infty}
\lA q\rA_{L^\infty(0,T;L^\infty)}\lA \mathfrak{B}\rA_{L^\infty(0,T;L^\infty)}
\end{align*}
and hence
\begin{equation}
\lA \psi\rA_{H^{s}} \le  K' \left\{\lA T_q U \rA_{H^{s}} + \lA \psi\rA_{L^2}+\lA \eta\rA_{H^s}\right\},
\end{equation}
where $K'$ depends only on $\lA (\eta,\psi)\rA_{L^\infty(0,T;H^{s-1}\times H^{s-3/2})}$. By using the inequality \eqref{im} for 
$\lA T_\beta U\rA_{L^2}$, the estiumate \eqref{iM1} for $\lA \psi\rA_{L^2}$, the previous estimate for $\eta$, and the fact that $K'$ satisfies 
the same estimate as $K$ does, we conclude that
$$ 
\lA \psi\rA_{L^\infty(0,T;H^{s})} \le C(M_0)+T C(M(T)).
$$
We end up with $M(T)\le  C(M_0)+T C(M(T)) $. This completes the proof of Proposition~\ref{ueps}.

\subsection{}

Consider $(\eta,\psi)\in C^{0}([0,T];H^{s+\mez}(\xR^d)\times H^{s}(\xR^d))$ solution to the system
\begin{equation*}
\left\{
\begin{aligned}
&\left( \partial_t+   T_{V} \cdot\partialx J_\eps 
+
\mathcal{L}^\eps\right)\begin{pmatrix}\eta \\ \psi\end{pmatrix}
=f( J_\eps \eta,J_\eps \psi),\\
& (\eta,\psi)\arrowvert_{t=0}=(\eta_{0},\psi_0).
\end{aligned}
\right.
\end{equation*} 
We now prove uniform estimates for solutions $(\vareta,\varpsi)$ to the linear system  
\begin{equation}\label{A3bis}
\left\{
\begin{aligned}
&\left( \partial_t+   T_{V} \cdot\partialx J_\eps 
+
\mathcal{L}^\eps\right)\begin{pmatrix}\vareta \\ \varpsi\end{pmatrix}
=F,\\
& (\vareta,\varpsi)\arrowvert_{t=0}=(\vareta_{0},\varpsi_0).
\end{aligned}
\right.
\end{equation}
To clarify notations, write \eqref{A3} in the compact form
$$
E(\eps,\eta,\psi)\begin{pmatrix}\eta \\ \psi\end{pmatrix}=f (J_\eps \eta,J_\eps \psi)
$$
Then, with this notations, we shall prove estimates for the system
$$
E(\eps,\eta,\psi)\begin{pmatrix}\vareta \\ \varpsi\end{pmatrix}=F.
$$
We shall also use the following notation: given $r\ge 0$, $T>0$ and 
two real-valued functions $u_1,u_2$, we set
\begin{equation}
\lA (u_1,u_2)\rA_{X^r(T)}
\defn \lA (u_1,u_2) \rA_{L^\infty(0,T;H^{r+\mez}\times H^r)}
\end{equation}
We shall prove the following extension of Proposition~\ref{ueps}. 
\begin{prop}\label{ueps2}
Let $d\ge 1$, $s>2+d/2$ and $0\le \sigma\le s$. 
Then there exist a non-decreasing function $C$ such that, 
for all $\eps\in [0,1]$, all $T\in ]0,1]$ and all $\vareta,\varpsi,\eta,\psi,F$ such that
$$
E(\eps,\eta,\psi)\begin{pmatrix}\eta \\ \psi\end{pmatrix}=f (J_\eps \eta,J_\eps \psi),\quad 
E(\eps,\eta,\psi)\begin{pmatrix}\vareta \\ \varpsi\end{pmatrix}=F,
$$
and such that
\begin{align*}
&(\eta,\psi)\in C^{0}([0,T];H^{s+\mez}(\xR^d)\times H^{s}(\xR^d)),\\
&(\vareta,\varpsi)\in C^{1}([0,T];H^{\sigma+\mez}(\xR^d)\times H^{\sigma}(\xR^d)),\\
\quad
&F=(F_1,F_2)\in L^{\infty}([0,T];H^{\sigma+\mez}(\xR^d)\times H^{\sigma}(\xR^d)),
\end{align*}
we have
\begin{multline}\label{estuni2}
\lA (\vareta,\varpsi)\rA_{X^\sigma(T)}\le 
\widetilde{C}
\lA (\vareta_0,\varpsi_0)\rA_{H^{\sigma+\mez}\times H^\sigma}
\\
+T C\left( \lA (\eta,\psi)\rA_{X^s(T)}\right)
\lA (\vareta,\varpsi)\rA_{X^\sigma(T)}
+T\lA F \rA_{X^\sigma(T)},
\end{multline}
where 
$\widetilde{C}\defn C\left(\lA (\eta_0,\psi_0)\rA_{H^{s+\mez}\times H^s}\right)+T C\left( \lA (\eta,\psi)\rA_{X^s(T)}\right)$.
\end{prop}
\begin{rema}
By applying this proposition with $(\eta,\psi)=(\vareta,\varpsi)$ we obtain Propositon~\ref{ueps}.
\end{rema}
\begin{proof}We still denote by $p,q,\gamma,\wp$ the symbols already introduced above. 
They are functions of $\eta$ only. Similarly, $\mathfrak{B}$ and $V$ are functions of the coefficient 
$(\eta,\psi)$. We use tildas to indicate that the new unknowns we introduce 
depend linearly on $(\vareta,\varpsi)$, with some coefficients depending on the coefficients $(\eta,\psi)$.

i) Let us set
$$
\U= \varpsi -T_{\mathfrak{B}} \vareta,\quad
\vPhi
=
\begin{pmatrix}
T_p \vareta \\
T_q \U
\end{pmatrix}.
$$
As above, we begin by computing that 
$\vPhi$ satisfies 
\begin{equation*}
\left(\partial_{t}+ T_{V}\cdot\partialx J_\eps \right)  \vPhi +
\begin{pmatrix}
0 & - T_{\gamma} J_\eps \\
T_{\gamma} J_\eps  & 0 \end{pmatrix}
\vPhi
= \F,
\end{equation*}
with $\F=\F_{1}+\F_{2}+\F_{3}$ where
\begin{equation*}
\begin{aligned}
&\F_{1}=  
\begin{pmatrix}
T_{p}F_1  \\
T_{q}F_2 
\end{pmatrix},\\
&\F_{2}= \begin{pmatrix}
0 & -    ( T_p T_\lambda T_{1/q}J_\eps- T_\gamma J_\eps) \\
(T_q T_h T_\wp J_\eps -T_\gamma J_\eps)& 0 
 \end{pmatrix} \vPhi,\\
&\F_{3}=\left[ \partial_{t}+ T_{V}\cdot\partialx J_\eps , 
\begin{pmatrix}
T_{p} & 0 \\ 0 & T_{q}\end{pmatrix}
\begin{pmatrix}
I&0\\
-T_{\mathfrak{B}} & I 
\end{pmatrix}
\right] \begin{pmatrix}\vareta\\ \varpsi\end{pmatrix}.
\end{aligned}
\end{equation*}
Then we find that
\begin{equation*}
\lA \F\rA_{L^\infty(0,T;H^{\sigma}\times H^{\sigma})} \le 
C\left(  \lA (\eta,\psi)\rA_{X^s(T)}\right)\lA (\vareta,\varpsi)\rA_{X^\sigma(T)}
+\lA F \rA_{X^\sigma(T)},
\end{equation*}
for some non-decreasing function $C$ independent of $\eps$.

ii) Next, we introduce the symbol
\begin{equation*}
\beta\defn \left( \gamma^{(3/2)}\right)^{\frac{2\sigma}{3}} \in \Sigma^\sigma.
\end{equation*}
As above, we find that
\begin{align*}
\lA \left[T_{\beta},T_{\gamma}\right]J_\eps \rA_{H^\sigma\rightarrow L^2} &\le C(\lA (\eta,\psi)\rA_{X^{s}(T)}),\\
\lA T_\gamma \left[T_{\beta},J_\eps \right] \rA_{H^\sigma\rightarrow L^2} &\le C(\lA (\eta,\psi)\rA_{X^{s}(T)}), \\
\lA \left[T_{\beta},T_{V}\cdot\partialx J_\eps \right]\rA_{H^\sigma\rightarrow L^2}
&\le C(\lA (\eta,\psi)\rA_{X^{s}(T)}),\\
\lA \left[T_{\beta},\partial_t\right]\rA_{H^\sigma\rightarrow L^2}
&\le C(\lA (\eta,\psi)\rA_{X^{s}(T)}),
\end{align*}
for some non-decreasing function $C$ independent of $\eps\in [0,1]$. 
Therefore, by commuting the equation \eqref{Ya} with $T_\beta$, we find that 
$$
\vvarphi\defn T_\beta \vPhi
$$
satisfies
\begin{equation*}
\left(\partial_{t}+ T_{V}\cdot\partialx J_\eps \right) \vvarphi +
\begin{pmatrix}
0 & - T_{\gamma} J_\eps \\
T_{\gamma} J_\eps & 0 \end{pmatrix}
\vvarphi
= \F',
\end{equation*}
with
\begin{equation*}
\lA \F'\rA_{L^{\infty}(0,T;L^{2}\times L^{2})} \le 
C(\lA (\eta,\psi)\rA_{X^{s}(T)})
\| (\vareta,\varpsi)\|_{X^{\sigma}(T)}
+\lA F \rA_{X^\sigma(T)},
\end{equation*}
for some non-decreasing function $C$ independent of $\eps\in [0,1]$.

iii) Therefore, we obtain that for all $t\in [0,T]$, $\lA\vvarphi(t)\rA_{L^2\times L^2}^2-\lA\vvarphi(0)\rA_{L^2\times L^2}^2$ is bounded by
\begin{equation*}
C(\lA (\eta,\psi)\rA_{X^{s}(T)})\int_0^T \left( \lA \vvarphi(t')\rA_{L^2\times L^2}^2+\lA F'(t')\rA_{L^2\times L^2}^2\right)  dt'
\end{equation*}
which immediately implies that $\lA\vvarphi\rA_{L^\infty(0,T;L^2\times L^2)}$ is bounded by
$$
\lA\vvarphi(0)\rA_{L^2\times L^2}+
TC(\lA (\eta,\psi)\rA_{X^{s}(T)})\lA\vvarphi\rA_{L^\infty(0,T;L^2\times L^2)} 
+T \lA F\rA_{X^\sigma(T)}.
$$
Once this is granted, we end the proof as above.
\end{proof}

\section{Cauchy problem}
In this section we conclude the proof of Theorem~\ref{theo:Cauchy}.
We divide the proof into two independent parts: (a) Existence; (b) Uniqueness. 
We shall prove the uniqueness by an estimate for the difference of two solutions. 
With regards to the existence, as mentioned above, we shall obtain solutions to the system \eqref{system}
as limits of solutions to the approximate systems \eqref{A3} which were studied in the previous section. To do that, 
we shall begin by proving that:
\begin{enumerate}
\item For any $\eps>0$, the approximate systems \eqref{A3} are well-posed locally in time (ODE argument).
\item The solutions $(\eta_\eps,\psi_\eps)$ of the approximate system \eqref{A3} are uniformly 
bounded with respect to $\eps$ (by means of the uniform estimates in Proposition~\ref{ueps}).
\end{enumerate}

The next task is to show that the functions $\{(\eta_\eps,\psi_\eps)\}$ converge to a limit $(\eta,\psi)$ which is 
a solution of the water-waves system~\eqref{system}. To do this, one cannot apply standard compactness results since 
the Dirichlet-Neumann operator is not a local operator, at least with our very general geometric assumptions  (notice however that in the case of infinite depth or flat bottom, one can show this local property). 
To overcome this difficulty, as in \cite{LannesJAMS}, we shall prove that
\begin{enumerate}\setcounter{enumi}{2} 
\item The solutions $(\eta_\eps,\psi_\eps)$ form a Cauchy sequence in an appropriate bigger space (by an estimate of the difference 
of two solutions $(\eta_\eps,\psi_\eps)$ and $(\eta_{\eps'},\psi_{\eps'})$). 
\end{enumerate}
We next deduce that 
\begin{enumerate}\setcounter{enumi}{3} 
\item $(\eta,\psi)$ is a solution to \eqref{system}.
\end{enumerate}
The next task is to prove that 
\begin{enumerate}\setcounter{enumi}{4} 
\item $(\eta,\psi)\in C^{0}([0,T];H^{s+\mez}(\xR^d)\times H^{s}(\xR^d))$.
\end{enumerate}
Notice that, as usual once we know the uniqueness of the limit system, one can assert that the whole family $\{(\eta_\eps,\psi_\eps)\}$ converges 
to $(\eta,\psi)$.

\smallbreak

Clearly, to achieve these various goals, the main part of the work was already accomplished in the previous section. 

\subsection{Existence}
\begin{lemm}
For all $(\eta_0,\psi_0)\in H^{s+\mez}(\xR^d)\times H^s(\xR)$, and any $\eps>0$, the Cauchy problem
\begin{equation*}
\left\{
\begin{aligned}
&\left( \partial_t+   T_{V} \cdot\partialx J_\eps 
+
\mathcal{L}^\eps\right)\begin{pmatrix}\eta \\ \psi\end{pmatrix}
=f( J_\eps \eta,J_\eps \psi),\\
& (\eta,\psi)\arrowvert_{t=0}=(\eta_{0},\psi_0).
\end{aligned}
\right.
\end{equation*} 
has a unique maximal solution $(\eta_\eps,\psi_\eps)\in C^{0}([0,T_\eps[;H^{s+\mez}(\xR^d)\times H^{s}(\xR^d))$
\end{lemm}
\begin{proof}
Write \eqref{A3} in the compact form
\begin{equation}\label{edo}
\partial_t Y= \mathcal{F}_\eps(Y),\quad Y\arrowvert_{t=0}=Y_{0}.
\end{equation}
Since $J_\eps$ is a smoothing operator, \eqref{edo} is an ODE with values in a Banach space for any $\eps >0$. Indeed, 
it is easily checked that the function $\mathcal{F}_\eps$ is $C^1$ from $H^{s+\mez}(\xR^d)\times H^{s}(\xR^d)$ to itself 
(the only non trivial terms come from the Dirichlet-Neumann operator, 
whose regularity follows from Proposition~\ref{Lannes1}). 
The Cauchy Lipschitz theorem then implies the desired result. 
\end{proof}

\begin{lemm}
There exists $T_0>0$ such that $T_\eps\ge T_0$ for all $\eps \in ]0,1]$ and such that 
$\{ (\eta_\eps,\psi_\eps) \}_{\eps\in ]0,1]}$ is bounded in 
$C^{0}([0,T_0]; H^{s+\mez}(\xR^d)\times H^{s}(\xR^d))$. 
\end{lemm}
\begin{proof}The proof is standard.
For $\eps\in]0,1]$ and $T<T_\eps$, set
$$
M_\eps (T)\defn \lA (\eta_\eps,\psi_\eps)\rA_{L^{\infty}(0,T;H^{s+\mez}\times H^{s})}.
$$
Notice that automatically $(\eta_\eps,\psi_\eps) \in C^{1}([0,T_\eps[;H^{s+\mez}(\xR^d)\times H^{s}(\xR^d))$, so that one can apply 
Proposition~\ref{ueps} to obtain that there exists a continuous function $C$ such that, for all $\eps\in ]0,1]$ and all $T<T_\eps$
\begin{equation}\label{estuni}
M_\eps(T)\le C(M_0)+TC(M_\eps(T)),
\end{equation}
where we recall that $M_0=\lA (\eta_0,\psi_0)\rA_{H^{s+\mez}\times H^s}$. Let us set 
$M_1=2C(M_0)$ and choose $0<T_{0}\le 1$ small enough such that
$C(M_0)+T_0 C(M_1)< M_1$. We claim that
\begin{equation*}
  M_\varepsilon(T) <M_{1},\quad \forall T\in I\defn [0,\min\{T_{0},T_\varepsilon\}[.
\end{equation*}
Indeed, since $M_{\varepsilon}(0)=M_0<M_1$, assume that there exists $T\in I$ such that $M_\eps (T)=M_1$ then 
$$
M_1=M_\eps(T)\le C(M_0)+T C(M_\eps(T))\le C(M_0) + T_0 C(M_1)<M_1,
$$
hence the contradiction.

The continuation principle for ordinary differential equations then implies that $T_\varepsilon\geqslant T_0$ for 
all $\varepsilon\in ]0,1]$, and we have
$$
\sup_{\varepsilon\in]0,1]} \sup_{T\in[0,T_0]} M_\varepsilon(T) \le M_1.
$$
This completes the proof.
\end{proof}

\begin{lemm}\label{CSL}
Let $s'<s-\tdm$. 
Then there exists $0<T_1\le T_0$ such that $\{ (\eta_\eps,\psi_\eps) \}_{\eps\in ]0,1]}$ is a Cauchy sequence in 
$C^{0}([0,T_1]; H^{s'+\mez}(\xR^d)\times H^{s'}(\xR^d))$. 
\end{lemm}
\begin{proof}The proof is sketched in \S\ref{CSp} below.
\end{proof}

Then, as explains in the introduction to this section, the existence of a classical solution follows from standard arguments.

\subsection{Uniqueness}\label{sec.unique}
To complete the proof of Theorem~\ref{theo:Cauchy}, it remains to prove the uniqueness. 

\begin{lemm}\label{LUs}
Let 
$(\eta_j,\psi_j)\in C^0([0,T];H^{s+\mez}(\xR^d)\times H^s(\xR^d))$, $j=1,2$, be two solutions of system \eqref{system} 
with the same initial data, and such that the assumption $H_t$ is satisfied for all $t\in [0,T]$. Then 
$ (\eta_1,\psi_1)=(\eta_2,\psi_2)$. 
\end{lemm}

As we shall see, the proof of Lemma~\ref{LUs} requires a lot of care. 


Recall (see \S\ref{s.5.1}) that $(\eta,\psi)$ solves \eqref{system} if and only if
\begin{equation*}
\left( \partial_t+ T_V\cdot \partialx 
+
\mathcal{L}\right) \begin{pmatrix}\eta \\ \psi\end{pmatrix}
=f(\eta,\psi),
\end{equation*}
with
\begin{equation}\label{defiLf}
\mathcal{L}\defn \begin{pmatrix}
I & 0 \\
T_\mathfrak{B} & I \end{pmatrix}
\begin{pmatrix}
0 & -T_{\lambda}\\
T_{h} & 0 \end{pmatrix}
\begin{pmatrix}
I & 0 \\
-T_\mathfrak{B} & I \end{pmatrix},\quad
f(\eta,\psi)\defn\begin{pmatrix}
I & 0 \\
T_\mathfrak{B} & I \end{pmatrix}\begin{pmatrix} f^1 \\ f^2 \end{pmatrix}.
\end{equation}
where
\begin{equation*}
\begin{aligned}
f^1&=G(\eta)\psi -\big\{ T_{\lambda} ( \psi -T_\mathfrak{B} \eta) - T_V\cdot\partialx \eta \big\},\\
f^2&=\mez \la \partialx \psi\ra^2 +\mez \frac{\left( \partialx \eta \cdot \partialx \psi + G(\eta)\psi \right)^2}{1+\la \partialx\eta\ra^2} 
+ H(\eta) \\
&\quad + T_V \partialx \psi - T_\mathfrak{B}T_V\cdot \partialx  \eta  -T_{\mathfrak{B}}G(\eta)\psi 
+ T_{h}\eta -g\eta.
\end{aligned}
\end{equation*}

Introduce the notation
\begin{equation}\label{BjVj}
\mathfrak{B}_j = \frac{\partialx\eta_j\cdot\partialx\psi_j+G(\eta_j)\psi_j}{1+\la \partialx\eta_j\ra^2},\quad
V_j=\partialx\psi_j -\mathfrak{B}_{j}\partialx\eta_j,
\end{equation}
and denote by $\lambda_j, h_j$ the symbols obtained by replacing $\eta$ by $\eta_j$ in \eqref{dmu10}, \eqref{dh21} respectively. 
Similarly, denote by $\mathcal{L}_1$ the operator obtained by replacing $(\mathfrak{B},\lambda,h)$ by $(\mathfrak{B}_1,\lambda_1,h_1)$ in 
\eqref{defiLf}. To prove the uniqueness, the main technical lemma is the following.

\begin{lemm}\label{LU}The differences $\delta\eta\defn \eta_1-\eta_2$ and $\delta\psi\defn \psi_1-\psi_2$ satisfy a system of the form
\begin{equation*}
\left( \partial_t+ T_{V_1}\cdot \partialx 
+
\mathcal{L}_1\right) \begin{pmatrix}\delta\eta \\ \delta\psi\end{pmatrix}
=f,
\end{equation*}
for some remainder term such that
$$
\lA f\rA_{L^\infty(0,T;H^{s-1}\times H^{s-\tdm})}
\le C (M_1,M_2)N,
$$
where 
$$
M_j\defn \lA (\eta_j,\psi_j)\rA_{L^\infty(0,T;H^{s+\mez}\times H^s)},~
N\defn \lA (\delta\eta,\delta\psi)\rA_{L^\infty(0,T;H^{s-1}\times H^{s-\tdm})}.
$$
\end{lemm}

Assume this technical lemma for a moment, and let us deduce the desired result: $ (\eta_1,\psi_1)=(\eta_2,\psi_2)$. 
To see this we use our previous analysis. Introducing
$$
\delta U\defn \delta\psi-T_{\mathfrak{B}_1}\delta\eta=
\psi_1-\psi_2-T_{\mathfrak{B}_1}(\eta_1-\eta_2),
$$
and
$$
\delta\Phi\defn \begin{pmatrix} T_{p_1}\delta\eta \\ T_{q_1}\delta U \end{pmatrix},
$$
we obtain that $\delta\Phi$ solves a system of the form
\begin{equation*}
\partial_{t}\delta\Phi+T_{V_1}\cdot\partialx\delta\Phi  +
\begin{pmatrix} 0 &  - T_{\gamma_1} \\ T_{\gamma_1} & 0\end{pmatrix}\delta\Phi  
= F
\end{equation*}
with
$$
\lA F\rA_{L^\infty(0,T;H^{s-\tdm}\times H^{s-\tdm})}
\le C(M_1,M_2)N.
$$
Then it follows from the estimate \eqref{estuni2} applied with
$$
\eps=0,\quad \sigma=s-\tdm,\quad \vareta=\delta\eta,\quad \varpsi=\delta\psi,
$$
that $N$ 
satisfies and estimate of the form
$$
N\le T C(M_1,M_2) N. 
$$
By chosing $T$ small enough, this implies $N=0$ which is the desired result. Now clearly 
it suffices to prove the uniqueness for $T$ small enough, so that this completes the proof. 

It remains to prove Lemma~\ref{LU}. To do this, we begin with the following lemma.

\begin{lemm}
We have
\begin{align*}
\lA V_1 - V_2\rA_{H^{s-\frac{5}{2}}} &\le C \lA  (\delta\eta,\delta\psi)\rA_{H^{s-1}\times H^{s-\tdm}},\\
\lA \mathfrak{B}_1 - \mathfrak{B}_2\rA_{H^{s-\frac{5}{2}}} &\le C \lA  (\delta\eta,\delta\psi)\rA_{H^{s-1}\times H^{s-\tdm}},
\end{align*}
\begin{align*}
\sup_{\la\xi\ra=1} \bigl(|\partial_\xi^\alpha \big(\lambda_1^{(1)}(\cdot,\xi) -\lambda_{2}^{(1)}(\cdot,\xi)\big)| 
+ \| \partial_\xi^\alpha \big(\lambda_1^{(0)}(\cdot,\xi) -\lambda_{2}^{(0)}(\cdot,\xi)\big)\|_{H^{s-3}}\bigr)
&\le C \| \delta\eta\|_{H^{s-1}},\\
\sup _{\la\xi\ra=1} \bigl(| \partial_\xi^\alpha \big( h_1^{(2)}(\cdot,\xi) -h_{2}^{(2)}(\cdot,\xi)\big)|
+\| \partial_\xi^\alpha \big( h_1^{(1)}(\cdot,\xi) -h_{2}^{(1)}(\cdot,\xi)\big)\|_{H^{s-3}}\bigr)
&\le C \lA \delta\eta\rA_{H^{s-1}}
\end{align*}
for all $\alpha\in\xN^d$ and some constant $C$ depending only on $M_1$, $M_2$ and $\alpha$.
\end{lemm}
\begin{proof}
The last two estimates are obtained from the product rule in Sobolev spaces (using similar arguments as in the end of 
the proof of Lemma~\ref{comdtp}). With regards to the first two estimates, notice that, 
by definition of $\mathfrak{B_j},V_j$ (see \eqref{BjVj}), to prove them the only non trivial point is to prove that
$$
\lA G(\eta_1)\psi_1-G(\eta_2)\psi_2\rA_{H^{s-\frac{5}{2}}}\le C \lA  (\delta\eta,\delta\psi)\rA_{H^{s-1}\times H^{s-\tdm}}.
$$
Indeed, setting $\eta_t=t\eta_1 +(1-t)\eta_2$ we have
$$
G(\eta_1)\psi_1-G(\eta_2)\psi_2=G(\eta_1)\delta\psi+\int_0^1 dG(\eta_t)\psi_2 \cdot\delta\eta\,dt
=:A+B.
$$
It follows from Proposition~\ref{estDN} that 
$$
\lA A \rA_{H^{s-\frac{5}{2}}}\le C(M_1)\lA \delta\psi\rA_{H^{s-\tdm}}.
$$
Now thanks to Proposition~\ref{Lannes1} we can write
$$
B= -\int_0^1 \left[ G(\eta_t)(\mathfrak{B}_t\delta \eta)+\cnx (V_t \delta \eta)\right]\, dt,
$$
where $\mathfrak{B}_t=\mathfrak{B}(\eta_t,\psi_2)$, $V=V(\eta_t,\psi_2)$. Using again Proposition~\ref{estDN} we obtain
\begin{equation}\label{doneabove}
\lA B\rA_{H^{s-\frac{5}{2}}} \le C(M_1,M_2)\lA \delta\eta\rA_{H^{s-\tdm}},
\end{equation}
which completes the proof.
\end{proof}

\begin{coro}
We have
\begin{align*}
&\lA T_{V_1-V_2} \cdot\partialx \eta_2\rA_{H^{s-1}} \le C  \lA  (\delta\eta,\delta\psi)\rA_{H^{s-1}\times H^{s-\tdm}},\\
&\lA T_{V_1-V_2}\cdot\partialx \psi_2\rA_{H^{s-\tdm}}\le C  \lA  (\delta\eta,\delta\psi)\rA_{H^{s-1}\times H^{s-\tdm}}, \\
&\lA T_{\lambda_1-\lambda_2}\psi_2 \rA_{H^{s-1}}\le C  \lA  (\delta\eta,\delta\psi)\rA_{H^{s-1}\times H^{s-\tdm}},\\
&\lA T_{h_1-h_2}\eta_2 \rA_{H^{s-\tdm}}\le C  \lA  (\delta\eta,\delta\psi)\rA_{H^{s-1}\times H^{s-\tdm}},
\end{align*}
for some constant $C$ depending only on $M_1$ and $M_2$.
\end{coro}
\begin{proof}
According to Lemma~\ref{negmu}, we have
$$
\lA T_a u\rA_{H^\mu}\les \lA a\rA_{H^{\frac{d}{2}-\mez}}\lA u\rA_{H^{\mu+\mez}}.
$$
so using the previous lemma we obtain the first two estimates. The last two estimates comes from the bounds for 
$\lambda_1-\lambda_2$ and $h_1-h_2$ and Proposition~\ref{2d21} (again it suffices to apply 
the usual operators norm estimate \eqref{esti:quant1} for $s>3+d/2$). 
\end{proof}

Similarly, we obtain that, for any $u\in H^{s+\mez}$, 
$$
\lA T_{\mathfrak{B}_1-\mathfrak{B}_2} u\rA_{H^{s}} \le C  \lA  (\delta\eta,\delta\psi)\rA_{H^{s-1}\times H^{s-\tdm}}
\lA u\rA_{H^{s+\mez}}.
$$
Therefore, to prove Lemma~\ref{LU}, it remains only to estimate the difference
$$
f(\eta_1,\psi_1)-f(\eta_2,\psi_2),
$$
where $f(\eta,\psi)$ is defined in \eqref{defiLf}. To do this, the most delicate part 
is to obtain an estimate for
$$
f^1(\eta_1,\psi_1)-f^1(\eta_2,\psi_2),
$$
where recall the notation
\begin{equation}\label{t1a}
f^1(\eta,\psi)=G(\eta)\psi-\left\{ T_{\lambda} \bigl(\psi-T_{\mathfrak{B}}\eta \bigr) - T_{V} \cdot\partialx \eta \right\}.
\end{equation}
We claim that 
$$
\lA f^1(\eta_1,\psi_1)-f^1(\eta_2,\psi_2)\rA_{H^{s-1}}\le C(M_1,M_2) \lA (\delta\eta,\delta\psi)\rA_{H^{s-1}\times H^{s-\tdm}}.
$$
To prove this claim, we shall prove an estimate for the partial derivative of $f^1(\eta,\psi)$ with respect to $\eta$ 
(since $f^1(\eta,\psi)$ is linear with respect to $\psi$, the corresponding result for the partial derivative with respect to $\psi$ 
is easy). 
Let $(\eta,\psi)\in H^{s+\mez}(\xR^d)\times H^s (\xR^d)$ (again we forget the time dependence) and consider 
$\dot\eta\in H^{s-1}(\xR^d)$. Introduce the notation
$$
d_\eta f^1(\eta,\psi)\cdot\dot{\eta}=\lim_{\eps\rightarrow 0} \frac{1}{\eps}
\left( f(\eta+\eps\dot\eta,\psi)-f(\eta,\psi)\right).
$$
Then, to complete the proof of the uniqueness, it remains only to prove the following technical lemma.
\begin{lemm}Let $s>2+d/2$. Then, 
for all $(\eta,\psi)\in H^{s+\mez}(\xR^d)\times H^{s}(\xR^d)$, and for all $\dot\eta\in H^{s+\mez}(\xR)$,
$$
\lA d_\eta f^1(\eta,\psi)\cdot\dot{\eta}\rA_{H^{s-1}}\le C \lA \dot\eta\rA_{H^{s-1}},
$$
for some constant $C$ which depends only on the $H^{s+\mez}(\xR^d)\times H^s (\xR^d)$-norm of $(\eta,\psi)$.
\end{lemm}
\begin{rema}
The assumption $\dot\eta\in H^{s+\mez}(\xR)$ ensures that $d_\eta f^1(\eta,\psi)\dot\eta$ is well defined. However, 
of course, a key point is that we estimate the latter term in $H^{s-1}$ by means of only the 
$H^{s-1}$ norm of $\dot\eta$.
\end{rema}
\begin{proof}
To prove this estimate we begin by computing $d_\eta f^1(\eta,\psi)\dot\eta$. Given a coefficient $c=c(\eta,\psi)$ we use the notation
$$
\dot c=\lim_{\eps\rightarrow 0} \frac{1}{\eps}
\left( c(\eta+\eps\dot\eta,\psi)-c(\eta,\psi) \right).
$$
Using this notation for $\dot\lambda,\dot{\mathfrak{B}},\dot V$, we have
\begin{equation}\label{t1A}
\begin{aligned}
d_\eta f^1(\eta,\psi)\cdot\dot{\eta}&=-G(\eta)(\mathfrak{B}\dot\eta)-\cnx (V \dot\eta)\\
&\quad-\left\{ T_{\dot\lambda} \bigl(\psi-T_{\mathfrak{B}}\eta \bigr)  
-T_{\lambda} T_{\dot{\mathfrak{B}}}\eta- T_{\lambda} T_{\mathfrak{B}}\dot\eta  - T_{\dot V} \cdot\partialx \eta
- T_{V} \cdot\partialx \dot\eta \right\},
\end{aligned}
\end{equation}
We split the right-hand side into four terms 
(three of which are easy to estimate, whereas the last one requires some care): set
\begin{align*}
I_1&=V\cdot\partialx \dot \eta -T_V \cdot\partialx \dot\eta,\\
I_2&=- T_{\dot\lambda} \bigl(\psi-T_{\mathfrak{B}}\eta \bigr) ,\\
I_3&=-T_{\lambda} T_{\dot{\mathfrak{B}}}\eta,\\
I_4&=-G(\eta)(\mathfrak{B}\dot\eta)-(\cnx V)\dot\eta +T_{\lambda} T_{\mathfrak{B}}\dot\eta.
\end{align*}
To estimate $I_1$, we use that, for all function $a\in H^{s_0}(\xR^d)$ with $s_0>1+d/2$, we have
$$
\lA a u-T_a u\rA_{H^{\mu+1}}\le K \lA a\rA_{H^{s_0}}\lA u \rA_{H^\mu},
$$
whenever $u\in H^\mu(\xR^d)$ for some $0\le \mu\le s_0-1$. By applying this estimate with $s_0=s-1$, we obtain
$$
\lA I_1\rA_{H^{s-1}}=\lA (V-T_V)\cdot\partialx\dot\eta\rA_{H^{s-1}}\les \lA V\rA_{H^{s-1}}  
\lA \partialx\dot\eta\rA_{H^{s-1-1}}\le C \lA \dot\eta\rA_{H^{s-1}}.
$$
With regards to the second term, we use the arguments in the proof of 
Proposition~\ref{2d21} (notice that here, our symbol $\dot\lambda$ has not exactly the form~\eqref{eq.symbol}, but rather 
$$
F(\nabla \eta, \xi) \nabla \dot\eta+ G( \nabla \eta, \xi ) \nabla^2 \eta + K(\nabla \eta,\xi) \nabla \dot \eta \nabla^2 \eta 
$$
and the proof of Proposition~\ref{2d21} applies.) We obtain
$$
\lA I_2\rA_{H^s-1}\le C  \lA \dot\eta\rA_{H^{s-1}}.
$$
To estimate $I_3$, notice that \eqref{esti:quant1} implies that
$$
\lA I_3\rA_{H^{s-1}}\les M^{1}_0(\lambda)\lA T_{\dot{\mathfrak{B}}} \eta\rA_{H^{s-1+1}}\le 
C \lA T_{\dot{\mathfrak{B}}} \eta\rA_{H^{s-1+1}}.
$$
Next, using the general estimate 
$$
\lA T_a u\rA_{H^\mu}\le K \lA a\rA_{H^{\frac{d}{2}-m}}\lA u\rA_{H^{\mu+m}},
$$
we conclude
$$
\lA I_3\rA_{H^{s-1}}\le C \big\lVert \dot{\mathfrak{B}} \big\rVert_{H^{s-\frac{5}{2}}}\lA \eta\rA_{H^{s+\mez}}.
$$
Therefore, the desired result for $I_3$ will follow from the claim
$$
\big\lVert \dot{\mathfrak{B}} \big\rVert_{H^{s-\frac{5}{2}}}\le C \lA \dot\eta\rA_{H^{s-1}}.
$$
To see this, the only non-trivial point is to bound
$dG(\eta)\psi\cdot\dot{\eta}$, which was precisely done above (cf \eqref{doneabove}).

It remains to estimate $I_4$, which is the most delicate part. Indeed, one cannot estimate the terms separately, and we have to use a cancellation which comes 
from the identity $G(\eta)\mathfrak{B}=-\cnx V$ (see Lemma~\ref{cancellation}).

It follows from Proposition~\ref{prop:csystem3} that
$$
G(\eta)(\mathfrak{B}\dot\eta)=T_{\lambda^{(1)}} (\mathfrak{B}\dot\eta)
 +F(\eta,\mathfrak{B}\dot\eta),
\quad
G(\eta)\mathfrak{B}=T_{\lambda^{(1)}} \mathfrak{B}
 +F(\eta,\mathfrak{B}),
$$
where 
$$
\lA F(\eta,\mathfrak{B}\dot\eta)\rA_{H^{s-1}}
\le C\lA \dot\eta\rA_{H^{s-1}}, \quad \lA F(\eta,\mathfrak{B})\rA_{H^{s-1}}
\le C.
$$
Therefore
\begin{align*}
I_4&=-G(\eta)(\mathfrak{B}\dot\eta)-(\cnx V)\dot\eta +T_{\lambda} T_{\mathfrak{B}}\dot\eta\\
&=-T_\lambda (\mathfrak{B}\dot\eta ) -F(\eta,\mathfrak{B}\dot\eta)
-\dot \eta \cnx V +T_{\lambda} T_{\mathfrak{B}}\dot\eta\\
&=-T_\lambda (\mathfrak{B}\dot\eta)
-F(\eta,\mathfrak{B}\dot\eta)
-T_{\dot\eta}\cnx V -(\dot \eta -T_{\dot\eta} )\cnx V +T_{\lambda} T_{\mathfrak{B}}\dot\eta\\
\intertext{and hence using $\cnx V=-G(\eta)\mathfrak{B}$}
I_4&=-T_\lambda (\mathfrak{B}\dot\eta) -F(\eta,\mathfrak{B}\dot\eta)
+T_{\dot\eta}G(\eta)\mathfrak{B} +(\dot \eta -T_{\dot\eta} )\cnx V +T_{\lambda} T_{\mathfrak{B}}\dot\eta\\
\intertext{and paralinearizing $G(\eta)\mathfrak{B}$ and gathering terms we conclude}
I_4&=-T_\lambda (\mathfrak{B}\dot\eta)-F(\eta,\mathfrak{B}\dot\eta)
+T_{\dot\eta}\Big(T_\lambda \mathfrak{B}+F(\eta,\mathfrak{B})\Big)+(\dot \eta -T_{\dot\eta} )\cnx V +T_{\lambda} T_{\mathfrak{B}}\dot\eta
\end{align*}
then commuting $T_{\dot\eta}$ and $T_\lambda$ 
we conclude that
$$
I_4=J_1+J_2,
$$
where
\begin{align*}
J_1&=-T_{\lambda^{(1)}}\Bigl( \mathfrak{B}\dot\eta -T_{\dot\eta}\mathfrak{B} - T_{\mathfrak{B}}\dot\eta\Bigr)\\
J_2&=-T_{\lambda^{(0)}}(\mathfrak{B}\dot\eta)+[T_{\dot\eta},T_\lambda]\mathfrak{B}+T_{\dot\eta} F(\eta,\mathfrak{B})+(\dot \eta -T_{\dot\eta} )\cnx V-F(\eta,\mathfrak{B}\dot\eta).
\end{align*}
Now both terms $J_1$ and $J_2$ are estimated using symbolic calculus (namely we estimate the first term by means of (ii) in Theorem~\ref{lemPa}; and 
we estimate $J_2$ by means of \eqref{esti:quant1}, \eqref{esti:quant2} and (ii) in Theorem~\ref{lemPa}).
\end{proof}

\subsection{Sketch of the proof of Lemma~\ref{CSL}}\label{CSp}
Let $0<\eps_1<\eps_2$ and consider two solutions 
$(\eta_{\eps_j},\psi_{\eps_j})\in C^{0}([0,T];H^{s+\mez}(\xR^d)\times H^{s}(\xR^d))$ of \eqref{A3}. 
Introduce the notation
\begin{equation}\label{CSBjVj}
\mathfrak{B}_{\eps_j} = \frac{\partialx\eta_{\eps_j}\cdot\partialx\psi_{\eps_j}+G(\eta_{\eps_j})\psi_{\eps_j}}{1+\la \partialx\eta_{\eps_j}\ra^2},\quad
V_{\eps_j}=\partialx\psi_{\eps_j} -\mathfrak{B}_{j}\partialx\eta_{\eps_j},
\end{equation}
and denote by $\lambda_j, h_j$ the symbols obtained by replacing $\eta$ by $\eta_{\eps_j}$ in \eqref{dmu10}, \eqref{dh21} respectively. 
Here, the main technical lemma is the following.

\begin{lemm}\label{CSLU}
Let $0<\eps_1<\eps_2$, consider $s'$ such that
$$
\mez+\frac{d}{2}<s'<s-\tdm,
$$
and set 
$$
a=s-\tdm-s'.
$$
Then the differences $\delta\eta\defn \eta_{\eps_1}-\eta_{\eps_2}$ and $\delta\psi\defn \psi_{\eps_1}-\psi_{\eps_2}$ 
satisfy a system of the form
\begin{equation}\label{CSA2}
\left( \partial_t+ T_{V_{\eps_1}}\cdot \partialx J_{\eps_1}
+
\mathcal{L}^{\eps_1}\right) \begin{pmatrix}\delta\eta \\ \delta\psi\end{pmatrix}
=f,
\end{equation}
for some remainder term such that
$$
\lA f\rA_{X^{s'}(T)}
\le C \left\{ \lA (\delta\eta,\delta\psi)\rA_{X^{s'}(T)} + (\eps_2-\eps_1)^a\right\},
$$
for some constant $C$ depending only on $\sup_{\eps\in ]0,1]}\lA (\eta_\eps,\psi_\eps)\rA_{X^{s}(T)}$.
\end{lemm} 
To prove Lemma~\ref{CSLU}, we proceed as in the previous paragraph. The only difference is that 
we use the fact that
$$
\lA J_{\eps_2}-J_{\eps_1}\rA_{H^\mu \rightarrow H^{\mu-a}}
\le C (\eps_2-\eps_1)^a.
$$

Now, since for $t=0$ we have $\delta\eta=0=\delta\psi$, 
it follows from  Lemma~\ref{CSLU} and \eqref{estuni2} applied with
$$
\sigma=s',\quad \eps=\eps_1,\quad \vareta=\delta\eta,\quad \varpsi=\delta\psi,
$$
that $N$ 
satisfies and estimate of the form
$$
N\le T C \left\{ N + (\eps_2-\eps_1)^a\right\}. 
$$
By chosing $T$  and $\eps_2$ small enough, this implies $N=O((\eps_2-\eps_1)^a)$, which proves 
Lemma~\ref{CSL}. 

\smallbreak

\section{The Kato smoothing effect}\label{s4}

We consider a given solution 
$(\eta,\psi)$ of ~\eqref{system} on the time interval $[0,T]$ with $0<T<+\infty$, 
such that the assumption $H_t$ is satisfied for all $t\in [0,T]$ and such that
$$
(\eta,\psi)\in C^0\big([0,T];H^{s+\mez}(\xR)\times H^{s}(\xR)\big),
$$
for some $s>\frac{5}{2}$. In this section we prove Theorem~\ref{theo:main}. Namely, we shall prove that
$$
\L{x}^{-\mez-\delta}
(\eta,\psi)\in L^2\big(0,T;H^{s+\frac{3}{4}}(\xR)\times H^{s+\frac{1}{4}}(\xR)\big),
$$
for any $\delta>0$.

\subsection{Reduction to an $L^2$ estimate}
Let $\Phi_1,\Phi_2$ be as defined in Corollary~\ref{psym}. Then the complex-valued unknown 
$\Phi=\Phi_{1} + i \Phi_{2}$ satisfies a {\em scalar equation} of the form
\begin{equation}\label{PhiF}
\partial_{t}\Phi   +T_V\partial_x \Phi  +i T_\gamma \Phi =F,
\end{equation}
with  $F=F_{1}+iF_{2}\in L^{\infty}(0,T;H^{s}(\xR^d))$. 
Recall from Proposition~\ref{A2D} and \eqref{dh21} that, if $d=1$ then 
$$
\lambda^{(1)}=\la\xi\ra,\quad \lambda^{(0)}=0, \quad h^{(2)}= c^{2}\la\xi\ra^2,
$$
with
$$
c=(1+|\partial_x\eta|^2)^{-\tq}.
$$ 
Therefore, directly from the definition of $\gamma$ (cf Proposition~\ref{key}), notice that if $d=1$ then $\gamma$ simplifies to
\begin{equation*}
\gamma =  
c \la \xi\ra^{\tdm}-\frac{3i}{4}\xi\la \xi\ra^{-\mez}\partial_{x} c,
\end{equation*}
and hence modulo an error term of order $0$, $T_\gamma$ is given by $\la D_{x}\ra^{\frac{3}{4}} T_{c} \la D_{x}\ra^{\frac{3}{4}}$.

In this paragraph we shall prove that one can deduce Theorem~\ref{theo:main} from the following proposition.

\begin{prop}\label{psmooth}
Assume that $\varphi\in C^0([0,T];L^2(\xR))$ satisfies
\begin{equation*}
\partial_{t}\varphi   +T_V \partial_{x} \varphi  + iT_\gamma \varphi  =f,
\end{equation*}
with $f \in L^{1}(0,T;L^{2}(\xR))$. Then, for all $\delta>0$,
$$
\L{x}^{-\mez-\delta}\varphi \in L^{2}(0,T; H^{\frac{1}{4}}(\xR)).
$$
\end{prop}
We postpone the proof of Proposition~\ref{psmooth} to the next paragraph. 

The fact that one can deduce Theorem~\ref{theo:main} from the above proposition, 
though elementary, contains the idea that one 
simplify hardly all the nonlinear analysis by means of paradifferential calculus. 
 
\begin{proof}[Proof of Theorem~\ref{theo:main} given Proposition~\ref{psmooth}] 
As in the proof of Proposition~\ref{ueps} (cf \S\ref{sueps}), with
\begin{equation}\label{defibetabis}
\beta\defn c^{\frac{2}{3}s}\la \xi\ra^s.
\end{equation}
we find that the commutators $[T_\beta,\partial_t]$, 
$\left[T_{\beta},T_{\gamma}\right]$ and 
$[T_{\beta},T_{V}\partial_{x}]$ are of order $\leo s$. 
Consequently, \eqref{PhiF} implies that
\begin{equation*}
\left(\partial_{t}   +T_V \partial_{x}  + i T_{\gamma} \right) T_{\beta}\Phi  \in L^{\infty}(0,T;L^{2}(\xR)),
\end{equation*}
and hence,
$$
\left(\partial_{t}   +T_V \partial_{x}  +i T_\gamma\right) T_{\beta}\Phi \in L^{1}(0,T;L^{2}(\xR)).
$$
Therefore it follows from Proposition~\ref{psmooth} that
\begin{equation*}
\L{x}^{-\mez-\delta}T_{\beta}\Phi \in L^{2}(0,T; H^{\frac{1}{4}}(\xR)).
\end{equation*}
Since, by definition, $\Phi=T_p \eta + i T_q U$ where $T_p\eta$ and $T_q U$ are real valued functions, this yields
\begin{equation*}
\L{x}^{-\mez-\delta}T_{\beta}T_p\eta  \in L^{2}(0,T; H^{\frac{1}{4}}(\xR)),\quad
\L{x}^{-\mez-\delta}T_{\beta}T_q U \in L^{2}(0,T; H^{\frac{1}{4}}(\xR)),
\end{equation*}
and hence, since $\psi=U+T_{\mathfrak{B}}\eta$,
\begin{equation}\label{beta14}
\L{x}^{-\mez-\delta}T_{\beta}T_p\eta  \in L^{2}(0,T; H^{\frac{1}{4}}(\xR)),\quad
\L{x}^{-\mez-\delta}T_{\beta}T_q \psi \in L^{2}(0,T; H^{\frac{1}{4}}(\xR)),
\end{equation}
Since $\L{x}^{-\mez-\delta}\in \Gamma^{0}_{\rho}(\xR^d)$ for any $\rho\ge 0$, Theorem~\ref{theo:sc} implies that 
the commutators
$$
\left[\L{x}^{-\mez-\delta},T_\beta T_p\right], \quad \left[\L{x}^{-\mez-\delta},T_\beta T_q\right]
$$
are of order $s-1/2$ and $s-1$, respectively. Therefore, directly from \eqref{beta14} and the assumption
\begin{equation*}
\eta \in C^{0}([0,T]; H^{s+\mez}(\xR)),\quad \psi \in C^{0}([0,T]; H^{s}(\xR)),
\end{equation*}
we obtain
$$
T_{\beta} T_p \L{x}^{-\mez-\delta}\eta \in L^{2}(0,T; H^{\frac{1}{4}}(\xR)),
\quad 
T_{\beta} T_q \L{x}^{-\mez-\delta}\psi \in L^{2}(0,T; H^{\frac{1}{4}}(\xR)).
$$
Now since $\beta,p,q$ are elliptic symbols of order $s$, $1/2$, $0$, respectively, 
we conclude (cf Remark~\ref{R3.9} or Proposition~\ref{2d22})
$$
\L{x}^{-\mez-\delta}\eta \in L^{2}(0,T; H^{s+\tq}(\xR)),\quad \L{x}^{-\mez-\delta}\psi \in L^{2}(0,T; H^{s+\frac{1}{4}}(\xR)).
$$
This proves Theorem~\ref{theo:main}.
\end{proof}

\subsection{Proof of Proposition~\ref{psmooth}}To complete the proof of Theorem~\ref{theo:main}, it remains 
to prove 
Proposition~\ref{psmooth}. To do so, following the Doi approach, we begin with the following lemma.

\begin{lemm}\label{DoiL}There exists a symbol 
$$
a=a(x,\xi) \in 
\dot \Gamma^0_{\infty}(\xR)\defn
\bigcap_{\rho\ge 0} \dot \Gamma^0_{\rho}(\xR),
$$
such that, for any 
$\delta>0$ one can find $K>0$ such that
$$
\left\{ c \la \xi\ra^{\tdm},a\right\}(t,x,\xi)\ge K \L{x}^{-1-\delta}\la \xi\ra^{\mez},
$$
for all $t\in [0,T]$, $x\in \xR$, $\xi\in \xR\setminus\{0\}$.
\end{lemm}
\begin{proof}
Consider an increasing function $\phi\in C^\infty(\xR)$ such that $0\le \phi\le 1$ and 
$$
\phi(y)=1 \text{ for }y\ge 2,\quad \phi(y)=0 \text{ for }y\le 1.
$$
Now with $\eps>0$ a small constant chosen later on we set
\begin{equation}\label{C1}
\left\{
\begin{aligned}
\phi_{+}(y)&=\phi\left(\frac{y}{\eps}\right), ~\phi_{-}(y)=\phi\left(-\frac{y}{\eps}\right)=\phi_{+}(-y),\\
\phi_{0}(y)&=1-\left(\phi_{+}(y)+\phi_{-}(y)\right).
\end{aligned}
\right.
\end{equation}
These are $C^\infty$-functions and we see easily that,
\begin{equation}\label{C2}
\left\{
\begin{aligned}
&\phi'_{0}+\phi'_{+}+\phi'_{-}=0,\\
&\phi_{+}(y)-\phi_{-}(y)= \sgn (y) \phi_{+}(\la y\ra) \quad (y\in \xR),\\
&\phi'_{+}(y)-\phi'_{-}(y)= \phi'_{+}(\la y\ra) \quad (y\in \xR),\\
&\phi'_{0}(y)=-\sgn (y) \phi'_{+}(\la y\ra) \quad (y\in\xR).
\end{aligned}
\right.
\end{equation}
Now we set
\begin{equation}\label{C3}
a_{0}(x,\xi)=x\frac{\xi}{\la \xi\ra}, \quad x\in \xR,\xi\neq 0,
\end{equation}
and we introduce
\begin{equation}\label{C4}
\psi_{0}(x,\xi)=\phi_{0}\left(\frac{a_{0}}{\L{x}}\right),\quad
\psi_{\pm}(x,\xi)=\phi_{\pm}\left(\frac{a_{0}}{\L{x}}\right).
\end{equation}
Let us note that on the support of $\psi_{+}$ (resp.\ $\psi_{-}$) we have $a_{0}\ge \eps \L{x}$ (resp.\ $a_{0}\le -\eps \L{x}$) and that 
$\la a\ra$ is a small function on $\xR\times \xR\setminus 0$. Finally we set
\begin{equation}\label{C5}
a(x,\xi) 
=\frac{a_{0}}{\L{x}}\psi_{0}(x,\xi)+\left[ 2\eps +f(\la a_{0}\ra)\right] \left( \psi_{+}(x,\xi)-\psi_{-}(x,\xi)\right),
\end{equation}
where
$$
f(\sigma)=\int_{0}^\sigma \frac{dy}{\L{y}^{1+\delta}}.
$$

We compute
$$
I\defn \{ c \la\xi\ra^{\tdm},a\} =\sum_{j=1}^5 I_{j},
$$
where
\begin{align*}
&I_{1}= \frac{ \left\{c \la\xi\ra^{\tdm},a_0\right\} }{ \L{x} } \psi_{0},\quad 
I_{2}=a_{0} \left\{ c \la\xi\ra^{\tdm}, \frac{1}{\L{x}} \right\} \psi_{0},\quad 
I_{1}= \frac{a_{0}}{\L{x}}\left\{c \la\xi\ra^{\tdm}, \psi_{0}\right\},\\
&I_{4}=\left\{ c \la\xi\ra^{\tdm},f\left(\la a_{0}\ra\right)\right\}  \left( \psi_{+}(x,\xi)-\psi_{-}(x,\xi)\right)\\
&I_{5}=\left[ 2\eps +f(\la a_{0}\ra)\right]  \left(\left\{ c \la\xi\ra^{\tdm},\psi_{+}\right\}-\left\{ c \la\xi\ra^{\tdm},\psi_{-}\right\}\right).
\end{align*}
Using the obvious identity $\partial_{\xi} (\xi/\la \xi\ra)=0$ for $\xi\neq 0$, we have
$$
\{ c \la\xi\ra^{\tdm},a_{0}\}=\left\{ c \la \xi\ra^{\tdm},x\frac{\xi}{\la\xi\ra}\right\}
=\frac{3}{2}c \frac{\xi}{\la \xi\ra}\la \xi\ra^\mez \frac{\xi}{\la \xi\ra}= \frac{3}{2} c \la \xi\ra^\mez. 
$$
Therefore
\begin{equation}\label{C6}
I_{1}=\frac{3}{2} c \frac{\la \xi\ra^{\mez}}{\L{x}}\psi_{0}.
\end{equation}
Now
$$
\left\{ c \la\xi\ra^{\tdm}, \frac{1}{\L{x}} \right\} = \partial_{\xi}\left(c \la\xi\ra^{\tdm} \right)\partial_{x}\left( \frac{1}{\L{x}} \right)
=-\frac{3}{2}c\frac{\xi}{\la \xi\ra}\la \xi\ra^\mez \frac{x}{\L{x}^3},
$$
so that
$$
I_{2}=-\frac{3}{2}c\frac{\xi}{\la \xi\ra}\la \xi\ra^\mez \frac{a_{0}}{\L{x}} \frac{x}{\L{x}^2}\psi_{0}.
$$
On the support of $\psi_{0}$ we have, by \eqref{C4} and \eqref{C1}, $\la a_{0}\ra\le \eps \L{x}$. It 
follows that 
\begin{equation}\label{C7}
\la I_{2}\ra\le \frac{3\eps c \la \xi\ra^\mez}{2 \L{x}}\psi_{0}.
\end{equation}
On the other hand we have by \eqref{C4} and \eqref{C2},
$$
I_{3}=\frac{a_{0}}{\L{x}}\left\{ c \la\xi\ra^{\tdm}, \frac{a_{0}}{\L{x}} \right\} \phi'_{0}\left(\frac{a_{0}}{\L{x}}\right)
= -\frac{a_{0}}{\L{x}}\left\{ c \la\xi\ra^{\tdm}, \frac{a_{0}}{\L{x}} \right\}\sgn \frac{a_{0}}{\L{x}}  \phi'_{+}\left(\frac{\la a_{0}\ra}{\L{x}}\right),
$$
which implies
\begin{equation}\label{C8}
I_{3}=-\frac{\la a_{0}\ra }{\L{x}}\left\{ c \la\xi\ra^{\tdm}, \frac{a_{0}}{\L{x}} \right\}\phi'_{+}\left(\frac{\la a_{0}\ra}{\L{x}}\right).
\end{equation}
Using \eqref{C4} and \eqref{C1} we see that
$$
I_{4}=f'\left( \la a_{0}\ra\right) \left\{ c \la\xi\ra^{\tdm} ,a_{0}\right\} \psi_{+} + 
f'\left(-a_{0}\right) \left\{ c \la\xi\ra^{\tdm} ,a_{0}\right\} \psi_{-}, 
$$
so by \eqref{C6} and \eqref{C3},
\begin{equation}\label{C9}
I_{4}= \frac{1}{\L{x}^{-1-\delta}} \frac{3}{2}c \la \xi\ra^{\mez} \left(\psi_{+}+\psi_{-}\right).
\end{equation}
Finally,
$$
I_{5}=\left[ 2\eps +f(\la a_{0}\ra)\right] \left\{ c \la\xi\ra^{\tdm}, \frac{a_{0}}{\L{x}} \right\}\left(  \phi'_{+}\left(\frac{a_{0}}{\L{x}}\right)
- \phi'_{+}\left(\frac{a_{0}}{\L{x}}\right)\right),
$$
which, using \eqref{C2}, implies
\begin{equation}\label{C10}
I_{5}=\left[ 2\eps +f(\la a_{0}\ra)\right] \left\{ c \la\xi\ra^{\tdm}, \frac{a_{0}}{\L{x}} \right\}  \phi'_{+}\left(\frac{\la a_{0}\ra}{\L{x}}\right).
\end{equation}
Using \eqref{C6} and \eqref{C7} we see that if $\eps$ is small enough,
$$
I_{1}+I_{2}\ge c \frac{\la \xi\ra^\mez}{\L{x}^{1+\delta}} \psi_{0}.
$$
Therefore by \eqref{C9} we have
$$
I_{1}+I_{2}+I_{4}\ge c \frac{\la \xi\ra^\mez}{\L{x}^{1+\delta}}\left( \psi_{0}+\psi_{+}+\psi_{-}\right)\ge 
c \frac{\la \xi\ra^\mez}{\L{x}^{1+\delta}}.
$$
Now by \eqref{C8} and \eqref{C10} we have, 
\begin{equation}\label{C11}
I_{3}+I_{5}= \left[ 2\eps +f(\la a_{0}\ra)-\frac{\la a_{0}\ra}{\L{x}}\right] 
\left\{ c \la\xi\ra^{\tdm}, \frac{a_{0}}{\L{x}} \right\}  \phi'_{+}\left(\frac{a_{0}}{\L{x}}\right).
\end{equation}
The function $\phi_{+}$ being increasing one has $\phi_{+}'\ge 0$. On the support of 
$\phi_{+}'\left(\frac{a_{0}}{\L{x}}\right)$ we have $\eps\le \la a_{0}\ra/\L{x}\le 2/\eps$ so 
$2\eps -\la a_{0}\ra/\L{x}\ge 0$. 
By definition, $f'\ge 0$. Finally, by \eqref{C6} and \eqref{C7} we have 
$\left\{ c \la\xi\ra^{\tdm}, a_{0}/\L{x} \right\}\ge 0$. This ensures that
\begin{equation}\label{C12}
I_{3}+I_{5}\ge 0.
\end{equation}
We conclude, using \eqref{C11} and \eqref{C12} that
$$
\left\{ c \la\xi\ra^{\tdm}, a \right\} \ge c \frac{\la\xi\ra^\mez}{\L{x}^{1+\delta}},
$$
which proves the proposition since $c\ge K_1 (1+\lA \eta\rA_{L^\infty(0,T;H^{s-1})}^2)^{-\tq}>0$.
\end{proof}

We are now in position to prove Proposition~\ref{psmooth}. 

\begin{proof}[Proof of Proposition~\ref{psmooth}]
We begin by remarking that 
we can assume without loss of generality that $\varphi\in C^1(I;L^2(\xR))$ 
(A word of caution: to do so, instead of using the usual Friedrichs mollifiers, we need to use 
the operators $J_\eps$ introduced in~\S\ref{SAE}). 
This allows us to write
\begin{align*}
\frac{d}{dt}\left\langle T_{a}\varphi, \varphi\right\rangle &= 
\left\langle T_{\partial_{t}a}\varphi, \varphi\right\rangle +
\left\langle T_{a}\partial_{t}\varphi, \varphi\right\rangle +
\left\langle T_{a}\varphi, \partial_{t}\varphi\right\rangle \\
&=\left\langle T_{\partial_{t}a}\varphi, \varphi\right\rangle \\
&\quad-\left\langle T_{a} T_{V}\partial_{x}\varphi +T_{a} i T_\gamma \varphi -T_{a}f, \varphi \right\rangle \\
&\quad-\left\langle T_{a}\varphi,  +T_{V}\partial_{x}\varphi + i T_\gamma \varphi -f\right\rangle,
\end{align*}
where $\langle\cdot,\cdot\rangle$ denotes the $L^2$ scalar product. Introduce the commutator
$$
C\defn \left[ i T_\gamma , T_{a}\right].
$$
Since $\partial_t a=0$, the previous identity yields
\begin{equation}\label{af}
\begin{aligned}
\frac{d}{dt} \left\langle T_{a}\varphi, \varphi\right\rangle &= 
\left\langle C \varphi,\varphi \right\rangle +\left\langle  i (T_\gamma^* -T_\gamma) T_{a} \varphi,\varphi \right\rangle\\
&\quad + \left\langle \partial_{x}(T_{V}T_{a} \varphi )-T_{a} T_{V}\partial_{x}\varphi, \varphi \right\rangle \\
&\quad  +\left\langle T_{a}f,\varphi\right\rangle+ \left\langle T_{a}\varphi, f \right\rangle
\end{aligned}
\end{equation}
Since $a\in \dot\Gamma^0_0$, it follows from the usual estimates for paradifferential operators that
$$
\la \left\langle T_{a}\varphi, \varphi\right\rangle \ra \les \lA \varphi\rA_{L^2}^2,\quad 
$$
and
$$
\la  \left\langle T_{a}\varphi, f \right\rangle \ra+\la \left\langle T_{a}f,\varphi\right\rangle\ra \le  
K \lA \varphi\rA_{L^2}^2+K\lA f\rA_{L^2}^2,
$$
for some positive constant $K$. One easily obtains similar bounds for the second and third terms in the right hand-side of \eqref{af}. Indeed, by definition 
of $\gamma$ we know that $T_{\gamma}^*-T_{\gamma}$ is of order $0$. On the other hand, as alredy seen, it follows from 
Theorems~\ref{theo:sc} that
$\partial_{x}(T_{V}T_{a} \cdot )-T_{a} T_{V}\partial_{x}$ is of order $0$. 
Therefore, integrating \eqref{af} in time, we end up with
\begin{equation*}
\int_{0}^{T} \left\langle C \varphi , \varphi \right\rangle \, dt 
\le M \left\{ \lA \varphi(0)\rA _{L^2}^2+
\lA \varphi(T)\rA_{L^2}^2 
+\int_{0}^T \left( \lA \varphi\rA_{L^2}^2+\lA f\rA_{L^2}^2\right)  \, dt\right\},
\end{equation*}
where $M$ depends only on the $L^\infty(0,T;H^{s+\mez}(\xR)\times H^s(\xR))$-norm of $(\eta,\psi)$.

\smallbreak 

Hence to complete the proof it remains only to obtain a lower bound for the left hand-side. To do so, write
$$
i T_\gamma = 
iT_{c}\la D_{x}\ra^\tdm+\frac{3}{4}T_{\frac{\xi}{\la\xi\ra}\partial_{x}c} \la D_{x}\ra^\mez,
$$
and recall that, by definition of $a$ (see Lemma~\ref{DoiL}) there exists a constant $K$ such that
$$
\left\{c (t,x)\la \xi\ra^\tq,a(x,\xi)\right\}\ge K \L{x}^{-1-2\delta}  \la \xi\ra^{\mez},
$$
for some positive constant $K>0$. Since
$$
\left[T_{a}, T_{\frac{\xi}{\la\xi\ra}\partial_{x}c} \la D_{x}\ra^\mez \right] \text{ is of order  }\leo 0,
$$
Proposition~\ref{propd} below then implies that
$$
 \left\langle C \varphi , \varphi \right\rangle \ge  a \lA \L{x}^{-\mez-\delta}\varphi
 \rA_{H^\frac{1}{4}}^2 -A \lA \varphi
 \rA_{L^2}^2,
$$
for some positive constants $a,A$. 
This completes the proof of Proposition~\ref{psmooth} and hence of 
Theorem~\ref{theo:main}.
\end{proof}

\begin{prop}\label{propd}
Let $d\ge 1$ and $\delta>0$. Assume that $d\in \Gamma^{1/2}_{1/2}(\xR^d)$ is such that, for some positive constant $K$, we have
$$
d(x,\xi)\ge K \L{x}^{-1-2\delta}\la\xi\ra^\mez,
$$
for all $(x,\xi)\in \xR^d\times \xR^d\setminus \{0\}$. Then there exist two positive constants $0<a<A$ such that
$$
 \left\langle T_d u , u \right\rangle \ge  a\lA \L{x}^{-\mez-\delta}u
 \rA_{H^\frac{1}{4}}^2 -A \lA u \rA_{L^2}^2.
$$
\end{prop}
\begin{rema}
This proposition has been used for $d=1$. However, it might be useful for $d\ge 1$.
\end{rema}
\begin{proof}
Again, the difficulty comes from our low regularity assumption. 
Indeed, with more regularity (say $d\in \Gamma^{1/2}_{\rho}(\xR^d)$ with $\rho>2$) 
this follows from the sharp G\aa rding inequality proved in \cite{Bony}.

Consider a partition of unity as a sum of squares, such that 
$$ 1= \theta^2_0(x) + \sum_{j=1}^{\infty} \theta^2 (2^{-j}  x)= \sum_{j=0}^{\infty} \theta^2_j (x),
$$
where $\theta_0\in C^\infty_0 (\xR)$ and $\theta\in C^\infty(\xR)$ is supported in the annulus 
$\{x\in \xR\,:\, 1\le \la x\ra\le 3\}$.

Then 
$$
I= \left\langle T_d u, u\right\rangle 
= \sum_{j=0}^\infty \left\langle \theta^2_j T_d  u, u\right\rangle.
$$
The following result is an illustration of the pseudo-local property of paradifferential operators (see~\cite[p435]{Che} for similar results in this direction).
\begin{lemm}\label{L.7.4}
Let $ \widetilde \theta\in C^\infty_0 (]1/2,4[)$ equal to $1$ on the support of $\theta$, and set 
$\widetilde \theta_j (x) = \widetilde \theta( 2^{-j} \la x\ra )$ for $j\ge 1$. Also introduce 
$\widetilde \theta_0\in C^\infty_0 (\xR)$ equal to $1$ on the support of $\theta_0$. 
Then for all $\mu \in \xR$, all $j \in \xN$,  and all $N\in \xN$, the operator $R_j =\theta_j T_d (1- \widetilde \theta_j)$ 
is continuous from $H^\mu$ to $H^{\mu+N}$ with norm bounded by $C_N 2^{-jN}$.
\end{lemm}
\begin{proof} 
Writing (see~\eqref{eq.para})
\begin{align*}
&\theta_j T_d (1- \widetilde \theta_j)u (x)\\
&\quad= \frac1{(2\pi)^2} \int 
e^{i (x\cdot\xi - y\cdot \eta) } \theta_j(x)(1- \widetilde\theta_j(y))\widehat{d} (\xi-\eta,\eta) \psi(\eta) \chi(\xi-\eta, \eta)u(y) d y d\eta d\xi,
\end{align*}
we have
\begin{align*}
&\theta_j T_d (1- \widetilde \theta_j)u (x)\\
&\quad= \frac1{(2\pi)^2} \int 
e^{i (x-y)\cdot\eta}e^{ix\cdot\zeta}\theta_j(x)(1- \widetilde\theta_j(y))\widehat{d} (\zeta,\eta) \psi(\eta) \chi(\zeta, \eta)u(y) d y d\eta d\zeta.
\end{align*}
We then obtain the desired result from a non-stationary phase argument. Indeed, 
using that on the support of this integral we have $|x-y| > c 2^j$, we can integrate by parts using the operator 
$$ L = \frac{(x-y)\cdot \partial_\eta} { \la x-y\ra^2}.$$
Since $\chi(\zeta, \eta)$ is homogeneous of degree $0$ in $(\zeta,\eta)$, 
we obtain that $N$ such integration by parts gain $N$ powers of $2^{-j}$ and of $|\eta|^{-1}$. 
\end{proof}

Now, write
\begin{align*}
&\theta_j T_d  u \\
&\quad=\theta_j T_d  \widetilde\theta_j  u + \theta_j T_d (1- \widetilde\theta_j )u \\
&\quad=T_d \theta_j \widetilde\theta_j+[\theta_j,T_d]\widetilde \theta_j +  \theta_j T_d (1- \widetilde\theta_j )u\\
&\quad=T_{d}  T_{\widetilde\theta_j }  \theta_j +
T_{d}  (\widetilde\theta_j -T_{\widetilde\theta_j})\theta_j +[\theta_j,T_d]\widetilde \theta_j +  \theta_j T_d (1- \widetilde\theta_j )u\\
&\quad=T_{\widetilde\theta_j  d } \theta_j + (T_d T_{\widetilde\theta_j}-T_{\widetilde\theta_j d})\theta_j+
T_{d}  (\widetilde\theta_j  -T_{\widetilde\theta_j})\theta_j +[\theta_j,T_d]\widetilde \theta_j +  \theta_j T_d (1- \widetilde\theta_j )u.
 \end{align*}
The last term in the right hand side is estimated by means of Lemma~\ref{L.7.4}. 
With regards to the second term in the right-hand side, we use \eqref{esti:quant2} to obtain
$$
\sup_{j\in \xN} \lA T_{\widetilde\theta_j } T_d -T_{\widetilde\theta_j  d} \rA_{L^2\rightarrow L^2} \les 
\sup_{j\in \xN} M^0_{1/2}(\widetilde\theta_j )M^{1/2}_{1/2}(d)\les 1.
$$
The third term is estimated by means of the following inequality (see \cite{MePise}) 
$$
\lA \widetilde\theta_j -T_{\widetilde\theta_j } \rA_{L^2\rightarrow H^1}\les \lA \theta_j\rA_{W^{1,\infty}(\xR)}\les 1.
$$

Therefore, we conclude that
\begin{align*}
\left\langle (\theta_j)^2 T_d  u, u\right\rangle 
=\left\langle T_{\widetilde\theta_j  d} \theta_j u, \theta_j u\right\rangle +
\left\langle U_j , \theta_j u\right\rangle 
\end{align*}
for some sequence $(U_j)$ such that
$$
\sum_{j=0}^{\infty}\lA U_j \rA_{L^2}^2\les \sum_{j=0}^\infty 
\left( \big\| \widetilde{\theta}_j u\big\|_{L^2} \lA \theta_j u\rA_{L^2}	
 +2^{-j}\lA u\rA_{L^2}\lA \theta_j u\rA_{L^2}\right)
\les \lA u\rA_{L^2}^2.
$$

We want to prove
$$
\sum_{j=0}^{\infty} \left\langle T_{\widetilde\theta_j  d}  \theta_j u, \theta_j u\right\rangle \ge  a\lA \L{x}^{-\mez-\delta}u
 \rA_{H^\frac{1}{4}}^2 -A \lA u \rA_{L^2}^2.
$$
To do this, it suffices to prove
$$
\left\langle T_{\widetilde\theta_j  d}  \theta_j u, \theta_j u\right\rangle \ge  a 2^{-j(1+2\delta)} \lA  \theta_j u  \rA_{H^\frac{1}{4}}^2 
-A \lA U_j'' \rA_{L^2}^2
$$
for some $U_j''$ such that
$$
\sum_{j=0}^{\infty}\lA U_j '' \rA_{L^2}^2\le A \lA u\rA_{L^2}^2.
$$

Since $(\widetilde{\theta}_j d)^\mez \in \Gamma^{1/2}_{1/2}(\xR^d)$, by applying Theorem~\ref{theo:sc} (with $m=m'=1/2$ and $\rho=1/2$), we have
$$
\left\langle T_{\widetilde\theta_j  d}  \theta_j u, \theta_j u\right\rangle
=\lA T_{ (\widetilde\theta_j d)^{\mez}} \theta_j u\rA_{L^2}^2 +\left\langle R_j  \theta_j u,\theta_j u\right\rangle,
$$
where $R_j$ is uniformly bounded from $L^2$ to $L^2$. 
Now by assumption on $d$, we have
$$
\left( \widetilde{\theta}_j (x) d(x,\xi)\right)^\mez \ge K \widetilde{\theta}_j(x)2^{-j\left(\mez+\delta\right)}\la \xi\ra ^{\frac{1}{4}},
$$
where we used $0\le \widetilde{\theta}_j\le 1$. Therefore the symbol $e_j$ defined by
$$
e_j(x,\xi)\defn \left( \widetilde{\theta}_j (x) d(x,\xi)\right)^\mez+ K 2^{-j \left(\mez+\delta\right)} (1-\widetilde{\theta}_j(x))\la \xi\ra ^{\frac{1}{4}},
$$
satisfies the elliptic boundedness inequality
$$
e_j(x,\xi)\ge K 2^{-j \left(\mez+\delta\right)}\la \xi\ra ^{\frac{1}{4}}.
$$
As a result
$$
2^{-j \left(\mez+\delta\right)} \lA \theta_j u \rA_{H^{\frac{1}{4}}}\le 
K \lA T_{e_j} \theta_j u\rA_{L^2}+ K \lA \theta_j u\rA_{L^2}.
$$
The desired result then follows from the fact that $(1-\widetilde{\theta}_j)\theta_j=0$ which implies that
$$
T_{ (\widetilde\theta_j d)^{\mez}} \theta_j -T_{e_j} \theta_j 
=2^{-j \left(\mez+\delta\right)} T_{ (1-\widetilde{\theta}_j(x))\la \xi\ra ^{\frac{1}{4}}}\theta_j= R'_j \widetilde{\theta}_j,
$$
for some operator $R_j'$ uniformly bounded from $L^2$ to $L^2$.

This completes the proof of Proposition~\ref{propd}.
\end{proof}

\appendix
\section{The case of time dependent bottoms}\label{app}

The purpose of this section is to show that our analysis is still valid in the case 
of a time-dependent bottom. The only difference is indeed the definition of the Dirichlet-Neumann operator. 
In this case, we make the additional Lipschitz regularity assumption on the domain
\begin{itemize}
\item[\bf{$\mathbf{H_{2}}$)}]  
We assume that the domain $\Omega_2$ depends now on the time variable and its boundary is locally the graph of a 
function which is continuous in time with values in Lipschitz functions of $x$, and moreover $C^1$ in time with values in $L^\infty$. 
Namely, for any point $(x_0,y_0,t) \in \Gamma_t=\partial\Omega_t\setminus \Sigma_t$, 
there exists an orthonormal  coordinate system $(x', x_{d+1})$ and a function  $b: [0,T]\times \xR^d \mapsto b(t,x')$ 
which is $C^1$ in time with values $L^\infty$ and $C^0$ in time with values  Lipshitz function with respect to the $x'$ 
variable such that near $(x_0,y_0)$, $\Omega_t$ coincides with the set  
$$
\{ (x',x_{d+1},t)\,:\, x_n > b(t,x')\}.
$$
\end{itemize}
In this setting,  the natural boundary condition at the bottom is to ask  the normal velocity of the fluid 
to be equal to the displacement velocity of the bottom. As a consequence, the water-wave problem reads
\begin{equation}\label{sIbis}
\left\{
\begin{aligned}
&\Delta\phi+\partial_{y}^2\phi=0 &&\text{in \,}\Omega_t,\\
&\partial_{t} \eta = \partial_{y}\phi -\partialx\eta\cdot\partialx \phi &&\text{on } \Sigma_t, \\
&\partial_{t}\phi=-g \eta + \kappa H(\eta)-\frac{1}{2}\la \partialx\phi\ra^2 -\frac{1}{2}\la \partial_{y}\phi\ra^2
&&\text{on } \Sigma_t ,\\
&\partial_{n}\phi(m) = \frac {dm} {dt}  \cdot  n(m)  &&\text{for } m\in \Gamma_t,
\end{aligned}
\right.
\end{equation}
where here  $ \frac {dm} {dt} $ is the time 
derivative of the point $m$ on the boundary $\Gamma_t$. 
Notice that clearly, this quantity is dependent on the choice of coordinates defining the domain, 
but $ \frac{dm} {dt}  \cdot  n(m)$ is not. In the coordinate system above, 
the point on the boundary is $m(x',t)= (x',b(t,x'))$, 
$$
n(m) = \frac{ (\nabla_{x'} b ,1)} { \sqrt{ |\nabla_{x'}b|^2 + 1}},\qquad  \frac{dm } {dt} = ( x', \partial_t b),
$$
and the boundary condition reads
$$
(\nabla_{x'} b \cdot \nabla_{x'} \phi)(t,x', b(t,x'))+   \partial_{x_{d+1}} \phi(t,x', b(t,x')) = \partial_t b.
$$
Consequently, to define the Dirichlet-Neumann operator, the crucial step is to solve the system
\begin{equation}\label{eq.harmotime}
\Deltayx \phi= 0 \text{ in } \Omega_t, 
\qquad \phi \mid_{\Sigma_t} = \psi, \qquad \frac{ \partial \phi} {\partial n} \mid_{\Gamma_t} =  \frac {dm} {dt}  \cdot  n(m)= k(t,m).
\end{equation}
The Poincar\'e inequality obtained in Section~\ref{sec.2.1} can be precised. We shall show that one can so chose 
the weight $g=g(m)$ in Corollary~{corog} so that $g$ 
does not blow up as long as the point $m$ remains in a bounded set. `

Recall (cf Notation~\ref{notaD}) that $\mathscr{D}_0$ is 
the space of functions $u\in C^{\infty}(\Omega)$ such that 
$\nabla_{x,y}u\in L^2(\Omega)$, and $u$ equals to $0$ near the top boundary $\Sigma$.
\begin{lemm} 
For any point $m_0 \in \Gamma$ there exists $m_1\in \Omega$, $C>0$ and $\delta_0>0$ 
such that for any $0<\delta < \delta_0$, and any $u\in \mathscr{D}_0$,
$$
\int_{B(m_0, \delta) \cap \Omega}  |u|^2 \,dxdy\leq C\int_{B(m_1, \delta) \cap \Omega}  |u|^2 \,dxdy+C\int_{\Omega} |\nabla u |^2 \,dxdy.
$$
\end{lemm}
Indeed, using assumption $H2)$ and performing a Lipschitz change of variables near $m_0$, we are reduced to the case where the domain is 
$$ \Omega= \{(x_n , x'); x_n >0\}$$ and the point $m_0= (0,0)$. Choosing $m_1= (\epsilon, 0)$, Lemma~\ref{lem.accr} follows now from the same proof as for Lemma~\ref{lem.accr}. 

We now deduce easily using Lemma~\ref{lem.Poinc}, 
\begin{lemm}\label{lem.Poincbis}
Assume that the domain $\Omega$ satisfies the assumptions above. For any $m_0=(x_0,y_0)\in \overline{\Omega}$ 
there exists a 
neighboorhood $\omega$ of $m_0$ in $\xR^{d+1}$ and $C>0$ such that for any function $u \in \mathscr{D}_0$, we have 
$$
\int_{\omega\cap \Omega} |u|^2 \,dx dy\leq C \int_{\Omega} |\nabla_{x,y} u |^2 \,dxdy.
$$
\end{lemm}
\begin{coro}There exists a weight 
$ g \in L^\infty_{loc} ( \overline{\Omega})$, positive everywhere, equal to $1$ near the top 
boundary $\Sigma$ of $\Omega$, 
and such that  for any function $u \in \mathscr{D}_0$ equal to $0$ near $\Sigma$, we have 
$$
\int_\Omega g(x,y) |u|^2 \,dxdy \leq C \int_{\Omega} |\nabla_{x,y} u |^2 \,dxdy.
$$
\end{coro}
As a consequence of this result and usual trace theorems, 
\begin{coro}\label{trace}
There exists a weight $g$ in  $L^\infty_{\text{loc}} (\overline \Omega)$ equal to $1$ such that the map 
$$
u \in \mathscr{D}_0 \mapsto u\arrowvert_{\Gamma} \in L^2( \Gamma, g d\sigma)
$$
extends uniquely to a continuous map
$$
u \in H^{1,0}(\Omega) \mapsto u\arrowvert_{\Gamma} \in L^2( \Gamma, g d\sigma).
$$
\end{coro} 
We are now in position to define the Dirichlet-Neumann operator.
 Let $\psi (x) \in H^1( \xR^d)$. For $\chi \in C^\infty_0 (-1,1)$ equal to $1$ near $0$, we first define 
$$\widetilde \psi= \chi \left(\frac {y-\eta(x)} h\right) \psi(x) \in H^{1} ( \xR^{d+1})$$
Then let $\widetilde \phi$ be the unique variational solution of the system 
$$
- \Delta_{x,y} \widetilde \phi = \Delta_{x,y} \widetilde \psi, \qquad \widetilde \phi\arrowvert_{\Sigma} =0, 
\qquad \partial _n \widetilde \phi\arrowvert_{\Gamma} = k
$$
which is the unique function $\widetilde \phi \in H^1_0 ( \Omega)$ of
\begin{equation}\label{eq.variater}
 \forall v \in H^{1,0} ( \Omega), \qquad 
 \int_\Omega \nabla_{x,y} \widetilde \phi \cdot \nabla_{x,y}  v
= \int_{\Omega} v \Delta_{x,y} \widetilde \psi  - \int_{\Gamma} k_t v \arrowvert_{\Gamma}d\sigma.
 \end{equation}
Here notice that the first term in the right hand side of~\eqref{eq.variater} is, as in the 
time independent case, a bounded 
linear form on $H^{1,0} ( \Omega)$. Now, we make the additional assumption 
(which is always satisfied if the domain is time dependent only on a bounded zone).
\begin{itemize}
\item[\bf{$\mathbf{H_{3}}$)}]  
Assume that the time dependence of the domain (i.e. the function $k$) decays sufficiently fast 
near infinity, so that 
$$
\frac {dm} {dt}\cdot n(m) g(m)^{1/2}=k(m,t) g(m)^{1/2} \in L^\infty(0,T;L^2 ( \Gamma)).
$$
\end{itemize}
Then, according to Corollary~\ref{trace}  
the second term of the r.h.s. of~\eqref{eq.variater} 
is a bounded linear form on $H^{1,0}( \Omega)$ (uniformly with respect to time), and consequently~\eqref{eq.variater} has a unique variational solution.

 We now define $\phi = \widetilde \phi + \widetilde \psi$ and 
\begin{align*}
G(\eta, k) \psi  (x)&=
\sqrt{1+|\partialx\eta|^2}\,
\partial _n \phi\arrowvert_{y=\eta(x)},\\
&=(\partial_y \phi)(x,\eta(x))-\partialx \eta (x)\cdot (\partialx \phi)(x,\eta(x)).
\end{align*}
Notice that as in the previous section, a simple calculation shows that this definition is independent on the choice of the lifting function $\widetilde\psi$ as long as it remains bounded in $H^1$ and localized in the strip $-h <y \leq 0$.

Now, the proof of 
Theorem~\ref{th.4} is exactly the same as the proof of Theorems~\ref{theo:Cauchy} and~\ref{theo:main}, using this new definition of the Dirichlet-Neumann operator.

\providecommand{\bysame}{\leavevmode\hbox to3em{\hrulefill}\thinspace}
\providecommand{\href}[2]{#2}

\end{document}